\documentclass[a4paper,12pt,times,numbered,index]{PhDThesisPSnPDF}
\usepackage{hyperref}
\usepackage{bookmark}
\usepackage{array}
\usepackage{amsthm}
\usepackage{latexsym}
\usepackage{enumerate}
\usepackage[english]{babel}
\usepackage{graphicx}
\usepackage{tabularx}
\usepackage{rotating}
\newtheorem{thm}{Theorem}[section]
\newtheorem{prop}[thm]{Proposition}
\newtheorem{defi}[thm]{Definition}
\newtheorem{lem}[thm]{Lemma}
\newtheorem{cor}[thm]{Corollary}

\newtheorem{conj}[thm]{Conjecture}

\newtheorem{example}[thm]{Example}
\newtheorem{rem}[thm]{Remark}
\newtheorem{crt}[thm]{Criterion}

\def\Cat{{\rm{\overline{Cay}(Sym_n}},T_n)}
\def\Cayy{{\rm{Cay(Sym_n}},S)}
\def\Cay{{\rm{Cay(Sym_n}},S_n)}
\def\Aut{\mbox{\rm Aut}}
\def\Sym{\mbox{\rm Sym}}

\newcommand{\cupdot}{\mathbin{\mathaccent\cdot\cup}}

\input{preamble}

\title{\texorpdfstring{Combinatorial Properties of Block \\ Transpositions in Symmetric Groups}{LaTeX2e}}

\author{Annachiara Korchmaros}

\dept{Department of Mathematics, Computer Science, Business}

\university{Universit\`{a} degli Studi della Basilicata}
\crest{\includegraphics[width=0.25\textwidth]{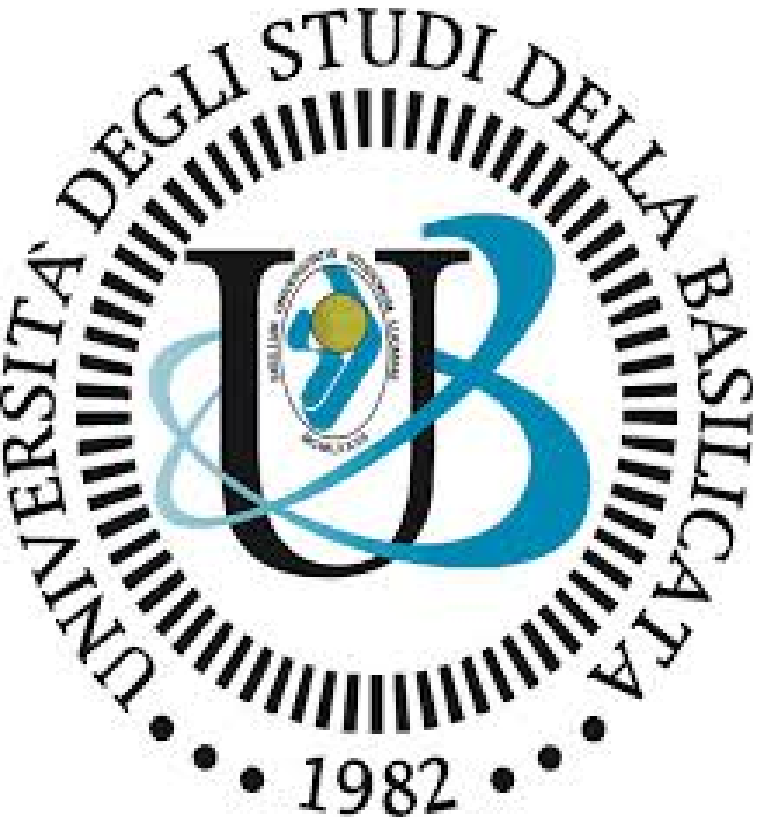}}

\degree{Doctor of Philosophy in Mathematics}

\college{Tutor: Professor Giorgio Faina}

\degreedate{March 2015}

\subject{LaTeX} \keywords{{LaTeX} {PhD Thesis} {Mathematics} {Universit\`{a} degli Studi della
Basilicata}}

\ifdefineAbstract
\fi

\ifdefineChapter
\fi

\begin{document}

\frontmatter

\begin{titlepage}
  \maketitle
\end{titlepage}

\include{dedication}

\begin{declaration}

I hereby declare that except where specific reference is made to the work of 
others, the contents of this dissertation are original and have not been 
submitted in whole or in part for consideration for any other degree or 
qualification in this, or any other university. This dissertation is my own 
work and contains nothing which is the outcome of work done in collaboration 
with others, except as specified in the text and Acknowledgements. This 
dissertation contains fewer than 65,000 words including appendices, 
bibliography, footnotes, tables and equations and has fewer than 150 figures.


\end{declaration}

\begin{acknowledgements}
Completion of this doctoral thesis was possible with the support of many people, in many countries. I would like to express my sincere gratitude to all of them.

I am particularly indebted to my adviser Giorgio Faina for allowing me to grow as a mathematician. Giorgio provided me with every bit of guidance and expertise that I needed during my first year; then, when I felt ready to venture into research on my own and branch out into new research fields, Giorgio gave me the freedom to do whatever I wanted while at the same time continuing to contribute advice and encouragement. Similar, profound gratitude goes to Mikl\'{o}s B\'{o}na for inviting me to the Department of Mathematics of the University of Florida in the fall and spring of $2013$. He has been actively interested in my work and has always been available to advise me. I am very grateful for his patience and immense knowledge in combinatorics that made him a great mentor. I am deeply thankful to Marien Abreu and Domenico Labbate for introducing me to the study of graph theory. Marien and Domenico have been a good source of inspiration and novel ideas. I also thank Angelo Sonnino for his help in compiling correctly the references of this thesis. I gratefully acknowledge the University of Basilicata, not only for providing the funding which allowed me to undertake my Ph.D. program, but also for giving me the opportunity to attend several international conferences and meet many eminent personalities in discrete mathematics.

I would like to thank many student colleagues for providing a stimulating and fun environment in which to learn and grow at University of Perugia, J\'{a}nos Bolyai Mathematical Institute, University of Florida, and University of Basilicata. I am especially grateful to Daniele Bartoli who started with me the studying of the fascinating world of block transpositions. My time in Gainesville and Szeged was made enjoyable in large part due to all friends that became part of my life. A special thank to Silvana Maiorano for helping me get through the difficult times, and for all the emotional support she provided. My gratitude goes to my loving and encouraging friend Matteo Zancanella whose faithful support during the final stages of this Ph.D. was immensely appreciated.

I would like to thank my large family for providing a loving environment for me. My grandmother, my sister, uncles, ants, and cousins were particularly supportive. Lastly, and most importantly, I wish to thank my parents. They raised me, supported me, taught me, and loved me, to them I dedicate this thesis.

\end{acknowledgements}

\begin{abstract}

A major problem in the study of combinatorial aspects of permutation group theory is to determine the distances in the symmetric group $\Sym_n$ with respect to a generator set. The difficulty in developing an adequate theory, as well as the hardness of the computational complexity, may dramatically vary depending on the particular features of the generator set. Tricky cases often occur, especially when
 the choice of the generator set is made by practical need. One well-known such a case is when the generator set $S_n$ consists of block transpositions which are special permutations defined in computational biology. It should be noted that ``the block transposition distance of a permutation'' is the distance of the permutation from the identity permutation in the Cayley graph $\Cay$, and ``sorting a permutation by block transpositions'' is equivalent to finding shortest paths in $\Cay$. Also, the automorphism group $\Aut(\Cay)$ of the associated Cayley graph $\Cay$ is the automorphism group of the metric space arising from the block transposition distance.

The original results in our thesis concern two issues, namely the lower and upper bounds on the block transpositions diameter of $\Sym_n$ with respect to $S_n$ and the automorphism group $\Aut(\Cay)$. A significant contribution is to show how from the toric equivalence can be obtained very useful bijective maps on $\Sym_n$ that we call \emph{toric maps}. Using the properties of the toric maps, we give a proof for the invariance principle of the block transposition distance within toric classes and discuss its role in the proof of the Eriksson bound.
Furthermore, we prove that $\Aut(\Cay)$ is the product of the right translation group by $\textsf{N}\rtimes\textsf{D}_{n+1}$, where $\textsf{N}$ is the subgroup fixing $S_n$ elementwise and $\textsf{D}_{n+1}$ is a dihedral group of order $2(n+1)$ whose maximal cyclic subgroup is generated by the toric maps.
Computer aided computation carried out for $n\leq 8$ supports our conjecture that $\textsf{N}$ is trivial. Also, we prove that the subgraph $\Gamma$ with vertex set $S_n$ is a $2(n-2)$-regular graph whose automorphism group is $\textsf{D}_{n+1}$. We show a number of interesting combinatorial aspects of $\Cay$, notably $\Gamma$ has as many as $n+1$ maximal cliques of size $2$, its subgraph $\Gamma(V)$ whose vertices are those in these cliques is a $3$-regular Hamiltonian graph, and $\textsf{D}_{n+1}$ acts faithfully on $V$ as a vertex regular automorphism group.

\end{abstract}


\tableofcontents

\listoffigures

\listoftables



\mainmatter

\chapter{Introduction}\label{intro}

The general problem of determining the distances in $\Sym_n$ with respect to a generator set has been intensively investigated in
combinatorial group theory and enumerative combinatorics. Its study has also been motivated and stimulated by practical applications, especially in computational biology, where the choice of the generator set depends on practical need that may not have straightforward connection with pure mathematics; see Chapter \ref{biol}.

In our thesis we mostly deal with ``\emph{sorting a permutation by block transpositions}''. This problem asks for the block transposition distance of a permutation with respect to the generator set of certain permutations called block transpositions. The idea of a block transposition arises from computational biology as an operator which acts on a string by removing a block of consecutive entries and inserting it somewhere else. The formal definition in terms of permutations is given in Chapter \ref{c3}.

Our original contributions concern two issues, namely the lower and upper bounds on the diameter of $\Sym_n$ with respect to the set $S_n$ of block transpositions and the group of automorphisms of the metric space arising from the block transposition distance.

Historically, the problem of sorting by block transpositions was introduced by Bafna and Pevzner in their seminal paper dating back 1998; see \cite{BP}.
The distribution of block transposition distances is currently known for $n\le 14$. It was computed by Eriksson et al. for $n\leq 10$; see \cite{EE}, by G\~{a}lvao and Diaz for $n=11,12,13$; see \cite{GD}, and by Gon\c{c}alves et al. for $n=14$; see \cite{GBH}. Bafna and Pevzner \cite{BP} also provided a polynomial-time $\textstyle{\frac{3}{2}}$-approximation algorithm to compute the distances.
Ever since, numerous investigations aim at designing polynomial-time approximation algorithms; see \cite{BE,FZ,Gu}, the best-known fixed-ratio algorithm is the $\textstyle{\frac{11}{8}}$-approximation due to Elias and Hartman \cite{EH}. Bulteau, Fertin, and Rusu \cite{BF} addressed the issue of determining the complexity class of the sorting by block transpositions. They were able to give a polynomial-time reduction from SAT which proves the $\textsf{NP}$-hardness of this problem. Therefore, it is challenging to determine the diameter $d(n)$, that is, the maximum of the block transposition distances.

In this direction, several papers have pursued lower and upper bounds on $d(n)$.
{}From \cite{EE}, $d(n)$ is known for $n\leq 15$:
\[\arraycolsep=1.4pt\def\arraystretch{2.2}
d(n)=\left\{
\begin{array}{cl}
\left\lfloor\dfrac{n+2}{2}\right\rfloor, &  3\leq n\leq 12\quad n=14,\\
\dfrac{n+3}{2}, & n=13,15.
\end{array}
\right.
\]In Section \ref{17}, we show that $d(17)=10$. For $n>15$, $$\left\lceil\dfrac{n+2}{2}\right\rceil\leq d(n)\leq \left\lfloor\frac{2n-2}{3}\right\rfloor.$$ The lower bound here is the Elias and Hartman \cite{EH}. That improves the previous lower bound due to Bafna and Pevzner \cite{BP}, and independently to Labarre \cite{la}. In Section \ref{labarre}, we give a survey of the approach used in \cite{la}. The upper bound on $d(n)$ available in the literature is named the Eriksson bound and stated in 2001; see \cite{EE}. However, the proof of the Eriksson bound given in \cite{EE} is incomplete, since it implicitly relies on the invariance of the block transposition distance within toric classes. For the definition of the toric equivalence in $\Sym_n$, see Chapter \ref{c2}. It should be noticed that this invariance principle has been claimed explicitly in a paper appeared in a widespread journal only recently; see \cite{CK}. Although Hausen had already mentioned it and sketched a proof in his unpublished Ph.D. thesis \cite{Ha}, Elias and Hartman were not aware of Hausen's work and quoted the Eriksson bound in a weaker form which is independent of the invariance principle; see \cite{EH}.

Our original contribution in this direction is to show how from the toric equivalence can be obtained bijective maps on $\Sym_n$ that leave the block transposition distances invariant; see Chapter \ref{c5}. Using the properties of these maps, we give an alternative proof for the above invariance principle. We also revisit the proof of the key lemma in \cite{EE} giving more technical details and filling some gaps.

A metric space on $\Sym_n$ arises from the block transposition distance in the usual way. Furthermore, the Cayley graph $\Cay$ on $\Sym_n$ with the generator set $S_n$ is a very useful tool in the study of this metric space. In fact, ``sorting by block transpositions'' is equivalent to finding the shortest paths in $\Cay$, as it was pointed out in \cite{EE,FL,Mo} without in-depth analysis. Our contribution is a study of $\Cay$, its combinatorial properties, and automorphism group in the spirit of the papers in the vast literature on this subject; see \cite{AB,ES,LL,Mo}. Our results in Chapter \ref{c6} show that $\Cay$ presents several interesting features. The subgraph $\Gamma$ with vertex set $S_n$, named \emph{block transposition graph}, has especially nice properties. As we prove in Section \ref{s61}, $\Gamma$ is a $2(n-2)$-regular graph whose automorphism group is a dihedral group $\textsf{D}_{n+1}$ of order $2(n+1)$ arising from the toric equivalence in $\Sym_n$ and the reverse permutation. Furthermore, we show that $\Gamma$ has as many as $n+1$ maximal cliques of size $2$ and look inside the subgraph $\Gamma(V)$ whose vertices are the $2(n+1)$ vertices of these cliques. We prove that $\Gamma(V)$ is $3$-regular. We also prove that $\Gamma(V)$ is Hamiltonian and $\textsf{D}_{n+1}$ is an automorphism group of $\Gamma(V)$ acting transitively (and hence regularly) on $V$. This confirms the famous Lov\'asz conjecture for $\Gamma(V)$. Regarding the automorphism group $\Aut(\Cay)$ of $\Cay$, our original contribution is given in Section \ref{s62}, where we prove that $\Aut(\Cay)$ is the product of the right translation group $R(\Cay)$ by $\textsf{N}\rtimes \textsf{D}_{n+1}$, where $\textsf{N}$ is the subgroup fixing every block transposition. Computer aided computation carried out for $n\leq 8$ and performed by using the package ``grape'' of GAP \cite{gap} supports our conjecture that $\textsf{N}$ is trivial, equivalently $\Aut(\Cay)=R(\Cay) \textsf{D}_{n+1}$. We also prove that $R(\Cay) \textsf{D}_{n+1}$ is isomorphic to the direct product of $\Sym_{n+1}$ by a group of order $2$.

\chapter{Connection with biology}\label{biol}

Our thesis is concerned with the study of problems motivated by biology. In this chapter, we give an informal overview of the concepts and problems
that we discuss in our thesis.

The blueprint of every organism is contained in the genome. The genome is specific for each species and changes only slowly over time. This is how
the species evolves, and the process is called \emph{evolution}. For the seek of simplicity, we say \emph{genes} for homologous markers of DNA (segments extractable from species which support the hypothesis that they belonged to the common ancestor of these species), \emph{chromosome} for the set of genes, and \emph{genome} for the set of all chromosomes of a given species; see Figure \ref{fig:DNAA}.
\begin{figure}[hb!]
\centering
\includegraphics[width=\textwidth]{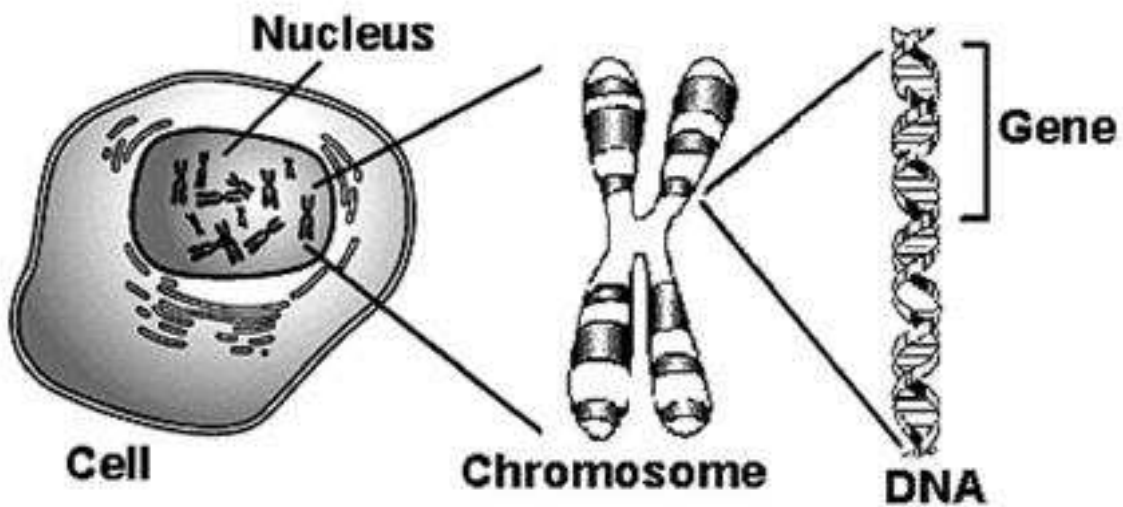}
\caption[Chromosome, DNA, and gene]{The relationship of a cell, chromosome, DNA, and gene (reprinted from \cite{BN}).}
\label{fig:DNAA}
\end{figure}
The possibility of extract a large amount of information from the DNA has given rise to methods for genome comparison. By comparing the genome of two species, we can estimate the time since these two species diverged. This number is the \emph{evolutionary distance} between the two species since it is relevant to the study of the evolution of the species.

It is known that DNA segments evolve by small and large mutations. Small or \emph{point mutations} involve nucleotides while structural variations of a DNA segment depend on large scale mutations called \emph{genome rearrangements}. Analysis of genomes in molecular biology began in the late $1930$s by Dobzhansky and Sturtevant \cite{DS} and continued by Palmer et al. in the late $1980$s \cite{P} demonstrated that different species may have essentially the same genes, but the gene order may differ between species. The rearrangements we consider in our thesis are the reversal of a substring in a chromosome, namely \emph{reversal} or inversion; see Figure \ref{fig:reversals}, and the deletion and subsequent reinsertion of a substring far from its original site, namely \emph{transpositions} or intrachromosomal translocations; see Figure \ref{fig:transpositions}.
\begin{figure}[hbt!]
\centering
\includegraphics[width=\textwidth]{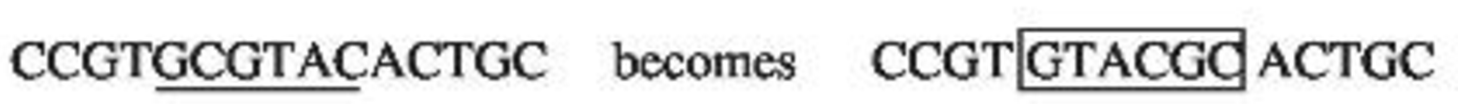}
\caption[Reversal in a chromosome]{Reversal of the underlined segment, resulting in the boxed segment (reprinted from \cite{FL}).}
\label{fig:reversals}

\includegraphics[width=\textwidth]{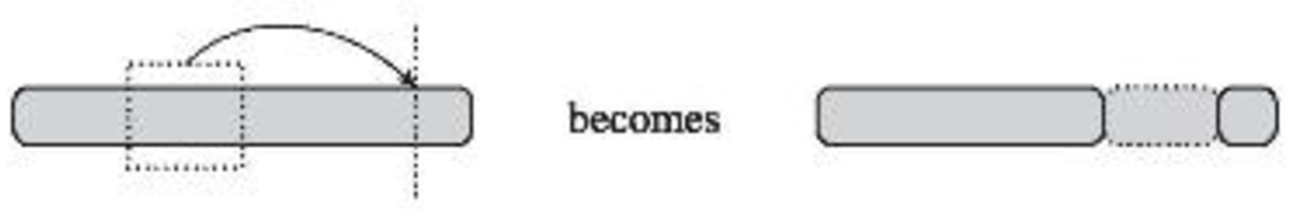}
\caption[Transposition in a chromosome]{Transposition of the dotted region in a chromosome (reprinted from \cite{FL}).}
\label{fig:transpositions}

\end{figure}

For example, the only major difference between the gene orders of two of the most well-known bacteria, Escherichia coli and Salmonella typhimurium, is an inversion of a long substring of the chromosomal sequence; see \cite{O}. For plants, in Palmer et al. \cite{P} compare the mitochondrial genomes of Brassica oleracea (cabbage) and Brassica campestris (turnip) and discover that only five inversions need to ``transform'' a cabbage into a turnip. In \cite{BP}, Bafna and Pevzner stress that researches on genomes of Epstein-Barr virus and Herpes simplex virus revealed that evolution of herpes viruses involved a number of inversions and transpositions of large fragments. In particular, a gene common in herpes virus precursor ``jumped'' from one location in the genome to another; see \cite{BP}. Also, Bafna and Pevzner assert that such examples convincingly prove that using genome rearrangements is a common mode of molecular evolution in mitochondrial, chloroplast, viral, and bacterial DNA. Therefore, a method to determine the distance between genomes of two different species applies a series of rearrangements to the blocks of genes of the first genome until the second one is obtained. The \emph{rearrangement distance} is the minimum number of mutations needed to transform a genome into another, and the \emph{genome rearrangement problem} consists in finding such distances for a specific set of rearrangements. The reason why the minimum number of mutations is studying comes from the \emph{parsimony hypothesis}: the most parsimonious scenario requires the least amount of changes since rearrangements are rarer events than point mutations.

It was not until $1982$ since combinatorialists started to formalize and be involved in the rearrangement problems, but in the last decade, a large body of work was devoted to these problems. The main reason why mathematicians became interested in this topic is the equivalence between the well-known sorting permutation problem and the rearrangement problem modeled by permutations; see Proposition \ref{sorting}. In their book \cite{FL}, Fertin et al. assume that genomes consist of a single chromosome, the order of genes in each genomes is known, and genomes share the same set and number of genes with a single copy of each gene, we can represent genomes and rearrangements by permutations on $\{1,2,\ldots,n\},$ where the labels are genes if permutations are genomes. Furthermore, working out an evolutionary scenario between two species requires to solve the problem of transforming a permutation to another by a minimum number of rearrangements. The interested reader is referred to \cite{AA,BP,FL,HP,KS}. 
\chapter{Background}\label{c2}

In this chapter, we fix notation and terminology concerning symmetric groups and graph theory. We limit ourselves to basic concepts, others will be introduced as they enter in play.
\section{Symmetric groups}\label{s21}
Throughout our thesis, $n$ denotes a positive integer. In our investigation, the cases $n\leq 3$ are trivial.
For a set $X$ of size $n$, $\Sym_n$ stands for the group of all permutations on $X$. For the seek of simplicity, $[n]=\{1,2,\ldots,n\}$ is usually taken for $X$. We mostly adopt the functional notation for permutations. Accordingly, a permutation
$$\pi=\left(
\begin{array}{cccc}
 1 & 2 & \cdots & n \\
\pi_1 & \pi_2 & \cdots & \pi_n \\
\end{array}
\right)$$on $[n]$ is denoted by $\pi=[\pi_1\,\pi_2\cdots \pi_n]$ with $\pi(t)=\pi_t$ for every $t\in[n]$. In particular, the \emph{reverse permutation} is $w=[n\,n-1\cdots 1]$ and $\iota=[1\,2\cdots n]$ is the \emph{identity permutation}.
For any $\pi,\nu\in\Sym_n$, $\pi\circ\nu$ is carried out by $\pi(\nu(t))$ for every $t\in [n]$.

For $k > 1$, a \emph{$k$-cycle} is a permutation $\pi = (i_1, \ldots, i_k)$ such that
$$
\pi_t =\left\{\begin{array}{ll}
t, &  t\notin\{i_1,\ldots,i_k\},\\
i_{j+1}, & t=i_j\quad\,1\le j\le k-1,\\
i_1,& t=i_k.
\end{array}
\right.
$$
It is well-known that every permutation can be written as a product of finitely many $2$-cycles. A permutation is \emph{even} if it can be written as a product of an even number of $2$-cycles. The \emph{alternating group} ${\rm{Alt}}_{n}$ is the set of all even permutations on $[n]$.

The \emph{conjugate of a permutation} $\pi$ by a permutation $\nu$ with $\pi,\nu\in\Sym_n$ is the permutation
$$\pi^{\nu}=\nu\circ\pi\circ\nu^{-1},$$ where $\nu\circ\pi$ is carried out by $\nu(\pi(t))$ for every $t\in [n]$.

\subsection{Rearrangement distances on symmetric groups}\label{rear dist}
For any inverse closed generator set $S$ of a finite group $G$ that does not contain the identity of $G$, a standard method provides a metric space whose points are the element $g\in G$. In our thesis, $G=\Sym_n$, and the choice of $S$ is motivated by applications to the rearrangement problem; see Chapter \ref{biol}. Accordingly, we name $S$
the \emph{rearrangement set} and use the term rearrangement distance defined as follows.

Let $\pi,\nu\in \Sym_n$ and let $\sigma_1,\cdots,\sigma_k\in S$ such that
\begin{equation}\label{dis}
\nu=\pi\circ\sigma_1\circ\cdots\circ\sigma_k.
\end{equation}The minimum number $d(\pi,\nu)$ of rearrangements occurring in (\ref{dis}) is the \emph{rearrangement distance of $\pi$ and $\nu$}.
Then, let us consider the map $$d\colon{\rm{Sym}_n}\times{\rm{Sym}_n}\to \mathbb{N}_0,$$ where $\mathbb{N}_0$ stands for the set of non-negative integers. Now, we show some properties of $d$.
\begin{lem}\label{matnov5}
The rearrangement distance is a distance on $\Sym_n$.
\end{lem}
\begin{proof}
For any $\pi,\nu,\mu\in Sym_n$ we have to show that the following axioms are satisfied.
\begin{itemize}
\item[\rm(I)] $d(\pi,\nu)\geq 0$ with equality if and only if $\pi=\nu$;
\item[\rm(II)] $d(\pi,\nu)=d(\nu,\pi)$;
\item[\rm(III)] $d(\pi,\nu)+d(\nu,\mu)\geq d(\pi,\mu)$ (\emph{triangular inequality}).
\end{itemize}
(I) Since $S$ is a generator set of $\Sym_n$, (\ref{dis}) holds for any $\pi,\nu\in Sym_n$. In fact, there exist $\tau_1,\cdots,\tau_l,\xi_1,\cdots,\xi_m\in S$ such that $$\nu=\tau_1\circ\cdots\circ\tau_l;\quad \pi^{-1}=\xi_1\circ\cdots\circ\xi_m.$$Thus
$$\nu=\pi\circ\xi_1\circ\cdots\circ\xi_m\circ\tau_1\circ\cdots\circ\tau_l=\pi\circ\sigma_1\circ\cdots\circ\sigma_k,$$ whenever $k=l+m$. Therefore, the first statement holds true.

(II) Since $S$ is inverse closed, (\ref{dis}) yields
$$\pi=\nu\circ\sigma_k^{-1}\circ\cdots\circ\sigma_1^{-1}.$$From this, the second statement holds.

(III) Assume $\nu=\pi\circ\sigma_1\circ\cdots\circ\sigma_{d(\pi,\nu)}$ and $\mu=\nu\circ\tau_1\circ\cdots\circ\tau_{d(\nu,\mu)}$ with $\sigma_1,\cdots,\sigma_{d(\pi,\nu)},\\\tau_1,\cdots,\tau_{d(\nu,\mu)}\in S$. Then
$$\mu=\pi\circ\sigma_1\circ\cdots\circ\sigma_{d(\pi,\nu)}\circ\tau_1\circ\cdots\circ\tau_{d(\nu,\mu)}$$whence the third statement follows. This concludes the proof.
\end{proof}
A distance $\delta$ on $\Sym_n$ is \emph{left-invariant} if for $\mu,\pi,\nu\in\Sym_n$,
$$\delta(\pi,\nu)=\delta(\mu\circ\pi,\mu\circ \nu).$$
\begin{prop}\label{distance}
The rearrangement distance is left-invariant.
\end{prop}
\begin{proof}
For any $\mu,\pi,\nu\in\Sym_n$, multiplying by $\mu$ both sides in (\ref{dis}) gives
$$\mu\circ\nu=\mu\circ\pi\circ\sigma_1\circ\cdots\circ\sigma_{d(\pi,\nu)}$$ whence $d(\pi,\nu)\geq d(\mu\circ\pi,\mu\circ \nu)$.

On the other hand, (\ref{dis}) applied to $\mu\circ\pi$ and $\mu\circ\nu$ shows
$$\mu\circ\nu=\mu\circ\pi\circ\sigma_1\circ\cdots\circ\sigma_{d(\pi,\nu)}$$ whence $d(\pi,\nu)\leq d(\mu\circ\pi,\mu\circ \nu)$.
\end{proof}
Since $S$ is a generator set, the following definition is meaningful.
\begin{defi}
{\em{The \emph{rearrangement distance of a permutation} $\pi$ is $d(\pi)$ if $\pi$ is the product of $d(\pi)$ rearrangements, but it cannot be obtained as the product of less than $d(\pi)$ rearrangements. The maximum of the rearrangement distances of permutations on $[n]$ is the \emph{rearrangement diameter} $d(n)$ of the symmetric group.}}
\end{defi}

\begin{rem}\label{link}
{\em{Note that $d(\pi)=d(\iota,\pi)=d(\pi,\iota)$ by Lemma \ref{matnov5}.}}
\end{rem}
Obviously, Lemma \ref{matnov5} and Proposition \ref{distance} hold true for the rearrangement distance of a permutation as well.

\begin{lem}\label{prop d}
Any permutation and its inverse have the same rearrangement distance.
\end{lem}
\begin{proof}
By Lemma \ref{matnov5} and Remark \ref{link},
$$d(\pi)=d(\pi,\iota)=d(\iota,\pi^{-1})=d(\pi^{-1})$$hence the statement holds.
\end{proof}

In his book \cite{B} B\'{o}na uses the general term of ``sorting a permutation'' on $[n]$ as the task of arranging $1,\ldots,n$ in increasing order efficiently. Referring to the rearrangement problem we consider in our thesis, ``arranging in increasing order'' means that starting of with a permutation $\pi$, each step is carried out multiplying the permutation obtained at the previous step and a rearrangement; while ``efficiently'' means with the minimum number of steps. Formally, \emph{sorting $\pi$ by rearrangements} consists in finding the minimum number of rearrangements $\sigma_1,\ldots,\sigma_k$ such that
$$\iota=\pi\circ\sigma_1\circ\cdots\circ\sigma_k$$
\begin{prop}
\label{sorting}
Computing the rearrangement distance between two permutations is equivalent to sorting a permutation by the same set of rearrangements.
\end{prop}
\begin{proof}
By Proposition \ref{distance}, for any $\mu,\nu$ permutations on $[n]$, $d(\pi,\nu)=d(\nu^{-1}\circ\pi,\iota)$\\ whence
$$\{d(\pi,\nu)|\,\pi,\nu\in \Sym_n\}=\{d(\mu,\iota)|\,\mu\in \Sym_n\}.$$The statement follows from Remark \ref{link}.
\end{proof}

\section{Graph theory}\label{gt}
In our thesis $\Gamma=\Gamma(V)$ is a finite simple undirected graph with vertex set $V$. For any two distinct vertices $u,v\in V$, if the pair $\{u,v\}$ is an edge of $\Gamma$, then $u$ and $v$ are the \emph{endpoints} of $e$. Also, $u$ and $v$ are \emph{incident} with $e$, and vice versa. Two vertices which are incident with a common edge are \emph{adjacent}, as are two edges which are incident with a common vertex.

The \emph{degree} of a vertex $v\in V$ in $\Gamma$ is the number of edges of $\Gamma$ incident with $v$. The degree of any vertex is at most $|V|-1$, if equality holds for every vertex, $\Gamma$ is a \emph{complete graph}. Furthermore, $\Gamma$ is a \emph{k-regular graph} if all vertices have the same degree $k$.

Suppose $C$ is an nonempty subset of $V$. The subgraph of $\Gamma$ whose vertex set is $C$ and whose edge set is the set of those edges of $\Gamma$ that have both endpoints in $C$ is the \emph{subgraph of $\Gamma$ induced by $C$}. A \emph{clique} $C$ of a graph $\Gamma$ is a subset of $V$ such that the subgraph induced by $C$ is a complete graph. When a clique $C$ cannot be extended by including one more adjacent vertex, then $C$ is a \emph{maximal clique}.

A \emph{bipartite graph} is one whose vertex set is partitioned into two subsets $U,T$ so that each edge has one endpoint in $U$ and one endpoint in $T$; such a partition $(U,T)$ is a \emph{bipartition} of the graph with components $U$ and $T$. A bipartite graph $(U,T)$ is \emph{biregular} if all vertices of $U$ have the same degree as well as all vertices of $V$. If the degree of the vertices in $U$ is $a$ and the degree of the vertices in $T$ is $b$, then the graph is a \emph{$(a,b)$-biregular}.

A \emph{walk} in a graph $\Gamma$ is a finite non-null sequence $W=v_0\,e_1\,v_1\,e_2\cdots e_k\,v_k$, whose terms are alternately vertices and edges, such that the endpoints of $e_i$ are $v_{i-1}$ and $v_i$, for $1\le i\le k$. Furthermore, $W$ is a walk \emph{beginning} with $v_0$ and \emph{ending} with $v_k$, and $W$ is a \emph{closed walk} if $v_0=v_k$. If the edges $e_1,\ldots,e_k$ are distinct, then $W$ is a \emph{trial} in $\Gamma$. A closed trial is a $\emph{cycle}$, and $W$ is a \emph{Hamiltonian cycle} in $\Gamma$ whether the cycle $W$ contains every vertex of $\Gamma$. When a graph $\Gamma$ contains a hamiltonian cycle, then $\Gamma$ is a \emph{Hamiltonian graph}. When the vertices $v_0,\ldots,v_k$ in the trial $W$ are distinct, $W$ is a \emph{path} in $\Gamma$. To seek of simplicity, in our thesis we indicate a path with respect its vertices, i.e.,
$$W=v_0, v_1,\cdots ,v_k.$$If for every two distinct vertices $u,v$ of $\Gamma$ there exists a path beginning with $u$ and ending with $v$, then $\Gamma$ is a \emph{connected graph}. In a connected graph $\Gamma(V)$, for any two vertices $u,v$, the length $d_{\Gamma}(u,v)$ of a shortest path beginning with $u$ and ending with $v$ is the \emph{distance between the vertices} $u$ and $v$. Obviously, $d_{\Gamma}$ is a metric on $V$, and the maximum distance between to vertices of $\Gamma$ is the \emph{diameter} of $\Gamma$.

An automorphism of a graph is an edge-invariant bijection of the vertex set $V$. In our thesis, $\Aut(\Gamma)$ denotes the group of the automorphisms of $\Gamma$. A graph $\Gamma$ is \emph{vertex-transitive} if, for any two vertices $u,v$ of $\Gamma$, there is an automorphism $\textsf{h}$ of $\Gamma$ such that $\textsf{h}(u)=v$. We end this section by stating the famous Lov\'asz conjecture.
\begin{conj}{\rm{[The Lov\'asz Conjecture]}}
{\em{Every finite connected vertex-transitive graph contains a Hamiltonian cycle except the five known counterexamples; see \cite{LL,BL}. }}
\end{conj}

The \emph{graph of a permutation} $\pi$ on $[n]$ is the directed graph $\Gamma(\pi)$ on the vertex set $[n]$ and edges $(i,j)$ whenever $\pi_i=j$, for every $i\in[n]$. Clearly, the cycles of $\Gamma(\pi)$ are the cycles of the decomposition of $\pi$ in disjoint cycles.

For further background on graph theory; see \cite{BM,BL}.

\subsection{Cayley graphs}
\label{cay}
Let $G$ be a group generated by an inverse closed set $Y$. Let $Y'$ denote the set of all nontrivial elements of $Y$. According to the handbook \cite[Chapter 27.3]{BL}, the \emph{(left-invariant) Cayley graph} $\rm{Cay}(G,Y')$ is an undirected simple graph with vertex set $G$ whose edges are the pairs $\{g,g\sigma \}$ with $g\in G$ and $\sigma\in Y'$. Figure \ref{fig:pgraph} shows the Cayley graph ${\rm{Cay(Sym_4}},Y')$, where $Y'$ is the set of all exchanges of adjacent elements.
\begin{figure}[hb!]
\centering
\includegraphics[width=\textwidth]{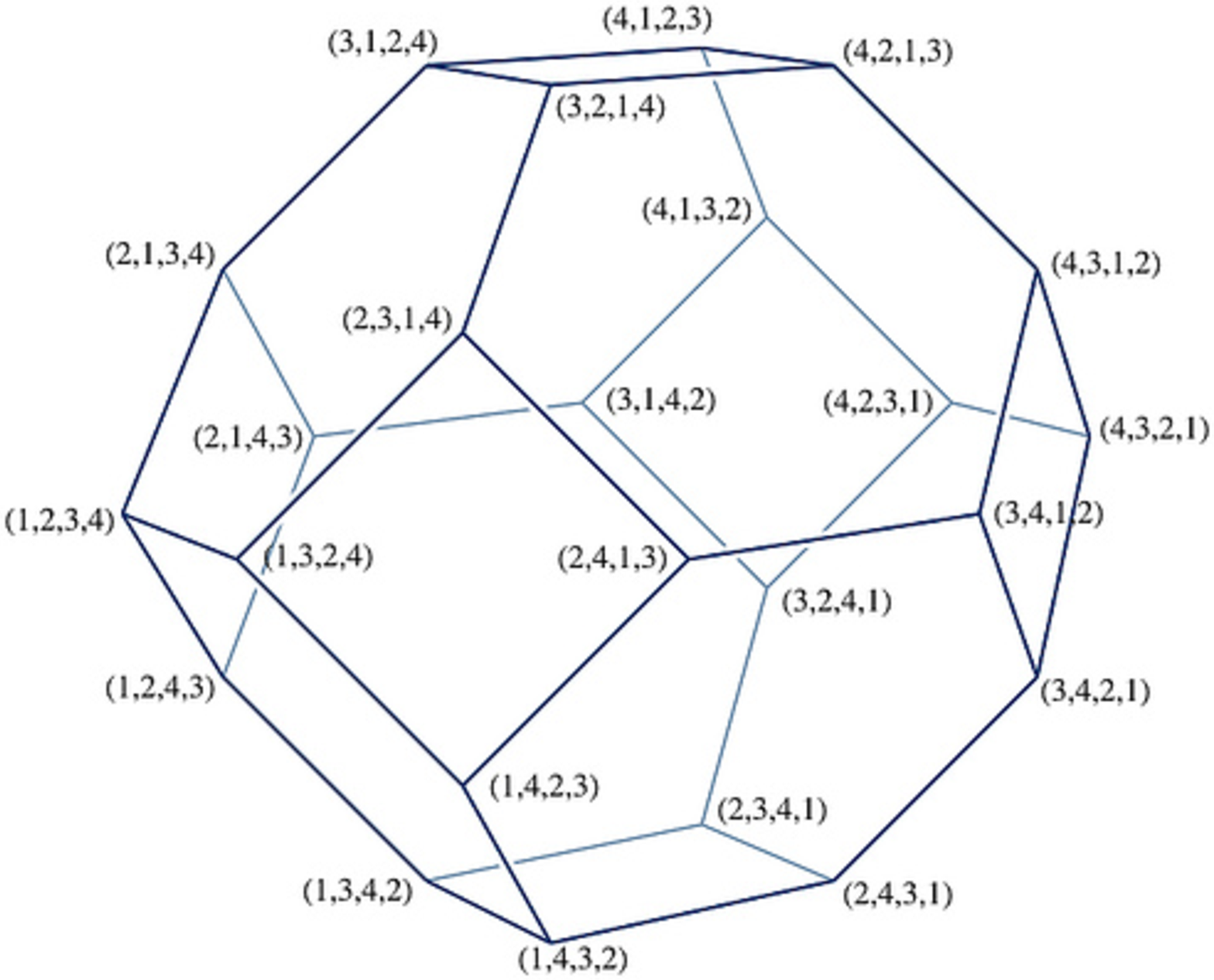}
\caption[The Permutohedron of order 4]{The \emph{Permutohedron} of order $4$, (reprinted from \cite{FL}).}
\label{fig:pgraph}
\end{figure}

In our thesis $G$ stands for the symmetric group $\Sym_n$ and $Y'$ is a rearrangement set $S$; see Section \ref{rear dist}.

By a classical result of Cayley, every $h\in \Sym_n$ defines a left translation $\textsf{h}$ which is the automorphism of $\Cayy$ that takes the vertex $\pi$ to the vertex $h\circ \pi$, and hence the edge $\{\pi,\rho\}$ to the edge $\{h\circ\pi,h\circ \rho\}$. Since we use the functional notation, we refer to $h$ as \emph{right translation}. These automorphisms form the \emph{right translation group} $R(\Cayy)$ of $\Cayy$. Clearly, $\Sym_n$ is isomorphic to $R(\Cayy)$. One may also consider the right-invariant Cayley graph whose edges are the pais $\{g,\sigma g\}$. This is admissible since the left-invariant and right-invariant Cayley graphs are isomorphic. In fact, the map taking any permutation to its inverse is such an isomorphism.

Since Cayley graphs are connected graphs, the distance between two permutations $\pi,\,\nu$, viewed as vertices of the right-invariant $\Cayy$ is the rearrangement distance between $\pi^{-1},\,\nu^{-1}$, and the rearrangement diameter of $\Sym_n$ is the diameter of $\Cayy$. Clearly, the rearrangements are the vertices of $\Cayy$ with distance $1$ from $\iota$.

There exists a vast literature on Cayley graphs. The interested reader is referred to \cite{AB}.

\section{Approximation algorithms}

This last section is dedicated to the reader unfamiliar with approximation algorithms.

It is common knowledge that many discrete optimization problems are $\textsf{NP}$-hard. Therefore, unless $\textsf{P}=\textsf{NP}$, there are no efficient algorithms to find optimal solutions to such problems, where an \emph{efficient algorithm} is one that runs in time bounded by a polynomial in its input size. If the widely verified conjecture that $\textsf{P}\neq\textsf{NP}$ were proved, we would not simultaneously have algorithms that find optimal solutions in polynomial time for any instance. At least one of these requirements must be relaxed in any approach to dealing with an $\textsf{NP}$-hard optimization problem. By far the most common approach is to relax the requirement of finding an optimal solution. This simplification has led to an enormous study of various types of heuristics such as genetic algorithms, and these techniques often yield good results in practice.

Throughout our thesis, we consider approximation algorithms for \emph{discrete optimization problems}. These algorithms try to find a solution that closely approximates the optimal solution in terms of its value. We assume that there is some objective function mapping each possible solution of an optimization problem to some nonnegative value, and an optimal solution to the optimization problem is one that either minimizes or maximizes the value of this objective function.

In his book \cite{WS} Williamson and Shmoys define an \emph{$\alpha$-approximation algorithm} for an optimization problem as polynomial time algorithm that for all instances of the problem produces a solution whose value is within a factor of $\alpha$ of the value of an optimal solution. For an $\alpha$-approximation algorithm, $\alpha$ is the \emph{performance guarantee} of the algorithm. In the literature, $\alpha$ is also often called the approximation ratio or approximation factor of the algorithm. Williamson and Shmoys follow the convention that $\alpha>1$ for minimization problems,
while $\alpha < 1$ for maximization problems. Since in our thesis we only discuss minimization problems, all approximation algorithms have performance guarantee bigger than $1$. Thus, a $2$-approximation algorithm is a polynomial-time algorithm that always returns a solution whose value is at most double the optimal value. The interested reader is referred to \cite{WS}.

\chapter{Block transpositions}\label{c3}

The most well-studied rearrangement is the block transposition. It should be noticed that a few authors use the shorter term of transposition that we avoid since the world ``transposition'' has a different meaning in the theory of permutation groups.

A block transposition, informally, is the operation that cuts out a certain portion (block) of a permutation and pastes it elsewhere in the same permutation. Equivalently, a block transposition is the operation that interchanges two adjacent substrings (blocks) of a permutation so that the order of entries within each block is unchanged.

\section{Notation and preliminaries}\label{s31}
For any three integers, named \emph{cut points}, $i,j,k$ with $0\leq i< j< k\leq n$, the \emph{block transposition} $\sigma(i,j,k)$ acts on a permutation $\pi$ on $[n]$ switching two adjacent subsequences of $\pi$, named \emph{blocks}, without altering the order of integers within each block. We define $\sigma(i,j,k)$ as a function:
\begin{equation}\label{feb9}
\sigma(i,j,k)_t =\left\{\begin{array}{ll}
t, &  1\leq t\leq i\quad k+1\leq t\leq n,\\
t+j-i, & i+1\leq t\leq k-j+i,\\
t+j-k, & k-j+i+1\leq t\leq k.
\end{array}
\right.
\end{equation}This shows that $\sigma(i,j,k)_{t+1}=\sigma(i,j,k)_{t} +1$ in the intervals:
\begin{equation}
\label{function}
[1,i],\quad[i+1, k-j+i],\quad[k-j+i+1,k],\quad[k+1, n],
\end{equation}where
\begin{align}
\label{cuppoints}
\sigma(i,j,k)_i&=i;&\sigma(i,j,k)_{i+1}&=j+1;&\sigma(i,j,k)_{k-j+i}&=k;\\
\sigma(i,j,k)_{k-j+i+1}&=i+1;&\sigma(i,j,k)_{k}&=j;&\sigma(i,j,k)_{k+1}&=k+1.\notag
\end{align}
Actually, $\sigma(i,j,k)$ can also be represented as the permutation
\begin{equation}
\label{eq22ott12}
\sigma(i,j,k)=\left\{\begin{array}{ll}
[1\cdots i\,\, j+1\cdots k\,\, i+1\cdots j\,\, k+1 \cdots n], & 1\leq i\quad k< n,\\
{[j+1\cdots k\,\, 1\cdots j\,\, k+1 \cdots n]}, & i=0\quad k< n,\\
{[1\cdots i\,\,j+1\cdots n\,\, i+1\cdots j]}, & 1\leq i\quad k=n,\\
{[j+1\cdots n\,\, 1\cdots j]}, & i=0\quad k=n
\end{array}
\right.
\end{equation}such that the action of $\sigma(i,j,k)$ on $\pi$ is defined as the product
$$\pi\circ\sigma(i,j,k)=[\pi_1\cdots \pi_i\,\,\pi_{j+1}\cdots \pi_k\,\,\pi_{i+1}\cdots \pi_{j}\,\,\pi_{k+1}\cdots \pi_n].$$Therefore, applying a block transposition on the right of $\pi$ consists in switching two adjacent subsequences of $\pi$, namely \emph{blocks}, without changing the order of the integers within each block. This may also be expressed by $$[\pi_1\cdots\pi_i|\pi_{i+1}\cdots\pi_j|\pi_{j+1}\cdots\pi_k|\pi_{k+1}\cdots\pi_n].$$

{}From now on, $S_n$ denotes the set of all block transpositions on $[n]$. The following example shows that $S_n$ is not a subgroup of $\rm{Sym_n}.$
\begin{example}
\em{Assume $n=8$. By (\ref{eq22ott12}), $\sigma(2,4,6)=[1\,2\,5\,6\,3\,4\,7\,8]$. Thus, from (\ref{feb9}),
$$\sigma(0,1,2)\circ\sigma(2,4,6)=[2\,1\,5\,6\,3\,4\,7\,8],$$ where there are five increasing substrings, namely $1-2-56-34-78.$ From (\ref{function}), a block transposition has at most four increasing substrings, then $\sigma(0,1,2)\circ\sigma(2,4,6)$ is not a block transposition.}
\end{example}

\section{Equations involving block transpositions} \label{equabl}

Now, we prove several equations involving block transpositions that are meaningful in Section \ref{sorting eq}.
\begin{lem}\label{j_1}
For any integers $i,j_1,j_2,k$ such that $0\leq i<j_1,j_2<k\leq n$ the following properties.
\begin{itemize}
\item[\rm(i)] $\sigma(i,j,k)=\sigma(i,i+1,k)^{j-i}$;
\item[\rm(ii)] $\sigma(i,j_2,k)\circ\sigma(i,j_1,k)=\sigma(i,i+t,k)$
\end{itemize}
hold, where $t$ is the smallest positive integer such that $t\equiv j_1+j_2-2i{\pmod{k-i}}$.
\end{lem}
\begin{proof}
(i) A straightforward computation of (\ref{feb9}) shows that
\begin{equation}\label{pg}
(\sigma(i,i+1,k)\circ\sigma(i,j,k))_t=\left\{\begin{array}{ll}
t, & 1\leq t\leq i\quad k+1\leq t\leq n,\\
t+j+1-i, & i+1\leq t\leq k-j+i-1,\\
t+j+1-k, & k-j+i\leq t\leq k.
\end{array}
\right.
\end{equation}Therefore, the above product is equal to $\sigma(i,j+1,k)$. From this, by induction on $m$,
$$\sigma(i,i+1,k)^m=\sigma(i,i+m,k),\quad \mbox{ for } m=1,\ldots, k-i-1,$$ and (i) follows.

(ii) Furthermore, by (\ref{pg}), $$\sigma(i,i+1,k)^{k-i}=\iota.$$This shows that for any two integers $i,k$ with $1\leq i<k\leq n$, the set of block transpositions $\sigma(i,j,k)$ with $j$ ranging in the interval $(i,k)$ is a cyclic group of order $k-i$, generating by $\sigma(i,i+1,k)$. Hence (ii) follows.
\end{proof}

\begin{cor}\label{matnov2}
\begin{itemize} $S_n$ has the following properties.
\item[\rm(i)] $|S_n|=n(n+1)(n-1)/6$.
\item[\rm(ii)] For any two positive integers $i,\,k$ with $i<k\le n$, the subgroup generated by $\sigma(i,i+1,k)$ consists of all $\sigma(i,j,k)$
together with the identity.
\item[\rm(iii)] $S_n$ is power and inverse closed. In particular, for any cut points $i,j,k$ and a positive integer $m$,
\begin{equation}\label{feb11++}
\sigma(i,j,k)^m=\sigma(i,i+t,k),
\end{equation}where $t$ is the smallest positive integer such that $t\equiv m(j-i) {\pmod{k-i}}$, and
\begin{equation}\label{eqa18oct}
\sigma(i,j,k)^{-1}=\sigma(i,k-j+i,k).
\end{equation}
\end{itemize}
\end{cor}

Here, we give some properties of block transpositions in terms of cycle permutations. By (\ref{feb9}),
\begin{equation}
\label{eq31nov}
\sigma(i,i+1,k)=(i+1,\cdots, k).
\end{equation}Therefore, from Corollary \ref{matnov2} (iii), a block transposition $\sigma(i,j,k)$ is a power of the cycle $(i+1,\cdots, k)$. More precisely,
\begin{equation}\label{pg1}
\sigma(i,j,k)=(i+1,\ldots,k)^t,\qquad t\equiv j-i{\pmod{k-i}}.
\end{equation}
\begin{rem}
{\em{By (\ref{eq31nov}), $\sigma(i,i+1,k)$ is a cycle of size $k-(i+1)$. Therefore, $\sigma(i,j,k)$ is a cycle if and only if the smallest positive integer $t$ such that $t\equiv j-i{\pmod{k-i}}$ is prime to $k-(i+1)$. This follows from (\ref{pg1}), taking into account that if $c$ is a cycle of $\Sym_n$ of size $d$, then a necessary and sufficient condition for $c^t$ to be a cycle is that $gcd(t,d)=1$.}}
\end{rem}
\begin{prop}
\label{22nov2013} For any two integers $k_1,k_2$ with $2\le k_1<k_2$,
$$\sigma(0,1,k_1)^{-1}\circ\sigma(0,1,k_2)=\sigma(k_1-1,k_1,k_2).$$
\end{prop}
\begin{proof} From (\ref{eq31nov}), $\sigma(0,1,k_1)^{-1}=(k_1,k_1-1,\ldots,1)$ and $\sigma(0,1,k_2)=(1,\ldots,k_2)$. Therefore, $$\sigma(0,1,k_1)^{-1}\circ\sigma(0,1,k_2)=(k_1,k_1+1,\ldots k_2).$$ Since $(k_1,k_1+1,\ldots k_2)=\sigma(k_1-1,k_1,k_2)$ by (\ref{pg1}), the statement follows.\qquad\end{proof}

{}From now on $\beta$ stands for $\sigma(0,1,n)$.  In particular, by Corollary \ref{matnov2} (ii), $\beta$ generates a subgroup of order $n$ that often appears in our arguments.
\begin{prop}
\label{nov23} For any two integers $i,k$ with $0\leq i<k<n$,
$$\beta\circ\sigma(i,i+1,k)\circ\beta^{-1}=\sigma(i+1,i+2,k+1).$$
\end{prop}
\begin{proof} By (\ref{eq31nov}), $\beta\circ \sigma(i,i+1,k)\circ \beta^{-1}$ is the cycle $(i+2,i+3,\ldots,k+1)$. Hence the statement follows from (\ref{eq31nov}).
\end{proof}

\begin{cor}\label{pz3}
For any cut points $i,j,k$ with $k\neq n$,
$$\beta\circ \sigma(i,j,k)\circ \beta^{-1}=\sigma(i+1,j+1,k+1).$$
\end{cor}
\begin{proof} By Lemma \ref{j_1} (i) and Proposition \ref{nov23},
$$\beta\circ\sigma(i,j,k)\circ\beta^{-1}= \sigma(i+1,i+2,k+1)^{j-i}.$$ Now, the claim follows from (\ref{feb11++}).
\end{proof}

We stress that the hypothesis $k<n$ in Proposition \ref{nov23} cannot be dropped as $\sigma(i,j+1,k+1)$ is not a block transposition when $k=n$. This gives a motivation for the following proposition.
\begin{prop}\label{23nov2013a}
For every integer $i$ with $0\leq i\leq n-2$,
$$\beta^{-i}\circ\sigma(i,i+1,n)\circ\beta^{-1}=\left\{
\begin{array}{ll}
\iota, & i=0,\\
\sigma(1,n-i,n), & i\geq 1.
\end{array}
\right.$$
\end{prop}
\begin{proof}
For $i=0$, the claim is a straightforward consequence of the definition of $\beta$.
Therefore, $i\ge 1$ is assumed. We show that
\begin{equation}\label{pz2}
\sigma(i,i+1,n)\circ\beta^{-1}=\sigma(0,i,i+1).
\end{equation}By (\ref{eqa18oct}), $\sigma(0,i,i+1)^{-1}=\sigma(0,1,i+1),$  then (\ref{pz2}) is equivalent to $$\sigma(i,i+1,n)=\sigma(0,1,i+1)^{-1}\circ\sigma(0,1,n).$$ Hence (\ref{pz2}) follows from Proposition \ref{22nov2013}. Now, since $\beta^{-i}=\beta^{n-i}=\sigma(0,n-i,n)$, the claim follows from (\ref{feb9}) and (\ref{pz2}).
\end{proof}

\begin{cor}
\label{19novbis2013}
For any two integers $i,j$ with $1\leq i<j\leq n-1$,
\begin{itemize}
\item[\rm(i)] $\sigma(i,j,n)\circ \beta^i=\beta^i\circ\sigma(0,j-i,n-i)$;
\item[\rm(ii)] $\beta^{n-j}\circ \sigma(i,j,n)\circ \beta^i=\sigma(n-j,n-j+i,n)$;
\item[\rm(iii)] $\beta^{n-j+1}\circ \sigma(i,j,n)\circ \beta^{-1}=\sigma(1,n-j+1,n-j+1+i)$.
\end{itemize}
\end{cor}
\begin{proof}
(i) By Lemma \ref{j_1} (i), $\sigma(i,j,n)$ is a power of $\sigma(i,i+1,n)$. This together with Proposition \ref{23nov2013a} gives
$$\beta^{-i}\circ\sigma(i,j,n)\circ\beta^{i}=(\sigma(1,n-i,n)\circ\beta^{i+1})^{j-i}.$$ By (\ref{feb9}),
$\sigma(1,n-i,n)\circ\beta^{i+1}=\sigma(0,1,n-i)$. Then, the claim in case (i) follows from (\ref{feb11++}).

(ii) Since $\beta^{n-j}\circ \sigma(i,j,n)\circ \beta^i=\beta^{n-j+i}\circ(\beta^{-i}\circ \sigma(i,j,n)\circ \beta^i)$, (ii) follows from (i) and (\ref{feb9}).

(iii) Also, as
$$\beta^{n-j+1}\circ \sigma(i,j,n)\circ\beta^{-1}=\beta\circ(\beta^{n-j}\circ \sigma(i,j,n)\circ\beta^i)\circ\beta^{-i-1},$$ (iii) follows from (ii) and (\ref{feb9}).
\end{proof}

We may observe that Corollary \ref{19novbis2013} (i) remains valid for $i=0$ while the products on the left-hand side of (ii), as well as of (iii), give the identity permutation.
\begin{lem}
\label{22novbis2013}
For any three integers $i,k,t$ with $1\leq i< k \leq n$ and $2\leq t\leq n-1$, the following hold.
\begin{itemize}
\item[\rm(i)] $\beta^{t-1}\circ \sigma(i,i+1,n)\circ \beta^{-t}=\sigma(t-1,i+t-1,i+t)$, for $i+t-1<n$;
\item[\rm(ii)] $\beta^{t}\circ \sigma(i,i+1,n)\circ \beta^{-t}=\sigma(i-n+t,i-n+t+1,r)$, for $i+t-1\geq n$;
\item[\rm(iii)] $\beta^{t}\circ \sigma(i,i+1,k)\circ \beta^{-t}=\sigma(i+t,i+t+1,k+t)$, for $t\le n- k$;
\item[\rm(iv)] $\beta^{n-i}\circ \sigma(i,i+1,k)\circ \beta^{-(n-k+1)}=\sigma(1,k-i,n)$;
\item[\rm(v)] $\beta^{t-1}\circ \sigma(i,i+1,k)\circ \beta^{-t}=\sigma(k-n-1+t,t+i-1,t+i)$, for $n-k+2\leq t\leq n- i$;
\item[\rm(vi)] $\beta^{t}\circ \sigma(i,i+1,k)\circ \beta^{-t}=\sigma(i+t-n,i+t-n+1,k+t-n)$, for $n-i+1\leq t\leq n- 1$.
\end{itemize}
\end{lem}
\begin{proof}
The proof is by induction on $t$.

(i) For $t=2$, the left-hand side of (i) is
\begin{equation}\label{pz4}
\beta^{i+1}\circ(\beta^{-i}\circ \sigma(i,i+1,n)\circ\beta^{-1})\circ\beta^{-2}.
\end{equation}Proposition \ref{23nov2013a} applied to $i\geq 1$ shows that (\ref{pz4}) is the same as $\beta^{i+1}\circ \sigma(i,n-i,n)\circ\beta^{-1}$. Since $i+2<n$, (i) for $t=2$ follows from Corollary \ref{19novbis2013} (iii). Suppose that (i) holds for $t-1$. As $n>i+t-1>i+t-2$, the inductive hypothesis yields that
$$\beta^{t-1}\circ \sigma(i,i+1,n)\circ \beta^{-t}=\beta\circ\sigma(t-2,i+t-2, i+t-1)\beta^{-1}.$$Therefore, (i) follows from Corollary \ref{pz3} by $i+t-1<n$.

(ii) Since $i\leq n-2$, the hypothesis in (ii) yields $t\ge 3$. Let $t=3$. Then $n-i=2$, and the left-hand side of (ii) reads
\begin{equation}\label{alf}
\beta^{n-i}\circ(\beta\circ\sigma(i,i+1,n)\circ\beta^{-2})\circ\beta^{-1}.
\end{equation}Observe that $\beta\circ\sigma(i,i+1,n)\circ\beta^{-2}$ can be computed applying (i) to $t=2$. The result is $\sigma(1,i+1,i+2)=\sigma(1,n-1,n)$, showing that (\ref{alf}) and $\beta^{n-i}\circ\sigma(1,n-1,n)\circ\beta^{-1}$ are the same. Now, (ii) for $t=3$ follows from Corollary \ref{19novbis2013} (iii) applied to $n-j+1=2$. Suppose that (ii) holds for $t-1$. According to the hypothesis $i+t-1\geq n$, two cases are distinguished, namely $i+t-2\geq n$ and $i+t-2=n-1$. In the former case, write the left-hand side of (ii) as
\begin{equation}\label{alf1}
\beta\circ(\beta^{t-1}\circ\sigma(i,i+1,n)\circ\beta^{-t+1})\circ\beta^{-1}.
\end{equation}By the inductive hypothesis, (\ref{alf1}) is the same as $$\beta\circ\sigma(i-n+t-1,i-n+t,t-1)\circ\beta^{-1}.$$Thus, (ii) follows from Corollary \ref{pz3} since $t-1\leq n-2<n$. In the latter case, write the left-hand side of (ii) as
\begin{equation}\label{alf2}
\beta^2\circ(\beta^{t-2}\circ\sigma(i,i+1,n)\circ\beta^{-(t-1)})\circ\beta^{-1}.
\end{equation}As $i+t-2<n$, (i) applied to $t-1$ shows that (\ref{alf2}) and $\beta^2\circ\sigma(t-2,n-1,n)\circ\beta^{-1}$ are the same.
Since $\sigma(i+n-t,i-n+t+1,r)=\sigma(1,2,t)$, it remains to observe that $$\beta^2\circ\sigma(t-2,n-1,n)\circ\beta^{-1}=\sigma(1,2,t)$$follows from Corollary \ref{19novbis2013} (iii) applied to $n-j+1=2$. Hence the statement holds in case (ii).

(iii) Since $t\leq n- k$, a straightforward inductive argument depending on Proposition \ref{nov23} completes the proof for case (iii).

(iv) The left-hand side of (iv) can be written as
\begin{equation}\label{us-3}
\beta^{k-i}\circ(\beta^{n-k}\circ\sigma(i,i+1,k)\circ\beta^{-n+k})\circ\beta^{-1}.
\end{equation}Observe that $\beta^{n-k}\circ\sigma(i,i+1,k)\circ\beta^{-n+k}$ can be computed using (iii) for $t=n-k$. The result is  $\sigma(i+n-k,i+n-k+1,n)$, showing that (\ref{us-3}) coincides with $$\beta^{k-i}\circ\sigma(i+n-k,i+n-k+1,n)\circ\beta^{-1}.$$ Hence (iv) follows from Proposition \ref{23nov2013a} applied to $i\geq 1$.

(v) Let $t=n-k+2$, the smallest value of $t$ admitted in (v). Then, the left-hand side of (v) reads
\begin{equation}\label{alf3}
\beta^{-k+i+1}\circ(\beta^{n-i}\circ\sigma(i,i+1,k)\circ\beta^{-n+k-1})\circ\beta^{-1}.
\end{equation}{}From (iv), $\beta^{n-i}\circ\sigma(i,i+1,k)\circ\beta^{-n+k-1}=\sigma(1,k-i,n)$ which shows that $\beta^{-k+i+1}\circ\sigma(1,k-i,n)\circ\beta^{-1}$ and (\ref{alf3}) are the same. Here, to show (v) for $t=n-k+2$, compute first Corollary \ref{19novbis2013} (iii) for $i=1$ and $j=k-i>1$. The result is $$\beta^{n-k+i+1}\circ\sigma(1,k-i,n)\circ \beta^{-1}=\sigma(1,n-k+i+1,n-k+i+2).$$ Since the right-hand side of this equation is equal to that in (v) for $t=n-k+2$, we are done. Suppose that (v) holds for $t-1$. As $n-i\geq t>t-1$, the inductive hypothesis yields that
$$\beta^{t-1}\circ \sigma(i,i+1,k)\circ \beta^{-t}=\beta\circ\sigma(k-n+t-2,t+i-2,t+i-1)\circ\beta^{-1}.$$Since $t+i-1<n$, Corollary \ref{pz3} applies and the claim follows in case (v).

(vi) For $t=n+1-i$, the left-hand side of (vi) is
\begin{equation}\label{pz7}
\beta^2\circ(\beta^{n-i-1}\circ \sigma(i,i+1,k)\circ\beta^{-n+i})\circ\beta^{-1}.
\end{equation} (v) applied to $t=n-i$ shows that (\ref{pz7}) is the same as $\beta^2\circ \sigma(k-i-1,n-1,n)\circ\beta^{-1}$. Corollary \ref{19novbis2013} (iii) after replacing $i$ by $k-i-1$ and $j$ by $n-1$ gives $$\beta^2\circ \sigma(i,n-1,n)\circ\sigma^{-1}=\sigma(1,2,k-i+1)$$
which is exactly the right-hand side in (vi) for $t=n+i-1$. Here, suppose that (vi) holds for $t-1$. As $n-1\geq t>n+1-i$, the inductive hypothesis yields that
$$\beta^{r}\circ \sigma(i,i+1,k)\circ \beta^{-t}=\beta\circ\sigma(i+t-n-1,i+t-n, k+t-n-1)\circ\beta^{-1}.$$ Since $k+t-1-n-1<n$, the claim follows from Corollary \ref{pz3} in case (vi). This completes the proof.
\end{proof}

\begin{thm}\label{mainsec}
For any three integers $i,k,t$ with $1\leq i<k\leq n$, there exists an integer $s$ such that
$$\beta^s\circ\sigma(i,i+1,k)\circ\beta^{-t}=\sigma(i',j',k')$$ with $1\leq i'<j'<k'\leq n$.
\end{thm}
\begin{proof}
As $\beta^n=\iota$, it suffices to prove the theorem for $1\leq t\leq n-1$. If $t=1$, the claim follows from Proposition \ref{23nov2013a}, applied to $i\geq 1$, and Proposition \ref{nov23}.

For $2\leq t\leq n-1$, the claim follows from Lemma \ref{22novbis2013}. In fact, $i'\geq 1$ holds in all cases.
\end{proof}

\section{The toric equivalence in the symmetric group}\label{s32}

Before discussing and stating the contributions obtained in our thesis, it is convenient to exhibit some useful equivalence relations on permutations introduced by Eriksson and his coworkers; see \cite{EE}. For this purpose, we consider permutations on the set $[n]^0=\{0,\ldots,n\}$ together with its block transpositions acting on the symmetric group $\Sym_n^0$ on $[n]^0$. For any $-1\leq i< j < k\le n$,
we have such a block transposition $\overline{\sigma}(i,j,k)$ on $[n]$. They form the set $\overline{S_n}$ containing $S_n$, in the sense that every block transposition $\sigma(i,j,k)$ is naturally embedded in $\overline{S_n}$ by the map $\sigma(i,j,k)\mapsto [0\,\sigma(i,j,k)]$. Here, and in the sequel, $[0\,\pi]$ stands for the permutation $[0\,\pi_1\,\cdots \pi_n]$ on $[n]^0$ arising from a permutation $\pi$ on $[n]$.

The equations obtained in Section \ref{equabl} apply to $\overline{S_n}$ whenever one takes into account that $n$ is replied by $n+1$, and $x-1$ replies every $x$ with $1\leq x\leq n$. Therefore, $\beta$ is replaced by $\alpha=\overline{\sigma}(-1,0,n)$, where $\alpha^{n+1}=\iota$ and $\beta^{-1}$ is substituted by $\alpha^n$. In particular, the statement of Theorem \ref{mainsec} applied to $[n]^0$ reads:
\begin{prop}\label{22novter2013a}
For any three integers $i,k,t$ with $0\leq i<k\leq n$, there exists an integer $s$ such that
$$\alpha^s\circ\bar\sigma(i,i+1,k)\circ\alpha^{-t}=\bar\sigma(i',j',k')$$ with $0\leq i'<j'<k'\leq n$.
\end{prop}
The equations of Corollary \ref{19novbis2013} applied to $[n]^0$ reads:\\
For any two integers $i,j$ with $-1\leq i-1<j-1\leq n-2$,
\begin{itemize}
\item[\rm(i)] $\bar\sigma(i-1,j-1,n)\circ \alpha^{i-1}=\alpha^{i-1}\circ\bar\sigma(-1,j-i,n-i+1)$;
\item[\rm(ii)] $\alpha^{n-j+2}\circ \bar\sigma(i-1,j-1,n)\circ \alpha^{i-1}=\bar\sigma(n-j+1,n-j+i,n)$;
\item[\rm(iii)] $\alpha^{n-j+2}\circ \bar\sigma(i-1,j-1,n)\circ \alpha^{n}=\bar\sigma(0,n-j+2,n-j+2+i)$.
\end{itemize} However, replying $i-1$ with $i$ and $j-1$ with $j$, the equations of Corollary \ref{19novbis2013} hold also for $[n]^0$.

Eriksson and his coworkers observed that the $n+1$ permutations arising from $\overline{\pi}\in \Sym_n^0$ under cyclic index shift form an equivalence class $\pi^\circ$ containing a unique permutation of the form $[0\,\pi]$. A formal definition of $\pi^\circ$ is given below.
\begin{defi}\label{circ}
{\em{Let $\pi$ be a permutation on $[n]$. The \emph{circular permutation class} $\pi^\circ$ is obtained from $\pi$ by inserting an extra element $0$ that is considered a predecessor of $\pi_1$ and a successor of $\pi_n$ and taking the equivalence class under cyclic index shift. So $\pi^\circ$ is circular in positions being represented by $[0\,\pi_1\cdots \pi_{n-1}\,\pi_n],\,[\pi_n\, 0\,\pi_1\cdots\pi_{n-1}],$ and so on. The \emph{linearization of a circular permutation $\pi^\circ$} is a permutation $\pi$ obtained by removing the e-\\lement $0$ and letting its successor be the first element of $\pi$. It is customary to denote by $\pi^\circ$ any representative of the circular class of $\pi$ and $\equiv^\circ$ the equivalence relation.}}
\end{defi}
A necessary and sufficient condition for a permutation $\overline\pi$ on $[n]^0$ to be in $\pi^\circ$ is the existence of an integer $r$ with $0\le r \le n$ such that
\begin{equation}\label{aug4}
\overline\pi_x=[0\,\pi]_{x+r},\quad \mbox{ for }0\le x \le n,
\end{equation}where the indices are taken mod$(n+1)$.
The second equivalence class is an expansion of the circular permutation class, as it also involves cyclic value shifts.
\begin{defi}
{\em{Let $\pi$ be a permutation on $[n]$, and let $m$ be an integer with $1\leq m\leq n$. The $m$-step cyclic value shift of the circular permutation $\pi^\circ$ is the circular permutation $m+\pi^\circ =[m\,m + \pi_1\cdots m+ \pi_n]$, where the integers are taken mod$(n+1)$. The \emph{toric class in $\Sym_n^0$} $\pi_\circ^\circ$ is obtained from $\pi^\circ$ by taking the $m$-step cyclic value shifts of the circular permutations in $\pi^\circ$. So, $\pi_\circ^\circ$ is circular in values, as well as in positions.}}
\end{defi}
In general, the toric class of $\pi$ comprises $(n+1)^2$ permutations, but it may consist of a smaller number of permutations and can even collapse to a unique permutation. The latter case occurs when $\pi$ is the identity permutation or the reverse permutation. The number of elements in a toric class is always a divisor of $n+1$, and there are exactly $\varphi(n+1)$ classes that have only one element, where $\varphi$ is the Euler function; see \cite{C}.
The following example comes from \cite{la}.
\begin{example}\label{exla}
\em{Let $n=7$ and let $\pi=[4\,1\, 6\, 2\, 5\, 7\, 3]$. Then $\pi^\circ=[0\,4\,1\, 6\, 2\, 5\, 7\, 3]$, and $\pi_\circ^\circ$ consists of the permutations below together with their circular classes.
$$\begin{array}{lll}
0 + \pi^\circ = [0\, 4\, 1\, 6\, 2\, 5\, 7\, 3],&1 + \pi^\circ = [1\, 5\, 2\, 7\, 3\, 6\, 0\, 4],&2 + \pi^\circ = [2\, 6\, 3\, 0\, 4\, 7\, 1\, 5],\\
3 + \pi^\circ = [3\, 7\, 4\, 1\, 5\, 0\, 2\, 6],&4 + \pi^\circ = [4\, 0\, 5\, 2\, 6\, 1\, 3\, 7],&5 + \pi^\circ = [5\, 1\, 6\, 3\, 7\, 2\, 4\, 0],\\
6 + \pi^\circ = [6\, 2\, 7\, 4\, 0\, 3\, 5\, 1],&7 + \pi^\circ = [7\, 3\, 0\, 5\, 1\, 4\, 6\, 2].\\
\end{array}$$}
\end{example}A necessary and sufficient condition for two permutations $\overline{\pi},\overline{\rho}\in\Sym_n^0$ to be in the same toric class is the existence of integers $r,s$ with $0\le r,s \le n$ such that $\rho=\pi_{x+r}-\pi_s$ holds for every $1\le x \le n,$ where the indices are taken mod $n+1$. In particular, for $\overline{\pi}=[0\,\pi]$ and $\overline{\rho}=[0\,\rho]$, this necessary and sufficient condition reads: there exists an integer $r$ with $0\le r \le n$ such that\begin{equation}\label{circ_pw}
\rho=\pi_{x+r}-\pi_r,\quad \mbox{ for }1\le x \le n,
\end{equation}where the indices are taken mod$(n+1)$. This gives rise to the following definition already introduced in \cite[Definition 7.3]{la} but not appearing explicitly in \cite{EE}.
\begin{defi}
{\em{Two permutations $\pi$ and $\rho$ on $[n]$ are \emph{torically equivalent} if $[0\,\pi]$ and $[0\,\rho]$ are in the same toric class.}}\end{defi}In Example \ref{exla}, the torically equivalent permutations are
$$\begin{array}{llll}
{[4\, 1\, 6\, 2\, 5\, 7\, 3],}&[4\, 1\, 5\, 2\, 7\, 3\, 6],&[4\, 7\, 1\, 5\, 2\, 6\, 3],&[2\, 6\, 3\, 7\, 4\, 1\, 5],\\
{[5\, 2\, 6\, 1\, 3\, 7\, 4],}&[5\, 1\, 6\, 3\, 7\, 2\, 4],&[3\, 5\, 1\, 6\, 2\, 7\, 4],&[5\, 1\, 4\, 6\, 2\, 7\, 3].\\
\end{array}$$
Since $\alpha=[1\,2\cdots n\,0]$, by Definition \ref{circ}, every circular permutation $\pi^\circ$ is the product of $[0\,\pi]$ by a power of $\alpha$, namely
\begin{equation}\label{alpha_powers}
\alpha^r_x\equiv x+r {\pmod{n+1}},\quad \mbox{ for } 0\leq x\leq n.
\end{equation}Take a permutations $\pi$ on $[n]$. From (\ref{aug4}), a necessary and sufficient condition for a permutation $\overline\pi$ on $[n]^0$ to be in $\pi^\circ$ is the existence of an integer $r$ with $0\le r \le n$ such that $\overline\pi=[0\,\pi]\circ\alpha^r$. Therefore, by (\ref{circ_pw}), a permutation $\rho$ on $[n]$ is torically equivalent to $\pi$ if and only if
\begin{equation}\label{toric_zero}
[0\,\rho]=\alpha^{-\pi_r}\circ[0\,\pi]\circ\alpha^r,\quad \mbox{ for } 0\le r \le n.
\end{equation}
Since $(\alpha^{-\pi_r}\circ[0\,\pi]\circ\alpha^r)_x=\pi_{x+r}-\pi_r$ for every $0\le x \le n$, this gives rise to the \emph{toric map} $\textsf{f}_r$ on $\Sym_n$ with $0\le r\le n,$ defined by
\begin{equation}\label{eq2oct9}
\textsf{f}_r(\pi)=\rho\Longleftrightarrow[0\,\rho]= \alpha^{-\pi_r}\circ [0\,\pi] \circ \alpha^r.
\end{equation}
\begin{defi}
\em{The \emph{toric class in $\Sym_n$} of $\pi$ is
\begin{equation}\label{eq1oct10}
\textsf{F}(\pi)=\{\textsf{f}_r(\pi)|\,r=0,1,\ldots,n\}.
\end{equation}}
\end{defi}
Since
\begin{equation}
\label{eq9oct}
(\textsf{f}_r(\pi))_t=\pi_{r+t}-\pi_r,\qquad\mbox{ for }1\leq t\leq n,
\end{equation}where the indices are taken mod$(n+1)$.

{}From (\ref{eq2oct9}), $\textsf{f}_s\circ \textsf{f}_r=\textsf{f}_{s+r}$, where the indices are taken mod$(n+1)$. Hence, $\textsf{f}_r=\textsf{f}^{\,r}$ and $\textsf{f}_{-r}=\textsf{f}_{\,n+1-r}$ with $\textsf{f}=\textsf{f}_1$, and the set
$$\textsf{F}=\{\textsf{f}_r|\,r=0,1,\ldots, n\}$$ is a cyclic group of order $n+1$ generated by $\textsf{f}$.

If $\pi\in \Sym_n$ and $0\le r \le n$, then
\begin{equation}
\label{lem3oct11}
\textsf{f}^{-1}_{r}(\pi)=\textsf{f}_{\pi_r}(\pi^{-1}).
\end{equation}
In particular, $\textsf{f}^{-1}_{r}(\pi)=\textsf{f}_{r}(\pi^{-1})$ provided that $\pi_r=r$.

The \emph{reverse map} $\textsf{g}$ on $\rm{Sym_n}$ is defined by
\begin{equation}
\label{eq2oct9b}
\textsf{g}(\pi)=\rho\Longleftrightarrow[0\,\rho]=[0\,w]\circ [0\,\pi] \circ [0\,w].
\end{equation}
$\textsf{g}$ is an involution, and
\begin{equation}
\label{eq9octb}
(\textsf{g}(\pi))_t=n+1-\pi_{n+1-t},\qquad\mbox{ for }1\leq t\leq n.
\end{equation}
Also, for all $0\le r\le n$,
\begin{equation}
\label{eqoct15a}
\textsf{g}\circ\textsf{f}_r \circ \textsf{g}=\textsf{f}_{n+1-r}
\end{equation}since $[0\,w]\circ\alpha^r \circ [0\,w]=\alpha^{-r}.$

\section{The equivalence with sorting circular permutations}\label{sorting eq}

In the study of the problem of sorting a permutation by block transpositions, several authors have tacitly allowed the possibility of replacing permutations $\pi$ on $[n]$ with the corresponding circular permutations $\pi^\circ$. Doing so, the following claim has actually been accepted to be true.
\begin{prop}\label{sort eq}
The problem of sorting permutations by block transpositions is equivalent to the problem of sorting circular permutations by block transpositions.
\end{prop}
That this has been an issue, it was observed by Hartman and Shamir \cite{HS}, even though they did not address the question whether such replacements might cause gaps in the proofs. Here, we settle this question definitely by proving the proposition below from which Proposition \ref{sort eq} follows, being $[0\,\pi]$ a representant of the circular class of $\pi$.
\begin{prop}
\label{th23jan} For any permutation $\pi$ on $[n]$,
$$d(\pi)=d([0\,\pi]).$$
\end{prop}
\begin{proof}
A minimum factorization of $\pi$ induces a factorization of $[0\,\pi]$, then $d([0\,\pi])\leq d(\pi)$. Now, we show that $d([0\,\pi])\geq d(\pi)$.
Let $m=d([0\,\pi])$ and take $\overline{\sigma}_1,\ldots,\overline{\sigma}_m\in \overline{S_n}$ such that
\begin{equation}
\label{eqjan82}
[0\,\pi]=\overline{\sigma}_1\circ\cdots \circ\overline{\sigma}_m,
\end{equation}where $\overline{\sigma}_u=\overline{\sigma}(i,j,k)$ with some $-1\le i<j<k\le n$ depending on $u$ for $1\leq u\leq m$.
For $k<n$, Corollary \ref{pz3} applied to $\overline{S_n}$ yields $$\overline{\sigma}_u=
\alpha^{-1}\circ\overline{\sigma}(i+1,j+1,k+1)\circ\alpha.$$ As $i+1\geq 0$, we have $\overline{\sigma}(i+1,j+1,k+1)=[0\,\sigma(i+1,j+1,k+1)]$. Then, denoting $\sigma(i+1,j+1,k+1)$ by $\sigma_u(i_u,j_u,k_u)$, we get
$$\overline{\sigma}_u=\alpha^{-1}\circ [0\,\sigma_u]\circ \alpha.$$
Therefore, each such $\overline{\sigma}_u$ with $k<n$ may be replaced by $\alpha^{-1}\circ [0\,\sigma_u]\circ \alpha$ in (\ref{eqjan82}). If $k=n$ and $i\geq 0$, $$\overline{\sigma}_u=\overline{\sigma}(i,j,n)=\alpha^{-n+j-1}\circ[0\,\sigma_u]\circ\alpha$$with $\sigma_u=\sigma(i_u,j_u,n)\in S$. On the other hand, $\overline{\sigma}(-1,j,n)=\alpha^{j-1}$, by Lemma \ref{j_1} (i). From this, $[0\,\pi]$ is product of powers of $\alpha$ and
block transpositions of $S_n$ embedded in $\overline{S_n}$. Using Lemma \ref{j_1} (i), we may also replace any block transposition $[0\,\sigma_u]$ by $[0\,\sigma(i_u,i_u+1,k_u)]^{j_u-i_u}$. Now, Proposition \ref{22novter2013a} shows that
\begin{equation}\label{jo}
[0\,\pi]=\alpha^t\circ[0\,\tau_1]\circ \cdots \circ [0\,\tau_m],
\end{equation}where $\tau_1,\ldots,\tau_m\in S_n$ and $0\le t \le n$. Actually $t=0$, since $[0\pi]$ begins with $0$, and the image
of $0$ in the right-hand side of (\ref{jo}) is $\alpha^t_0=t$. Therefore
$$[0\,\pi]=[0\,\tau_1]\circ \cdots \circ [0\,\tau_m],$$whence $\pi=\tau_1\circ\cdots \circ \tau_m$. This proves that $d([0\,\pi])\geq d(\pi)$.
\end{proof}

\section{The Shifting lemma}

Proposition \ref{22novter2013a} shows that for any integers $i,k,r$ with $0\leq i<k\leq n$, there exists an integer $s$  and cut points $i',j',k'$ such that $$\sigma(i,i+1,k)\circ\alpha^{n+1-t}=\alpha^{n+1-s}\circ\sigma(i',j',k').$$In this section we compute the exact values of $s,i',j',k'$.

\begin{lem}{\rm{[Shifting Lemma]}}
\label{lemc19ott2013}
Let $\sigma(i,j,k)$ be any block transposition on $[n]$. Then, for every integer $r$ with $0\leq r\leq n$, there exists a block transposition $\sigma(i',j',k')$ on $[n]$ such that
$$[0\,\sigma(i,j,k)]\circ\alpha^r=\alpha^{[0\,\sigma(i,j,k)]_r}\circ[0\,\sigma(i',j',k')].$$
\end{lem}
\begin{proof}
Let $\sigma=[0\,\sigma(i,j,k)]$. Since $0\leq i<j<k\leq n$, one of the following four cases can only occur:
\begin{itemize}
\item[\rm(I)] $0\leq i-r<k-j+i-r<k-r\leq n$;
\item[\rm(II)]$0\leq k-j+i-r<k-r<n+1+i-r\leq n$;
\item[\rm(III)]$0\leq k-r<n+1+i-r<n+1+k-j+i-r\leq n$;
\item[\rm(IV)]$0\leq n+1+i-r<n+1+k-j+i-r<n+1+k-r\leq n$.
\end{itemize}
If the hypothesis in case (I) is satisfied, we have, by (\ref{cuppoints}) and (\ref{alpha_powers}),
\begin{align}
(\sigma\circ\alpha^r)_0&=r;&(\sigma\circ\alpha^r)_{i-r}&=i;&(\sigma\circ\alpha^r)_{i+1-r}&=j+1;\notag \\
(\sigma\circ\alpha^r)_{k-j+i-r}&=k;&(\sigma\circ\alpha^r)_{k-j+i+1-r}&=i+1;&(\sigma\circ\alpha^r)_{k-r}&=j;\notag\\
(\sigma\circ\alpha^r)_{k+1-r}&=k+1;&(\sigma\circ\alpha^r)_{n-r}&=n;&(\sigma\circ\alpha^r)_{n+1-r}&=0;\notag\\
(\sigma\circ\alpha^r)_n&=r-1,\notag
\end{align}where superscripts and indices are taken mod$(n+1)$. By (\ref{function}), $(\sigma\circ\alpha^r)_{t+1}=(\sigma\circ\alpha^r)_{t} +1$ in the intervals $[0,i-r],\,[i-r+1,k-j+i-r],\,[k-j+i-r+1,k-r],\,[k-r+1,n-r],\,[n-r+1,n]$. From this,
$$(\alpha^{-r}\circ\sigma\circ\alpha^r)_{t+1}=(\alpha^{-r}\circ\sigma\circ\alpha^r)_{t} +1$$in the intervals $[0,i-r],\,[i-r+1,k-j+i-r],\,[k-j+i-r+1,k-r],\,[k-r+1,n]$, and
\begin{align}
(\alpha^{-r}\circ\sigma\circ\alpha^r)_0&=0;&(\alpha^{-r}\circ\sigma\circ\alpha^r)_{i-r}&=i-r; \notag\\
(\alpha^{-r}\circ\sigma\circ\alpha^r)_{i+1-r}&=j+1-r;&(\alpha^{-r}\circ\sigma\circ\alpha^r)_{k-j+i-r}&=k-r;\notag\\
(\alpha^{-r}\circ\sigma\circ\alpha^r)_{k-j+i-r}&=k-r;&(\alpha^{-r}\circ\sigma\circ\alpha^r)_{k-j+i+1-r}&=i+1-r;\notag\\
(\alpha^{-r}\circ\sigma\circ\alpha^r)_{k-r}&=j-r;&(\alpha^{-r}\circ\sigma\circ\alpha^r)_{k+1-r}&=k+1-r;\notag\\
(\alpha^{-r}\circ\sigma\circ\alpha^r)_n&=n.\notag
\end{align}Since $\sigma_r=r$ and $0\leq i-r<j-r<k-r\leq n$, the statement follows in case (I) from (\ref{function}) and (\ref{cuppoints}) with $i'=i-r,\,j'=j-r,\,k'=k-r$.
Now, suppose $0\leq k-j+i-r<k-r<n+1+i-r\leq n$. By (\ref{cuppoints}) and (\ref{alpha_powers}),
\begin{align}
(\sigma\circ\alpha^r)_0&=\sigma_r;&(\sigma\circ\alpha^r)_{k-j+i-r}&=k;&(\sigma\circ\alpha^r)_{k-j+i+1-r}&=i+1; \notag\\
(\sigma\circ\alpha^r)_{k-r}&=j;&(\sigma\circ\alpha^r)_{k+1-r}&=k+1;&(\sigma\circ\alpha^r)_{n-r}&=n; \notag\\
(\sigma\circ\alpha^r)_{n+1-r}&=0;&(\sigma\circ\alpha^r)_{n+1+i-r}&=i;&(\sigma\circ\alpha^r)_{n+2+i-r}&=j+1; \notag\\
(\sigma\circ\alpha^r)_n&=\sigma_r-1, \notag
\end{align}where superscripts and indices are taken mod$(n+1)$. Therefore, we obtain $\sigma_r-j-2=n-(n+2+i-r)$ hence $\sigma_r=-(i-j-r)$, and
\begin{align}
(\alpha^{i-j-r}\circ\sigma\circ\alpha^r)_0&=0;&(\alpha^{i-j-r}\circ\sigma\circ\alpha^r)_{k-j+i-r}&=k+i-j-r; \notag\\
(\alpha^{i-j-r}\circ\sigma\circ\alpha^r)_{k-j+i+1-r}&=n+2+2i-j-r;&(\alpha^{i-j-r}\circ\sigma\circ\alpha^r)_{k-r}&=n+1+i-r; \notag\\
(\alpha^{i-j-r}\circ\sigma\circ\alpha^r)_{k+1-r}&=k+1+i-j-r;\notag\\
(\alpha^{i-j-r}\circ\sigma\circ\alpha^r)_{n+1+i-r}&=n+1+2i-j-r;&(\alpha^{i-j-r}\circ\sigma\circ\alpha^r)_{n+2+i-r}&=n+2+i-r; \notag\\
(\alpha^{i-j-r}\circ\sigma\circ\alpha^r)_n&=n.\notag
\end{align}(\ref{function}) gives $(\sigma\circ\alpha^r)_{t+1}=(\sigma\circ\alpha^r)_{t} +1$ in the intervals $[0,k-j+i-r],\,[k-j+i-r+1,k-r],\,[k-r+1,n-r],\,[n-r+1,n+1+i-r],\,[n+2+i-r,n]$. Therefore, $(\alpha^{i-j-r}\circ\sigma\circ\alpha^r)_{t+1}=(\alpha^{i-j-r}\circ\sigma\circ\alpha^r)_{t} +1$ in the above intervals. Since $k-r=n+1+i-r-(n+1+2i-j-r)+k-j+i-r$, the statement follows in case (II) from (\ref{function}) and (\ref{cuppoints}) with $i'=k-j+i-r,\,j'=n+1+2i-j-r,\,k'=n+1+i-r$. Assume $0\leq k-r<n+1+i-r<n+1+k-j+i-r\leq n$. By (\ref{cuppoints}) and (\ref{alpha_powers}),
\begin{align}
(\sigma\circ\alpha^r)_0&=\sigma_r;&(\sigma\circ\alpha^r)_{k-r}&=j;&(\sigma\circ\alpha^r)_{k+1-r}&=k+1; \notag\\
(\sigma\circ\alpha^r)_{n-r}&=n;&(\sigma\circ\alpha^r)_{n+1-r}&=0;&(\sigma\circ\alpha^r)_{n+1+i-r}&=i; \notag\\
(\sigma\circ\alpha^r)_{n+2+i-r}&=j+1;&(\sigma\circ\alpha^r)_{n+1+k-j+i-r}&=k;&(\sigma\circ\alpha^r)_{n+2+k-j+i-r}&=i+1; \notag\\
(\sigma\circ\alpha^r)_n&=\sigma_r-1,\notag
\end{align}where superscripts and indices are taken mod$(n+1)$. It is straightforward to check that $\sigma_r-2-i=n-(n+2+k-j+i-r)$ hence $\sigma_r=-(k-j-r)$. Furthermore, the following relations
\begin{align}
(\alpha^{k-j-r}\circ\sigma\circ\alpha^r)_0&=0;&(\alpha^{k-j-r}\circ\sigma\circ\alpha^r)_{k-r}&=k-r; \notag\\
(\alpha^{k-j-r}\circ\sigma\circ\alpha^r)_{k+1-r}&=2k-j-r+1; \notag\\
(\alpha^{k-j-r}\circ\sigma\circ\alpha^r)_{n+1+i-r}&=n+1+k-j+i-r; \notag\\
(\alpha^{k-j-r}\circ\sigma\circ\alpha^r)_{n+2+i-r}&=k-r+1; \notag\\
(\alpha^{k-j-r}\circ\sigma\circ\alpha^r)_{n+1+k-j+i-r}&=2k-j-r; \notag\\
(\alpha^{k-j-r}\circ\sigma\circ\alpha^r)_{n+2+k-j+i-r}&=n+2+k-j+i-r;&(\alpha^{k-j-r}\circ\sigma\circ\alpha^r)_n&=n \notag
\end{align}hold. By (\ref{function}), $(\sigma\circ\alpha^r)_{t+1}=(\sigma\circ\alpha^r)_{t} +1$ in the intervals $[0,k-r],\,[k-r+1,n-r],\,[n-r+1,n+1+i-r],\,[n+2+i-r,n+1+k-j+i-r],\,[n+2+k-j+i-r,n]$. Then we obtain $(\alpha^{k-j-r}\circ\sigma\circ\alpha^r)_{t+1}=(\alpha^{k-j-r}\circ\sigma\circ\alpha^r)_{t} +1$ in the same intervals. Since $n+1+i-r=n+1+k-j+i-r-(2k-j-r)+k-r$, the statement follows in case (III) from (\ref{function}) and (\ref{cuppoints}) with $i'=k-r,\,j'=2k-j-r,\,k'=n+1+k-j+i-r$. To deal with case (IV), it is enough to use the same argument of case (I) replying $i-r$ with $n+1+i-r$, $j-r$ with $n+1+j-r$, and $k-r$ with $n+1+k-r$. Hence the statement follows with $i'=n+1+i-r,\,j'=n+1+j-r,\,k'=n+1+k-r$ and $\sigma_r=r$.
\end{proof}
In the proof of Lemma \ref{lemc19ott2013}, several equations linking $\alpha$ and block transpositions are given. Some of these are also useful for the present investigation and listed below in terms of toric maps.
\begin{cor}\label{eqbt}
For any positive integer $r\le n,$
\begin{itemize}
\item[\rm(i)] $\emph{\textsf{f}}_r(\sigma(i,j,k))=\sigma(i-r,j-r,k-r)$ if $0\leq i-r<k-j+i-r<k-r\leq n$;
\item[\rm(ii)] $\emph{\textsf{f}}_r(\sigma(i,j,k))=\sigma(k-j+i-r,n+1+2i-j-r,n+1+i-r)$ if $0\leq k-j+i-r<k-r<n+1+i-r\leq n$;
\item[\rm(iii)] $\emph{\textsf{f}}_r(\sigma(i,j,k))=\sigma(k-r,2k-j-r,n+1+k-j+i-r)$ if $0\leq k-r<n+1+i-r<n+1+k-j+i-r\leq n$;
\item[\rm(iv)] $\emph{\textsf{f}}_r(\sigma(i,j,k))=\sigma(n+1+i-r,n+1+j-r,n+1+k-r)$ if $0\leq n+1+i-r<n+1+j-r<n+1+k-r\leq n$.
\end{itemize}
\end{cor}
The following result states the invariance of $S_n$ under the action of toric maps and the reverse map.
\begin{prop}
\label{th1} Toric maps and the reverse map take any block transposition to a block transposition.
\end{prop}
\begin{proof}
For toric maps, the assertion follows from Lemma \ref{lemc19ott2013}. For the reverse map, (\ref{eq22ott12}) yields
\begin{equation}\label{eq1oct11}
\textsf{g}(\sigma(i,j,k))=\sigma(n-k,n-j,n-i)
\end{equation}
whence the assertion follows.
\end{proof}

\chapter{Block transposition distance}\label{c4}

In Section \ref{rear dist} we have discussed the concept of a rearrangement distance in a general setting. From now on, expect in Chapter \ref{c7}, we focus on the case of $S=S_n$, where $S_n$, as in Chapter \ref{c3}, denotes the set of block transpositions of $\Sym_n$. Consequently, the term of the rearrangement distance (diameter) is replaced by \emph{block transposition distance (diameter)}. Furthermore, sorting a permutation by block transpositions is equivalent to compute the block transposition distance between two permutations. In this chapter and in Chapter \ref{c5}, we treat two important topics on block transpositions, namely the distribution of block transposition distances and bounds on the block transposition diameter.

\section{Distribution of the block transposition distance}\label{s41}
The effective values of block transposition distances are currently known for $n\le 14$ while the block transposition diameter is also known for $n=15$. Table \ref{table:tab3} reports the exact value of $d(n)$ for $n\leq 15$ due to Eriksson et al.; see \cite{EE}. Table \ref{table:tab2} shows the distributions of the block transposition distances. The computation was carried out by Eriksson et al. for $n\leq 10$, see \cite{EE}, and by G\~{a}lvao and Diaz for $n=11,12,13$, see \cite{GD}, and by Gon\c{c}alves et al. for $n=14$; see \cite{GBH}.
It should be stressed that the above tables were obtained by computer. Interestingly, $d(17)=10$ can be directly proven using \cite[Theorem 8]{EH} together with Proposition \ref{sort eq} and hence without the use of a computer; see Section \ref{17}.
\begin{table}[!htbp]
\caption[Block transposition diameter]{Known values of the block transposition diameter of $\Sym_n$.}
\centering
\label{table:tab3}
\begin{tabular}{l*{16}{r}}
\toprule
 $n$
 & 1 & 2 & 3& 4 & 5 & 6 & 7 & 8 & 9 & 10 & 11 & 12 & 13 & 14 & 15\\
 \midrule
 diameter & 0 & 1 & 2 & 3 & 3 & 4 & 4 & 5 & 5 & 6 & 6 & 7 & 8 & 8 & 9\\
 \bottomrule
\end{tabular}
\end{table}
\begin{sidewaystable}[!htbp]
\caption[Block transposition distances]{The number of permutations $\pi$ in $Sym_n$ with $d(\pi)=k$, for $1\leq n\leq 14$.}
\centering
\label{table:tab2}
\begin{tabular}{r*{10}{r}}
\toprule
 $n \backslash k$
 & 0 & 1 & 2 & 3& 4 & 5 & 6 & 7 & 8\\
 \midrule
  1 & 1 & 0 & 0 & 0 & 0 & 0 & 0 & 0 & 0\\
  2 & 1 & 1 & 0 & 0 & 0 & 0  & 0 & 0 & 0\\
  3 & 1 & 4 & 1 & 0 & 0 & 0  &0 &0 & 0\\
  4 & 1 & 10 & 12 & 1 & 0 & 0 &0 &0 & 0\\
  5 & 1 &20&68&31&0&0 &0&0 & 0\\
  6 & 1 &35&259&380&45&0&0&0 & 0\\
  7 & 1 &56&770&2700&1513&0 & 0 & 0 & 0\\
  8 & 1 &84 & 1932 & 13467 &22000 & 2836 & 0 & 0 & 0 \\
  9 & 1 & 120& 4284 & 52512 &191636 & 114327 & 0 & 0 & 0\\
  10 & 1 & 165 & 8646 & 170907 & 1183457 & 2010571 & 255053 & 0 & 0\\
  11 & 1 & 220 & 16203 & 484440 & 5706464 & 21171518 & 12537954 & 0 &0\\
  12 & 1 & 286 & 28600 & 1231230 & 22822293 & 157499810 & 265819779 & 31599601 &0\\
  13 & 1 & 364 & 48048 & 2864719 & 78829491 & 910047453 & 3341572727 & 1893657570 & 427\\
  14 & 1 & 455 & 77441 & 6196333 & 241943403 & 4334283646 & 29432517384 & 47916472532 & 5246800005\\
 \bottomrule
\end{tabular}
\end{sidewaystable}
\section{Lower bounds on the block transposition diameter}

Bulteau, Fertin, and Rusu  proved in \cite{BF} that sorting a permutation by block transpositions is a $\textsf{NP}$-hard problem. Unfortunately, this has prevented the researchers from building a useful database for larger values of $n$. Therefore, the current investigation is aimed at determining lower bounds on the block transposition diameter for $n>15$. As a matter of fact the achievement of such an objective is still challenging.

For the rest of the chapter we deal with lower bounds, upper bounds being treated in Chapter \ref{c5}. In the next section, we present an interesting approach introduced recently by Doignon and Labarre \cite{DL} which also gives an alternative proof for the lower of Bafna and Pevzner appeared more than ten years earlier; see \cite{BP}. Actually, the best known lower bound is better than that one, and it is treated in the last section of this chapter.

\subsection{The Bafna-Pevzner-Labarre lower bound}\label{labarre}

A lower bound on the block transposition diameter appeared the first time in 1998, in the paper \cite{BP} of Bafna and Pevzner. These authors realized that good lower bounds should be close to $n/2$. They looked inside the possible variance from $n/2$ and were able to express it in terms of certain graphs, called cycle graphs. The concept of cycle graph is a very important one, and such as, it has been introduced several times, in slightly different but equivalent way. Here, we present the definition used by Bafna and Pevzner in \cite{BP}. Refer to \cite{FL} for an equivalent definition.
\begin{defi}
{\em{The \emph{cycle graph} $G=G(\pi)$ of a permutation $\pi$ on $[n]$ is the
directed graph on the vertex set $\{0,1,\ldots,n,n+1\}$ and $2n+2$ edges that are colored either black or gray as follows. Let $\pi_0=0$ and let $\pi_{n+1}=n+1$.
\begin{itemize}
\item For $1\leq i\leq n+1$, $(\pi_i,\pi_{i-1})$ is a \emph{black edge} .
\item For $0\leq i\leq n$, $(i, i + 1)$ is a \emph{gray edge}.
\end{itemize}}}
\end{defi}
Figure \ref{fig:graph} shows the cycle graph of a permutation.
\begin{figure}[!htbp]
\centering
\includegraphics[width=\textwidth]{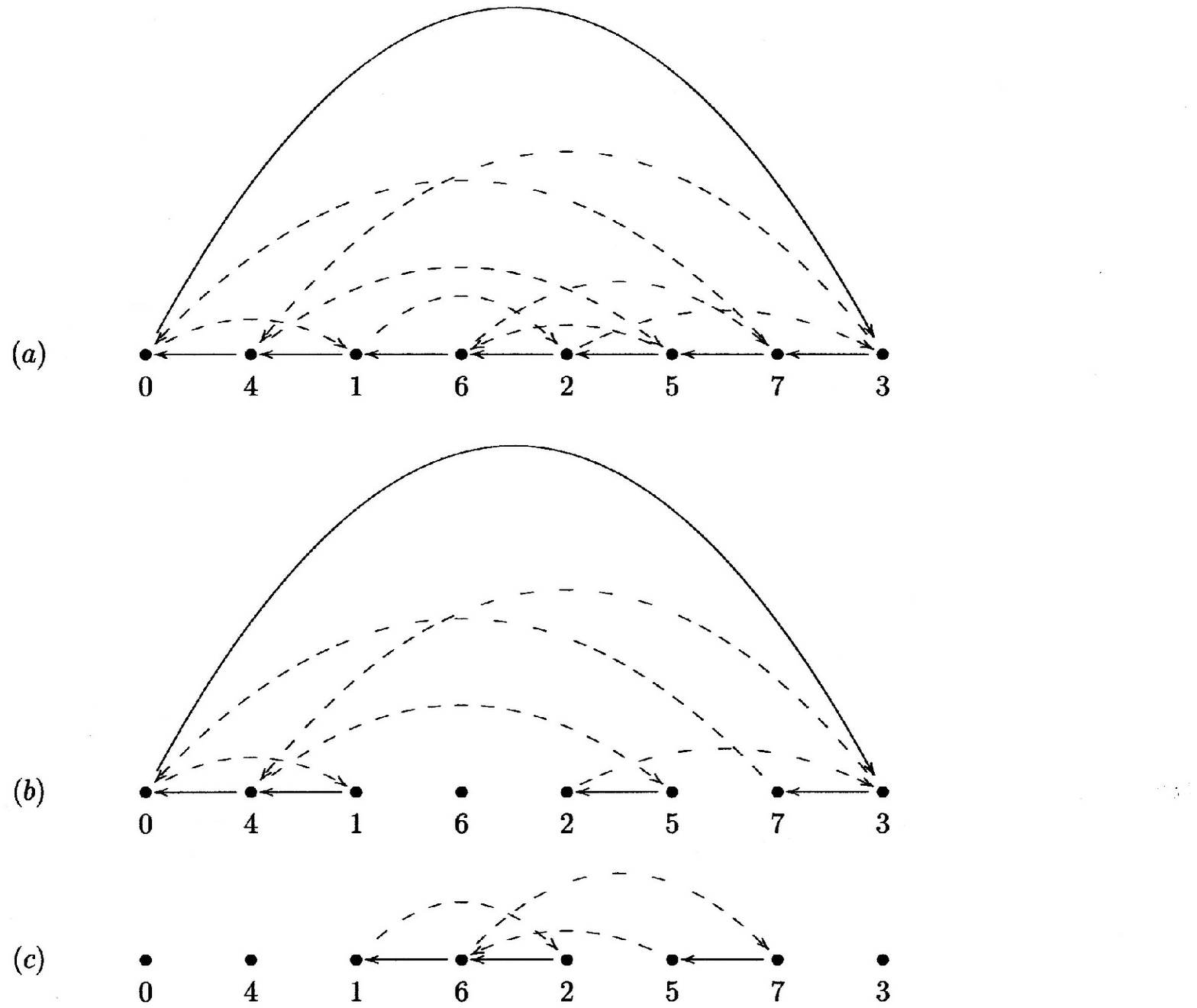}
\caption[Cycle graph and its decomposition]{(a) The cycle graph of $[4\,1\,6\,2\,5\,7\,3]$; (b),(c) its decomposition into two alternating cycles (reprinted from \cite{DL}).}
\label{fig:graph}
\end{figure}An \emph{alternating cycle} in $G$ is a cycle where the colors of the edges alternate. Since any vertex except $0$ and $n+1$ have one incoming edge and one outgoing edge of each color, $G$ is uniquely partitioned into alternating cycles. The \emph{length of an alternating cycle} of $G$ is the number of black edges that it contains, and a $k$-\emph{cycle} in $G$ is an alternating cycle of length $k$. When $k$ is odd, then $k$-cycle is called an \emph{odd cycle}, and $c_{odd}(G(\pi))$ is the number of odd cycles. In \cite[Theorem 2.4]{BP}, Bafna and Pevzner proves that
\begin{equation}\label{matnov7}
d(\pi)\geq \textstyle{\frac{1}{2}} \lfloor n+1-c_{odd}(G(\pi))\rfloor.
\end{equation}

In the context of block transpositions, one may ask about a possible, potential role of the natural map $\varphi$ which turns the permutation $\pi=[\pi_1 \pi_2 \cdots \pi_n]$ on $[n]$ to the permutation on $[n]^0$ represented by the ($n+1$)-cycle $(0,\pi_n,\pi_{n-1}, \ldots, \pi_1)$. Doignon and Labarre \cite{DL} worked in this direction by means of the map $\textsf{p}$ that takes a permutation $\pi$ on $[n]$ to the permutation $\alpha\circ \varphi(\pi)$ on $[n]^0$, $\alpha=[1\,2\cdots 0]=(0,1,\ldots, n)$; see Section \ref{s32}. Here, we give a survey of the results of these authors which had a significant impact on sorting by block transpositions.

The usefulness of $\textsf{p}$ is due to the following property, proved by Labarre in \cite[Lemma 3.1]{la}. For any $\pi,\nu \in \Sym_n$,
$$\textsf{p}(\nu\circ\pi)=\textsf{p}(\nu)\circ \textsf{p}(\pi)^{\nu},
$$where $\textsf{p}(\pi)^{\nu}$ denotes the conjugate of $\textsf{p}(\pi)$ by $[0\,\nu]$.

Obviously, $\textsf{p}$ is injective, and hence the image set $\rm{Im}(\textsf{p})$ of $\textsf{p}$ is a proper subset of $\Sym_n^0$, where $\Sym_n^0$ indicates the group of permutations on $[n]^0$; see Section \ref{s32}. In his investigation of $\textsf{p}$,
Labarre found an interesting relationship between factorizations within a restricted family of permutations on $[n]^0$ and factorizations into block transpositions in $\Sym_{n}$. As stated in the following theorem, this restricted family arises from $\rm{Im}(\textsf{p})\cap {\rm{Alt}}_{n+1}$, where ${\rm{Alt}}_{n+1}$ is the alternating group on $[n]^0$.
\begin{thm}
\label{labarretheorem3.2}
Let $\mathcal{C}$ be the union of the conjugacy classes of $\Sym_{n+1}$ which have nontrivial intersection with ${\rm{Alt}}_{n+1}$. Then, any factorization of $\pi\in \Sym_n$ into $k$ block transpositions yields a factorization of $\emph{\textsf{p}}(\pi)$ into $k$ factors from $\mathcal{C}$.
\end{thm}
A further useful property of $\textsf{p}$ pointed out by Labarre \cite[Lemma 4.3]{la} is that $\textsf{p}$ turns any block transposition on $[n]$ to a $3$-cycle on $[n]^0$:
\begin{equation}
\label{labarrelemma4.3} \textsf{p}(\sigma(i,j,k))=(i,k,j).
\end{equation}
Moreover, for any $\pi,\nu\in \Sym_n$,
$$\textsf{p}((\pi\circ\nu)^{-1})= (\textsf{p}(\nu)^{-1}\circ\textsf{p}(\pi^{-1}))^{\nu^{-1}};$$see \cite[Corollary 7.2]{la} and
$$\textsf{p}(\pi^w)=\textsf{p}({\pi^{-1}})^{[0\,w]\circ\alpha},$$where $w=[n\,n-1\cdots 1]$; see \cite[Lemma 7.3]{la}.

The relationship between $\textsf{p}$ and the toric equivalence was also worked out. The main result is \cite[Lemma 7.8]{la} and stated in the following proposition.
\begin{prop}
\label{labarrelem7.8} Let $\pi,\pi'$ be torically equivalent permutations on $[n]$. If $\pi'=\textsf{f}_r(\pi)$ then $\textsf{p}(\pi')=\textsf{p}(\pi)^{\alpha^r}.$
\end{prop}

Theorem \ref{labarretheorem3.2} together with (\ref{labarrelemma4.3}) provides a lower bound on $d(n)$ Theorem \ref{labarretheorem3.2} given by the length of a minimal factorization of $\textsf{p}(\pi)$ into $3$-cycles. In \cite{J}, Jerrum showed that such a length is $(n + 1-{\rm{c_{odd}}}(\Gamma(\textsf{p}(\pi)))/2$, where $\Gamma(\textsf{p}(\pi))$ is the permutation graph of $\textsf{p}(\pi)$ and ${\rm{c_{odd}}}(\Gamma(\textsf{p}(\pi)))$ is the number of odd alternative cycles in $\Gamma(\textsf{p}(\pi))$. Therefore,
$$d(\pi)\geq \textstyle{\frac{1}{2}} (n+1-c_{odd}(\Gamma(\textsf{p}(\pi)).$$

Labarre also pointed out that $\Gamma(\textsf{p}(\pi))$ and $G(\textsf{p}(\pi))$ have the same number of $k$-cycles for every $k$. Therefore, his lower bound coincides with (\ref{matnov7}) which we call the \emph{Bafna-Pevzner-Labarre lower bound}.

\subsection{The Elias-Hartman-Eriksson lower bound}\label{eliasHartman}

In \cite[Theorem 8]{EH} Elias and Hartman computed this distance of two classes of permutation for odd values of $n>15$.
\begin{prop}\label{hart}
Let $n>15$ be odd and let $i$ ranging over $0,1,\ldots,(k-2)/2$ with $k$ an even positive integer.

For $n=13+2k$,
$$[0\pi]=[0\,4\,3\,2\,1\,5\,13\,12\,11\,10\,9\,8\,7\,6\,\cdots 14+4i\,17+4i\,16+4i\,15+4i\cdots].$$

For $n=15+2k$,
$$[0\pi]=[0\,4\,3\,2\,1\,5\,15\,14\,13\,12\,11\,10\,9\,8\,7\,6\cdots 16+4i\,19+4i\,18+4i\,17+4i\cdots].$$Then $d([0\,\pi])= \dfrac{n+3}{2}$.
\end{prop}
In Section \ref{th23jan} we proved that $d(\pi)=d([0\,\pi])$. Therefore, Proposition \ref{hart} leads to the following lower bound lower bound, for odd values of $n>15$,
\begin{equation}\label{EHd}
d(n)\geq\dfrac{n+3}{2}.
\end{equation}

For even values of $n$, the lower bound is due Eriksson et al. \cite[Theorem 4.2]{EE} who computed the distance of the reverse permutation.
\begin{prop}\label{w}
For $n\geq 3$, $$d(w)=\left\lfloor\dfrac{n+2}{2}\right\rfloor.$$
\end{prop}
In the proof of Proposition \ref{w} Eriksson et al. give explicitly a sorting algorithm for $w$. Here, we show such an algorithm when $n$ is odd. For a proof of the optimality of this algorithm, see \cite{EE}.

\vspace{0.2cm}
\framebox[14.5cm]{
\begin{minipage}{14.5cm}
\begin{enumerate}
\vspace{0.5cm}
\item[]{\bf{Reverse permutation sorting algorithm on $[n]$}}
\item Let $r={n+1}/{2}$. Cut the block $|r\, r-1|$ and paste it at the beginning of $w$.\\Hence $w$ is turned in $w^{(1)}$, where
$$w^{(1)}=[r\,r-1\, n\cdots r+1\,r-2\cdots 1].$$
\item If $n=5$, then go to 4.; otherwise cut the block $|r+1\,r-2|$ and paste it\\between $r$ and $r-1$. Hence $w^{(1)}$ is turned in $w^{(2)}$, where
$$w^{(2)}=[r\,r+1\,r-2\,r-1\cdots r+2\,r-3\cdots 1].$$
\item If $n=7$, then go to 4.; otherwise, for every $k$ with $3\leq k\leq r-1$, cut the\\block $|r+k-1\,r-k|$ and paste it between $r+k-2$ and $r-k+1$.\\Hence $w^{(2)}$ is turned in $w^{(r-1)}$, where
$$w^{(r-1)}=[r\,r+1\cdots n-1\,1\,2\cdots r-1\,n].$$
\item Cut the block $|r\,r+1\cdots n-1|$ and paste between $r-1$ and $n$. Hence\\$w^{(r-1)}$ is turned in $\iota$.
\vspace{0.5cm}
\end{enumerate}
\end{minipage}
}
\vspace{0.5cm}

To sort $w$ if $n$ is even, apply the reverse sorting algorithm on $[n-1]$ to $w$. This turns $w$ in $\pi\in\Sym_n$ in $n/2$ steps, where
$$\pi=[n\,1\,2\cdots n-1].$$ Cutting $n$ and then pasting it at the end of the permutation turns $\pi$ into $\iota$.

Therefore, from Proposition \ref{w} and (\ref{EHd}), for $n>15$,
\begin{equation}\label{bestl}
d(n)\geq\left\lceil\dfrac{n+2}{2}\right\rceil.
\end{equation}So far nobody have achieved a better lower bound than (\ref{bestl}) which we call the \emph{Elias-Hartman-Eriksson lower bound}.

\chapter{Upper bound on the block transposition diameter}\label{c5}

Regarding upper bounds, the strongest one available in the literature is the \emph{Eriksson bound}, stated in $2001$; see \cite{EE}: For $n\geq 9$,
$$d(n)\leq \left\lfloor\frac{2n-2}{3}\right\rfloor.$$However, the proof of the Eriksson bound given in \cite{EE} is incomplete since it implicitly relies on the invariance of $d(\pi)$ when $\pi$ ranges over a toric class.
It should be noticed that this invariance principle has been claimed explicitly in a paper appeared in a widespread journal only recently; see \cite{CK}, although Hausen had already mentioned it and sketched a proof in his unpublished Ph.D. thesis; see \cite{Ha}. Elias and Hartman were not aware of Hausen's work and quoted the Eriksson bound in a weaker form which is independent of the invariance principle; see \cite{EH}, Proposition \ref{propEE}.

In our thesis, we show how the toric maps on $\Sym_n$ leave the distances invariant. Using the properties of these maps, we give an alternative proof for the above invariance principle which we state in Theorem \ref{torically} and Theorem \ref{principle}. We also revisit the proof of the key lemma in \cite{EE}; see Proposition \ref{propEE}, giving more technical details and filling some gaps. A major related result is the \emph{invariance principle} stated in the following two theorems, where $\pi'$ is torically equivalent to $\pi$ if $\textsf{f}_r(\pi)=\pi'$ for some nonnegative integer $r\leq n$.
\begin{thm}\label{torically}
If two permutations $\pi$ and $\pi'$ on $[n]$ are torically equivalent, then $d(\pi)=d(\pi')$.
\end{thm}
\begin{thm}\label{principle}
Let $\pi,\varpi,\nu,\mu$ be permutations on $[n]$ such that the toric map $\textsf{f}_r$ takes $\pi$ to $\varpi$ and $\textsf{f}_l$ takes $\nu$ to $\mu$ with $l=(\nu^{-1}\circ\pi)_r$. Then $d(\pi,\nu)=d(\varpi,\mu)$.
\end{thm}We give a proof of Theorem \ref{torically} and Theorem\ref{principle} in Section \ref{invariance}.

An important role in the investigations of bounds on $d(n)$ is played by the number of bonds of a permutation, where a \emph{bond} of a permutation $\pi\in\Sym_n$ consists of two consecutive integers $x,x+1$ in the sequence $0\,\pi_1 \cdots \pi_n\, n+1$. $[0\,\pi]$ has a bond if and only if $\pi$ has a bond. It is easily seen that any two permutations in the same toric class have the same number of bonds.
A bond of $\pi^\circ$ is any $2$-sequence $x\,\overline{x}$ of $\pi^\circ$, where $\overline{x}$ denotes the smallest nonnegative integer congruent to $x+1$ mod$(n+1)$. As bonds are rotation-invariant, their number is an invariant of $\pi_\circ^\circ$. The main result on bonds is the following.
\begin{prop}{\rm{(see \cite[Lemma 5.1]{EE})}}\label{wrong}
Let $\pi$ be any permutation on $[n]$ other than the reverse permutation. Then there are block transpositions $\sigma$ and $\tau$ such that either $\pi\circ\sigma\circ\tau,$ or $\sigma\circ\pi\circ\tau,$ or $\sigma\circ\tau\circ \pi$ has three bonds at least.
\end{prop}

However, what the authors actually proved in their paper \cite{EE} is the following proposition.
\begin{prop}\label{propEE}
Let $\pi$ be any permutation on $[n]$ other than the reverse permutation. Then $\pi_\circ^\circ$ contains a permutation $\overline\pi$ on $[n]^0$ with $\overline\pi_0=0$ having the following properties. There are block transpositions $\sigma$ and $\tau$ such that either $\overline\pi\circ[0\,\sigma]\circ[0\,\tau],$ or $[0\,\sigma]\circ\overline\pi\circ[0\,\tau],$ or $[0\,\sigma]\circ[0\,\tau]\circ\overline\pi$ has three bonds at least.
\end{prop}An important consequence of Proposition \ref{propEE} is the following result.
\begin{cor}
\label{kitty}
Let $\pi$ be any permutation on $[n]$ other than the reverse permutation. Then there exist a permutation $\pi'$ on $[n]$ torically equivalent to $\pi$ and block transpositions $\sigma$ and $\tau$ such that either $\pi'\circ\sigma\circ\tau,$ or $\sigma\circ\pi'\circ\tau,$ or $\sigma\circ\tau\circ\pi'$ has three bonds at least.
\end{cor}Assume that the first case of Proposition \ref{propEE} occurs. Observe that $\overline\pi\circ[0\,\sigma]\circ[0\,\tau]=[0\,\pi']$. Since $[0\,\pi']$ has as many
bonds as $\pi'$ does, Corollary \ref{kitty} holds. The authors showed in \cite{EE} that Proposition \ref{wrong} together with other arguments yields the following upper bound on the block transposition diameter.
\begin{thm}
\label{EEmainbis}{\rm{[Eriksson Bound]}}
For $n\geq 9$,
$$d(n)\leq \left\lfloor\frac{2n-2}{3}\right\rfloor.$$
\end{thm} Actually, as it was pointed out by Elias and Hartman in \cite{EH}, Proposition \ref{propEE} only ensures the weaker bound
$\left\lfloor\frac{2n}{3}\right\rfloor.$ Nevertheless, Proposition \ref{wrong} and Corollary \ref{kitty} appear rather similar, indeed they coincide in the toric class. This explains why the Eriksson upper bound still holds; see Section \ref{tiger}. In our thesis, we complete the proof of the Eriksson bound. We show indeed that the Eriksson bound follows from Proposition \ref{propEE} together with Theorem \ref{torically}. We also revisit the proof of Proposition \ref{propEE} giving more technical details and filling some gaps.

\section{The proofs of the main theorems}\label{invariance}

Now, we are in a position to prove Theorem \ref{torically}. Take two torically equivalent permutations $\pi$ and $\pi'$ on $[n]$.
Let $d(\pi)=k$ and let $\pi=\sigma_1\circ\cdots \circ \sigma_k$ with $\sigma_1,\ldots,\sigma_k\in S_n$. Then
$$[0\,\pi]=[0\,\sigma_1] \circ \cdots \circ [0\,\sigma_k].$$By (\ref{toric_zero}), there exists an integer $r$ with $0\le r \le n$ such that
\begin{equation}
\label{eq323jan}
[0\,\pi']=\alpha^{-\pi_r}\circ [0\,\sigma_1] \circ \cdots \circ [0\,\sigma_k] \circ \alpha^r.
\end{equation}Lemma \ref{lemc19ott2013} applied to $[0\,\sigma_k]$ allows us to shift $\alpha^r$ to the left in (\ref{eq323jan}), in the sense that $[0\,\sigma_k]\circ\alpha^{r}$ is replaced by $\alpha^t\circ[0\,\rho_k]$ with $t=-(\sigma_k)_r$ and $\rho_k\in S_n$. Repeating this $k$ times yields
$$[0\,\pi']=\alpha^s\circ[0\,\rho_1]\circ \cdots \circ [0\,\rho_k],$$ for some integer $0\le s \le n$. Actually, $s$ must be $0$ as $\alpha^s$ can fix $0$ only for $s=0$. Then $\pi'=\rho_1\circ\cdots\circ\rho_k$, and there exist $r_1,\cdots,r_k$ integers with $0\le r_i \le n$ such that
$$\pi'=\textsf{f}_{r_1}(\sigma_1)\circ \textsf{f}_{r_2}(\sigma_2)\circ\cdots\circ\textsf{f}_{r_k}(\sigma_k).$$
Therefore, $d(\pi')\le k=d(\pi)$. By inverting the roles of $\pi$ and $\pi'$, we also obtain $d(\pi)\le d(\pi')$. Hence the claim in Theorem \ref{torically} follows.

To prove Theorem \ref{principle}, it suffices to show that $d(\nu^{-1}\circ\pi)=d(\mu^{-1}\circ\varpi)$, by the left-invariance of the block transposition distance; see Proposition \ref{distance}. Let $d(\nu^{-1}\circ\pi)=h$ and let $\sigma_1,\ldots,\sigma_h\in S_n$ such that
\begin{equation}\label{a1}
[0\,\nu^{-1}]\circ[0\,\pi]=[0\,\sigma_1]\circ\cdots \circ [0\,\sigma_h].
\end{equation}Since $\textsf{f}_r$ takes $\pi$ to $\varpi$ and $\textsf{f}_l$ $\nu$ to $\mu$, we obtain $\alpha^{s}\circ[0\,\varpi]\circ\alpha^{-r}=[0\,\pi]$ and $\alpha^{l}\circ[0\,\mu^{-1}]\circ\alpha^{-t}=[0\,\nu^{-1}],$ where $r$ is an integer with $0\leq r \leq n$, $t=\nu_l$, and $s=\pi_r$, by (\ref{eq2oct9}). Hence
\begin{equation}\label{a2}
[0\,\nu^{-1}]\circ[0\,\pi]=\alpha^{l}\circ[0\,\mu^{-1}]\circ\alpha^{-t}\circ\alpha^{s}\circ[0\,\varpi]\circ\alpha^{-r}.
\end{equation} Since $\nu_l=\pi_r$, then $[0\,\mu^{-1}]\circ[0\,\varpi]=\alpha^{-l}\circ[0\,\sigma_1]\circ\cdots \circ[0\,\sigma_h]\circ\alpha^{r}$ follows from (\ref{a1}) and (\ref{a2}). By Lemma \ref{lemc19ott2013}, this may be reduced to $[0\,\mu^{-1}]\circ[0\,\varpi]=\alpha^q\circ[0\,\sigma'_1]\circ\cdots\circ[0\,\sigma'_h]$ for some integer $q$ with $0\leq q\leq n$ and $\sigma'_1,\ldots,\sigma'_h\in S_n$. As we have seen in the proof of Theorem \ref{torically}, $q$ must be $0$, and $d(\mu^{-1}\circ\varpi)=d(\nu^{-1}\circ\pi)$. Therefore, the claim  in Theorem \ref{principle} follows.

The proof of Proposition \ref{propEE} is constructive and involves several cases. In the following section, we prove a result on $2$-moves claimed without a proof in \cite{EE} and useful to our aim.

\section{Criteria for the existence of a 2-move}

For every permutation $\pi$, our algorithm will provide two block transpositions $\sigma$ and $\tau$ together with a permutation $\overline\pi\in \pi^\circ_\circ$ so that either $\overline\pi\circ[0\,\sigma]\circ[0\,\tau],$ or $[0\,\sigma]\circ\overline\pi\circ[0\,\tau],$ or $[0\,\sigma]\circ[0\,\tau]\circ\overline\pi$ has three bonds at least. For the seek of the proof, $\pi$ is assumed to be bondless, otherwise all permutations in its toric class has a bond, and two more bonds by two block transpositions can be found easily.

A \emph{k-move (to the right)} of $\overline\pi\in{\rm{Sym_n^0}}$ is a block transposition $\overline\sigma$ on $[n]^0$ such that $\overline\pi\circ\overline\sigma$ has (at least) $k$ more bonds than $\overline\pi$. A block transposition $\overline\sigma$ is a \emph{k-move to the left} of $\overline\pi$ if $\overline\sigma\circ\overline\pi$ has (at least) $k$ more bonds than $\overline\pi$.
\begin{crt}\label{2-move to the right}
A $2$-move of $\overline\pi\in{\rm{Sym_n^0}}$ exists if one of the following holds:
\begin{itemize}
\item[\rm(i)] $\overline\pi=[\cdots x\cdots y\,\overline x\cdots \overline y\cdots]$;
\item[\rm(ii)] $\overline\pi=[\cdots x\cdots \underline x\,\overline x\cdots].$
\end{itemize}
\end{crt}
\begin{proof}
Each of the following block transpositions:
$$x|\cdots y|\overline x\cdots |\overline y,\quad|x|\cdots \underline x|\overline x$$gives two new bonds for (i) and (ii), respectively.\qquad\end{proof}

In a permutation an ordered triple of values $x\cdots y\cdots z$ is \emph{positively oriented} if either $x<y<z$, or $y<z<x$, or $z<x<y$ occurs.
\begin{crt}\label{2-move to the left}
Let $\overline\pi$ be a permutation on $[n]^0$. A $2$-move to the left of $\overline\pi$ occurs if one of the following holds.
\begin{itemize}
\item[\rm(i)] $\overline\pi=[\cdots x\,y\dots z\,\overline x\cdots]$ and $x,y,z$ are positively oriented,
\item[\rm(ii)] $\overline\pi=[\cdots x\,y\,\overline x\cdots].$
\end{itemize}In particular, the following block transpositions on $[n]^0$ create a $2$-move to the left of $\overline\pi$.
\begin{itemize}
\item[]For {\rm{(i)}},
\item[\rm(I)] $\overline\sigma(x,x+z-y+1,z)$ if $x<y<z$;
\item[\rm(II)] $\overline\sigma(y-1,y-1+x-z,x)$ if $y<z<x$;
\item[\rm(III)] $\overline\sigma(z,z+y-1-x,y-1)$ if $z<x<y$.
\item[]For \rm{(ii)},
\item[\rm(IV)] $\overline\sigma(x,x+1,y)$ if $x<y$;
\item[(V)] $\overline\sigma(y-1,x-1,x)$ if $y<x$.
\end{itemize}
\end{crt}
\begin{proof}
Let $\overline\sigma=\overline\sigma(a,b,c)$ for any $a,\,b,\,c$ with $-1\leq a<b<c\leq n$. Then
$$\overline\sigma_t =\left\{\begin{array}{ll}
t, &  0\leq t\leq a\quad c+1\leq t\leq n,\\
t+ b- a, & a+1\leq t\leq c- b+ a,\\
t+ b- c, & c-b+a+1\leq t\leq c
\end{array}\right.$$follows from (\ref{feb9}). In particular, by (\ref{cuppoints}),
\begin{align}\label{tras}
\overline\sigma_a&= a;&\overline\sigma_{a+1}&=b+1;&\overline\sigma_{a+c-b}&=c;\\
\overline\sigma_{a+c-b+1}&=a+1;&\overline\sigma_{c}&=b;&\overline\sigma_{c+1}&=c+1.\notag
\end{align}
Our purpose is to determinate $a,b,c$ so that $\overline\sigma$ is a $2$-move to the left of $\pi$. For this, $\overline\sigma$ must satisfy the following relations:
\begin{equation}\label{2-move}
\overline\sigma_y=\overline\sigma_x+1;\quad\overline\sigma_{\overline x}=\overline\sigma_z+1.
\end{equation}

If the hypothesis in case (I) is satisfied, take $x$ for $a$. By (\ref{tras}), we obtain both $\overline\sigma_x=a$ and $\overline\sigma_{\overline x}=b+1$. This together with (\ref{2-move}) gives $\overline\sigma_y=a+1$ and $\overline\sigma_z=b$. By (\ref{tras}), we get
$$\left\{\begin{array}{l}
a=x\\
c=z\\
b=a+c-y+1.
\end{array}\right.$$Since $\pi$ is bondless, here $x<y-1$ may be assumed. This together with $y<z$ gives $x<x+z-y+1<z$, whence the statement in case (I) follows. To deal with case (IV) it is enough to use the same argument after switching $z$ and $y$.

If the hypothesis in case (II) is satisfied, take $\overline\sigma_x=b$ and $\overline\sigma_z=c$. This choice together with (\ref{2-move}) gives $\overline\sigma_y=b +1$ and $\overline\sigma_{\overline x}=c+1$. By (\ref{tras}), we have
$$\left\{\begin{array}{l}
a=y-1\\
c=x\\
b=a+c-z.\\
\end{array}\right.$$Since $y-1<y-1+x-z<x$, the statement in case (II) follows. Case (V) may be settled with the same argument after switching $z$ and $y$.

In case (III), let $\overline\sigma_x=c$ and $\overline\sigma_z =a$. This together with (\ref{2-move}) gives $\overline\sigma_y=c +1$ and $\overline\sigma_{\overline x}=a+1$. By (\ref{tras}), we obtain
$$\left\{\begin{array}{l}
a=z\\
c=y-1\\
b=a+c-x.\\
\end{array}\right.$$Since $z<z+y-1-x<y-1$, the statement holds.
\end{proof}

\begin{rem}\label{fix}
{\emph{Note that the block permutation on $[n]^0$ appearing in case (I), (III), and (IV) of Criterion \ref{2-move to the left} fixes $0$. In case (II) and (V), this occurs if and only if $y\neq 0$.}}
\end{rem}

\section{Reducible case}

For the proof of Proposition \ref{propEE}, we begin by constructing the required moves for reducible permutations. A permutation $\pi$ on $[n]$ is \emph{reducible} if for some $k$ with $0<k<n$ the segment $0\cdots \pi_k$ contains all values $0,\ldots,k$ while the segment $\pi_k\cdots n$ contains all values $k,\ldots, n$. In particular, $\pi_k=k$ is required. It is crucial to note that a reducible permutation collapses into a smaller permutation by erasing the segment $\pi_k\cdots \pi_n$. If a reverse permutation is produced in this way, we proceed by contracting the segment $0\cdots \pi_k$ to $0$. There may be that both contractions produce a reverse permutation. This only occurs when
\begin{equation}\label{red}
[0\,\pi]=[0\,k-1\, k-2 \cdots 1 \, k \, n \, n-1 \cdots k+1].
\end{equation}
After carrying out the $1$-move $k-1| k-2 \cdots 1 \, k \,n| n-1 \cdots  k+1|,$
Criterion \ref{2-move to the left} (III) applies to $k-1 \,n-1 \cdots 1 \, k$ whence Proposition \ref{propEE} follows in this case.

To investigate the other cases we show that a permutation of the form (\ref{red}) occurs after a finite number of steps. For this purpose, let $[0\,\pi]^0=[0\,\pi]$ and let $[0\,\pi]^l=[0\,\pi_1^l \cdots \pi_{n_l}^l]$ for any integer nonnegative integer $l$. First, we prove that reducing $[0\,\pi]^l$ diminishes the length of $[0\,\pi]^l$ by at least three. Observe that a reduced permutation $[0\,\pi]^l$ is bondless since $[0\,\pi]$ is bondless. As $k_l<n_l$, erasing the segment $\pi^l_{{n_l}-1}\,\pi^l_{n_l}$ gives $\pi^l_{n_l-1}=k_l$ and $\pi^l_{n_l}=k_l+1.$ If we contract $0\,\pi_1^l\,\pi_2^l$ into $0$, we obtain $k_l=\pi^l_2=2$ and $\pi^l_1=1.$ Actually, none of these possibilities can occur as $[0\,\pi]^l$ is bondless. Nevertheless, we are able to reduce the length of $[0\,\pi]^l$ by exactly three. For instance, let $[0\,\pi]^l=[0\,2\,1\,3\cdots]$. After contracting $l\leq\lfloor n/3\rfloor$ times the following four cases:
\begin{itemize}
\item[\rm(i)] $[0\,\pi]^l=[0]$;
\item[\rm(ii)] $[0\,\pi]^l=[0\,1]$;
\item[\rm(iii)] $[0\,\pi]^l=[0\,2\,1]$;
\item[\rm(iv)] $[0\,\pi]^l=[0\,1\,2]$
\end{itemize}
hold. In case (i) and (ii), we have $k_{l-1}=1$ and $k_{l-1}=2$, respectively. As $[0\,\pi]^l$ is bondless, only case (iii) occurs, a contradiction, since collapsing any segment of $[0\,\pi]^{l-1}$ does not produce a reverse permutation. Therefore, the assertion follows.

\section{Irreducible case}

It remains to prove Proposition \ref{propEE} when $\pi$ is bondless and irreducible. We may also assume that no permutation $\overline\pi\in\pi_\circ^\circ$ satisfies either Criterion \ref{2-move to the right} or Criterion \ref{2-move to the left} with a block transposition fixing $0$. Otherwise, getting a further $1$-move is trivial.

Up to toric equivalence, choose $\overline\pi$ fulfilling the minimality condition on $0\cdots 1$, that is, the shortest sequence $m\cdots \overline{m}$ in $\overline\pi$ occurs for $m=0$. To prove that such a permutation exists, we start with $[0\,\pi]=[0\cdots \pi_u\cdots \pi_v \cdots n]$, where $\pi_u=m, \pi_v=\overline{m}$. By (\ref{circ_pw}), $\pi$ is torically equivalent to $\pi'$ on $[n]$, where $\pi'$ is defined by $\pi'_x=\pi'_{x+u}-m$ for every integer $x$ with $1\le x \le n$, and the indices are taken mod$(n+1)$. Let $\overline \pi=[0\,\pi']$. Then $\overline\pi_0=0$ and $\overline\pi_{v-u}=\pi_v-m=1$.

We begin by observing that the minimality condition on $0\cdots 1$ always rules out the case $\overline\pi=[\cdots x\cdots \overline x\cdots 1\cdots]$. The absence of bonds rules out the extremal case $\overline\pi=[0\,1\cdots]$, while the absence of a $2$-move fixing $0$ makes it possible to avoid $\overline\pi=[0\,x_1\,1\cdots]$ by applying Criterion \ref{2-move to the left} (IV) to $0\,x_1\,1$. Therefore, we may write $\overline\pi$ in the form $[0\,x_1\cdots x_l\,1 \cdots]$ with $l\geq 2$. Note that $x_1>x_l$, otherwise Criterion \ref{2-move to the left} (I) applies to $0\,x_1\cdots x_l\, 1$, whence $\overline x_1$ is on the right of $1$ when $x_1\neq n$. Now, one of the $1$-move listed below
\begin{itemize}
\item[\rm(i)] $0|x_1\,x_2\cdots x_l|1\cdots|\overline x_1$, for $x_1\neq n$;
\item[\rm(ii)] $0|x_1\cdots x_l|1\cdots x_n|$, for $x_1=n,\,x_n\neq 1$;
\item[\rm(iii)] $0|x_1\cdots x_l|1|$, for $x_1=n,\,x_n= 1$
\end{itemize}
turns $\overline\pi$ in one of the following forms:
\begin{itemize}
\item[\rm(I)] $[0\cdots x_1\, x_2\cdots x_l\,\overline x_1\cdots]$, for $x_1\neq n,\,l\neq 2$;
\item[\rm(II)] $[0\cdots x_1\,x_2\,\overline x_1\cdots]$, for $x_1\neq n,\,l=2$;
\item[\rm(III)] $[0\cdots x_1\,x_2\cdots x_l]$, for $x_1=n,\,x_n\neq 1,\,l\neq 2$;
\item[\rm(IV)] $[0\cdots x_1\,x_2]$, for $x_1=n,\,x_n\neq 1,\,l=2$;
\item[\rm(V)] $[0\,1\,x_1\,x_2\cdots x_l]$, for $x_1=n,\,x_n= 1,\,l\neq2$;
\item[\rm(VI)] $[0\,1\,x_1\, x_2]$, for $x_1=n,\,x_n= 1,\,l=2$.
\end{itemize}Unless $x_1>x_2>x_l$, Proposition \ref{propEE} holds. In fact, there exists a permutation in $\pi_\circ^\circ$ that satisfies one of the hypotheses of Criterion \ref{2-move to the left}. More precisely, we may apply either Criterion \ref{2-move to the left} (II) or Criterion \ref{2-move to the left} (III) in case (I) and Criterion \ref{2-move to the left} (II) in case (III). For $l=2$, the statement follows from Criterion \ref{2-move to the left} (V). Some block transpositions in Criterion \ref{2-move to the left} (II) and \ref{2-move to the left} (V) may not fix $0$. This cannot actually occur, since we use block transpositions on $[n]^0$ of the form $[0\,\sigma(i,j,k)]$ in all cases; see Remark \ref{fix}.

Therefore, we may assume $\overline\pi=[0\,x_1\,x_2\cdots x_l\,1\cdots]$ with $x_1>x_2>x_l$. Two cases are treated separately according as $x_2=x_1-1$ or $x_2<x_1-1$.

\subsection{Case \texorpdfstring{$x_2= x_1-1$}{Lg}}\label{casee}

If $\overline\pi=[0\,x_1\,x_2\cdots x\cdots1\cdots \overline x\cdots]$ occurs for some $x$, then the $1$-move $$x_1|x_2\cdots x|\cdots 1\cdots|\overline x$$ turns $\overline\pi$ into $[0\,x_1\cdots 1\cdots x_2\cdots]$. After that, the existence of a $2$-move is ensured by Criterion \ref{2-move to the right}, whence Proposition \ref{propEE} holds.

Therefore, we may assume that if $x$ ranges over $x_3,x_4,\ldots,x_i,\ldots x_l$, then $\overline x$ is on the left of $x$. At each stage two cases arise depending upon whether $\overline x_i=x_{i-1}$ or $x_i=n$, where $l\neq i\neq 2$.

\subsubsection{Case $\overline x_i=x_{i-1}$}\label{imp}

Note that $x_1\,x_2\cdots x_l$ is a reverse consecutive sequence, and $n$ is on the right of $1$. As $\pi$ is bondless, two cases arise according as either $1<x_n<x_l$ or $x_1<x_n<n$.

In the former case, carrying out the $1$-move $|\overline {x}_l\,x_l\,1\cdots n|\cdots x_n|$, the resulting permutation is $[\cdots x_n\,\overline {x}_l\,x_l\,1\cdots]$. In the latter case, use the $1$-move $|x_1\cdots|1\cdots n|$ to obtain $[0\cdots n\,x_1\cdots x_n]$. In both cases, Proposition \ref{propEE} follows from Criterion \ref{2-move to the left} (II), applied to a block transposition that fixes $0$; see Remark \ref{fix}. More precisely, let $\overline\pi\circ[0\,\sigma]$ be the permutation obtained in both cases, and let $\overline\pi\circ[0\,\sigma]\circ\alpha^r$ with $1\leq r\leq n$ be the permutation that satisfies the hypothesis of Criterion \ref{2-move to the left} (II). Then $[0\,\tau]\circ\overline\pi\circ[0\,\sigma]\circ\alpha^r$ has three bonds at least. By Lemma \ref{lemc19ott2013}, there exists an integer $s$ with $1\leq s\leq n$ and a block transposition $\sigma'$ on $[n]$ such that $$[0\,\tau]\circ\overline\pi\circ[0\,\sigma]\circ\alpha^r=[0\,\tau]\circ\overline\pi\circ\alpha^s\circ[0\,\sigma'].$$ Since $\overline\pi\circ\alpha^s\in \pi_\circ^\circ$, Proposition \ref{propEE} follows.

\subsubsection{Case $x_1=n$}\label{matnov18}

In this case there exists $k$ with $2\leq k\leq n-2$ such that
$$\overline\pi=[0\,n\,n-1\cdots n-(k-2)\,n-(k-1)\,1\cdots n-k \cdots].$$Since $\pi$ is not the reverse permutation, $2$ is on the right of $1$ whence $2\leq k\leq n-2$. So two cases arise depending on the position of $2$ with respect to $n-k$.

If $2$ is on the left of $n-k$, then the $1$-move $|n-(k-1)\,1\,y\cdots |2\cdots n-k|$ turns $\overline\pi$ into $[\cdots n-(k-2)\,2\cdots 1\,y\cdots]$. As all integers $x$ with $n-(k-2)\leq x\leq n$ are in $0\cdots1$, this yields $y<n-(k-2)$. So Criterion \ref{2-move to the left} (I) applies to a permutation in the circular class of $[\cdots n-(k-2)\,2\cdots 1\,y\cdots]$, and the claim follows as in Section \ref{matnov18}.

If $2$ is on the right of $n-k$, consider the $1$-move $$|n-(k-2)\, n-(k-1)\, 1|\cdots n-k\cdots z|2.$$ If $z=n-k$, then the above transposition takes our permutation to $$[\cdots n-k\,n-(k-2)\,n-(k-1)\cdots].$$ The existence of a $2$-move is ensured by Criterion \ref{2-move to the left} (IV). Otherwise $z<n-k$, and Criterion \ref{2-move to the left} (II) applies to a permutation in the circular class of $$[\cdots z\,n-(k-2)\,n-(k-1)\,1\cdots]$$and a block transposition that fixes $0$; see Remark \ref{fix}. Hence the claim follows as in Section \ref{matnov18}.

\subsubsection{Case \texorpdfstring{$x_i=n$}{Lg}}

As we have seen before, $\overline x$ is on the left of $x$ for every $x$ in $0\cdots 1$. Therefore, when $x_i=n$, $x_{j-1}=\overline{x}_j$ for $1\leq j\leq i-1$, but it does not necessarily holds for all $j$ with $i\leq j\leq l$. However, there exists $h$ with $1\leq h \leq n-1-x_1$ such that for each $x\neq 0$ on the left of $x_l$ either $\overline{x}_l\leq x \leq x_1$ or $n-(h-1)\leq x \leq n$ occurs. Both these subsequences are decreasing by our minimality condition on $0\cdots 1$.

First, suppose the existence of $k$ with $3 \leq k\leq h$ so that
$$\overline\pi=[0\,x_1\,\underline {x_1}\cdots x_t\,\underline {x_t}\,n\,n-1\,n-2\cdots n-(k-3)\,n-(k-2)\,n-(k-1)\,\underline{\underline {x_t}}\cdots 1\cdots],$$
where $\overline {x}_l\leq \underline {x_t} \leq \underline {x_1}$ and $\underline{\underline x}$ stands for $\underline y$ with $y=\underline x$. Now, one of the following $1$-move:
$$\begin{array}{ll}
x_t|\underline {x_t}\,n\cdots n-(k-3)|n-(k-2)\,n-(k-1)\,\underline{\underline{x_t}}|,& k>3;\\
x_t|\underline {x_t}\,n|n-1\,n-2\,\underline{\underline{x_t}}|,& k=3
\end{array}$$turns $\overline\pi$ into $[\cdots x_t\,n-(k-2)\,n-(k-1)\,\underline{\underline{x_t}}\cdots]$. Therefore, Criterion \ref{2-move to the left} (II) applies to a permutation in the circular class of $[\cdots x_t\,n-(k-2)\,n-(k-1)\,\underline{\underline{x_t}}\cdots]$ and a block transposition fixing $0$; see Remark \ref{fix}. Hence the assertion follows as in Section \ref{matnov18}.

Note that case $\overline\pi=[\cdots\underline {x_t}\,n\,n-1\,\underline{\underline {x_t}}\cdots 1\cdots]$ does not occur. In fact, the existence of a $2$-move of a permutation in the toric class $\pi_\circ^\circ$ and a block transposition fixing $0$ is ensured by Criterion \ref{2-move to the left} (II) and Remark \ref{fix}. Therefore, we may assume that $$\overline\pi=[\cdots\underline {x_t}\,n\,\underline{\underline {x_t}}\cdots].$$ Now, a $2$-move of a permutation in the toric class of $\pi^\circ_\circ$ with a block transposition fixing $0$ is ensured by Criterion \ref{2-move to the left} (IV), a contradiction.

\subsection{Case \texorpdfstring{$x_2< x_1-1$}{Lg}}

If $\underline {x_1}$ is on the right of $1$, then there exists a $2$-move of $[0\,x_1\cdots 1 \cdots \underline {x_1}\cdots]$ by Criterion \ref{2-move to the right}, a contradiction. Therefore, $\underline {x_1}$ is on the left of $1$. We look for the biggest integer $k$ with $2\leq k\leq l-1$ such that
\begin{equation}\label{diseq}
x_1-(k-1)>x_2-(k-2)>\cdots>x_i-(k-i)>\cdots> x_{k-1}-1>x_k>x_l
\end{equation}holds. Note that (\ref{diseq}) holds for $k=2$ by $x_1-1>x_2>x_l$. Suppose that $\overline {x}_k$ is on the left of $1$ with $x_i=\overline {x}_k$ for some $i$. Then $1\leq i\leq k-1$ and $x_i-1\geq x_i-(k-i)>x_k$, a contradiction. Therefore, $\overline {x}_k$ must be on the right of $1$. The $1$-move $0|x_1\cdots x_l|1\cdots |\overline {x}_k$ turns $\overline\pi$ into $[0\,1\cdots x_k\,x_{k+1}\cdots x_l\,\overline {x}_k\cdots],$ and the following three possibilities arise:
\begin{itemize}
\item[\rm(i)] $x_l<x_k<x_{k+1}$;
\item[\rm(ii)] $x_{k+1}<x_l<x_k$;
\item[\rm(iii)] $x_l<x_{k+1}<x_k$.
\end{itemize}Proposition \ref{propEE} follows from Criterion \ref{2-move to the left} (III) in case (i) and from Criterion \ref{2-move to the left} (II) in case (ii), applied to a block transposition fixing $0$; see Remark \ref{fix}.

In the remaining case, adding $1$ to each side in (\ref{diseq}) gives $x_{i-1}-(k-i)> x_{k-1},$ where $1<i<k$. If $\overline {x}_{k-1}$ is on the left of $1$ and $x_{i-1}=\overline {x}_{k-1}$ for some $i>1$, then $x_{i-1}-1> x_{i-1}-(k-i)> x_{k-1}$, a contradiction. If $\overline {x}_{k-1}$ is on the left of $1$ and $x_1=\overline {x}_{k-1}$, subtracting $1$ from each side in (\ref{diseq}) gives
$$x_1-k>x_2-(k-1)>\cdots>x_i-(k+1-i)>\cdots> x_{k-1}-2>x_k-1.$$Here, $x_k-1>x_{k+1}$ cannot actually occur by our maximality condition on $k$. Therefore $x_{k+1}=x_k-1$. The $1$-move $|x_{k-1}\,x_k|\underline{x_k}\cdots1\cdots|\overline {x}_k$ turns $\overline\pi$ into $[0\,\overline {x}_{k-1}\cdots1\cdots x_{k-1}]$, and Proposition \ref{propEE} follows from Criterion \ref{2-move to the right}. Here, we consider $\overline {x}_{k-1}$ to be on the right of $1$. Adding $k-i$ to each side in (\ref{diseq}) gives $x_1-(i-1)> x_i$. Assume $\underline{x_1}$ is on the left of $x_k$. Since $x_2\neq x_1-1$, then  $x_i=x_1-1$ for some $i>2$ and $x_1-1>x_1-(i-1)> x_i$, a contradiction. Therefore, we may assume $\underline{x_1}$ is in on the right of $x_k$. Since $\overline {x}_{k-1}$ is on the right of $1$, the $1$-move
$$0|x_1\cdots x_k|\underline{x_k}\cdots\underline{x_1}|$$ turns $\overline\pi$ into $[0\,\underline{x_k}\cdots x_{k-1}\,x_k\cdots1\cdots\overline {x}_{k-1}\cdots]$. Now, there exists a $2$-move, namely $$\underline{x_k}|\cdots x_{k-1}|x_k\cdots1\cdots|\overline {x}_{k-1}.$$ Therefore, Proposition \ref{propEE} in case (iii) follows. This concludes the proof of Proposition \ref{propEE}.

\section{The proof of the Eriksson bound}\label{tiger}

Let $\pi$ be a permutation on $[n]$ with $n\geq 4$. We apply Corollary \ref{kitty} after dismissing the case where $\pi$ is the reverse permutation by virtue of Proposition \ref{w}. Assume that the first case occurs in Corollary \ref{kitty}, the other two cases may be investigated in the same way. By Proposition \ref{distance} and Corollary \ref{prop d},
\begin{equation}\label{fi}
d(\rho)\leq d(\rho\circ\sigma\circ\tau)+d(\tau^{-1})+d(\sigma^{-1}).
\end{equation} As the distance of a block transposition is $1$, the right-hand side in (\ref{fi}) is equal to $d(\rho\circ\sigma\circ\tau)+2$.
Collapsing bonds into a single symbol has the effect of collapsing $\rho\circ\sigma\circ\tau$ into a permutation on $[{n-3}]$. Then $d(\rho\circ\sigma\circ\tau)+2\leq d(n-3)+2$. By Theorem \ref{torically}, we obtain $d(\pi)=d(\rho)=\leq d(n-3)+2$, and then $$d(n)\leq d(n-3)+2.$$

Now, the argument in the proof of \cite[Theorem 4.2]{EE} may be used to finish the proof of the Eriksson bound. This also shows that the Eriksson bound holds only by virtue of Theorem \ref{torically}.
\section{A new value of the block transposition diameter}\label{17}

As we have mentioned in Section \ref{s41}, the exactly value of the block transposition diameter $d(n)$ is known only for $n\leq 15$. In this final section, we show that the exact value of $d(17)$ can be determined with a computer free argument using only the Eriksson bound together with the the Elias-Hartman-Eriksson lower bound; see Section \ref{eliasHartman}.
\begin{thm}
The block transposition diameter is $10$, for $n=17$.
\end{thm}
\begin{proof}
By Theorem \ref{EEmainbis}, $d(17)\leq 10$. On the other hand, Elias and Hartman exhibited a permutation $[0\,\pi]$ on $[17]^0$ with $d([0\,\pi])=10$, namely
$$[0\,\pi]=[0\,4\,3\,2\,1\,5\,13\,12\,11\,10\,9\,8\,7\,6\,14\,17\,16\,15].$$ Since we have proved that $d(\pi)=d([0\,\pi])$ in Proposition \ref{th23jan}, thus $d(17)=10$.
\end{proof}

\chapter{Cayley graph on symmetric groups with generating block transposition sets}\label{c6}

As a matter of fact, all our general results in this chapter hold for $n\geq 5$ while some of them are not valid for $n=4$.
For this reason, the case $n=4$ is treated in Section \ref{n=4}. Furthermore, since some of the proofs are carried out by induction on $n$, we must be sure that our results are valid for the smallest possible values of $n$ which are $5$ and $6$ in the present context. Bearing this in mind, we have thoroughly worked out these cases by a computer aided exhaustive search and present the relative results in Section \ref{n=5}, \ref{n=6}.

Since $S_n$ is an inverse closed generator set of $\Sym_n$ which does not contain $\iota$, by Corollary \ref{matnov2} (ii), (the left-invariant) Cayley graph $\Cay$ is an undirected simple graph, where $\{\pi,\rho\}$ is an edge if and only if $\rho=\sigma(i,j,k)\circ \pi$, for some $\sigma(i,j,k)\in S_n$; see Section \ref{cay}. Also, the vertices of $\Cay$ adjacent to $\iota$ are exactly the block transpositions.

\section{Automorphism group of the Cayley graph}
\label{sec:3}

By a result of Cayley, every $h\in \Sym_n$ defines a \emph{right translation} $\textsf{h}$ which is the automorphism of $\Cay$ that takes the vertex $\pi$ to the vertex $\pi\circ h$, and hence the edge $\{\pi,\rho\}$ to the edge $\{\pi\circ h,\rho\circ h\}$; see Section \ref{cay}. These automorphisms form the \emph{right translation group} $R(\Cay)$ of $\Cay$. Clearly, $\Sym_n\cong R(\Cay)$. Furthermore,
since $R(\Cay)$ acts regularly on $\Sym_n$, every automorphism of $\Cay$ is the product of a right translation by an automorphism fixing $\iota$.

One may ask if there is a nontrivial automorphism of $\Cay$ fixing $\iota$. The answer is affirmative by the following results.
\begin{lem}\label{mat27oct}
For any $\pi,\rho \in \Sym_n$,
\begin{itemize}
\item[\rm(i)] $\emph{\textsf{f}}_r(\pi\circ\rho)=\emph{\textsf{f}}_{\rho_r}(\pi)\circ\emph{\textsf{f}}_r(\rho)$;
\item[\rm(ii)] $\emph{\textsf{g}}(\pi\circ\rho)=\emph{\textsf{g}}(\pi)\circ\emph{\textsf{g}}(\rho)$.
\end{itemize}
\end{lem}
\begin{proof}
(i) From (\ref{eq2oct9}), $\textsf{f}_r(\pi\circ\rho)=\mu$ with
$$\begin{array}{lll}
[0\,\mu]&=& \alpha^{-(\pi\circ\rho)_r}\circ [0\,\pi] \circ [0\,\rho]\circ \alpha^r\\
{}&=&\alpha^{-(\pi\circ\rho)_r}\circ [0\,\pi]\circ\alpha^{\rho_r} \circ\alpha^{-\rho_r} [0\,\rho]\circ \alpha^r.
\end{array}$$ Now, the first assertion follows from (\ref{eq2oct9}).

(ii) By (\ref{eq2oct9b}), $\textsf{g}(\pi\circ\rho)=\xi$ with
$$\begin{array}{lll}
[0\,\xi]&=& [0\, w]\circ[0\,\pi] \circ [0\,\rho]\circ [0\, w]\\
{}&=& [0\, w] \circ [0\,\pi]\circ [0\, w]\circ [0\, w]\circ [0\,\rho]\circ [0\, w].
\end{array}$$ Here, the second assertion follows from (\ref{eq2oct9b}). This concludes the proof.
\end{proof}

\begin{prop}
\label{prop10oct} Toric maps and the reverse map are automorphisms of $\Cay$.
\end{prop}
\begin{proof}Let $\pi,\rho \in \Sym_n$ be any two adjacent vertices of $\Cay$. Then $\rho=\sigma\circ\pi$, for some $\sigma=\sigma(i,j,k)\in S_n$. Here, Lemma \ref{mat27oct} yields $\textsf{f}(\rho)=\textsf{f}_{\pi_1}(\sigma)\circ\textsf{f}(\pi)$ and $\textsf{g}(\rho)=\textsf{g}(\sigma)\circ\textsf{g}(\pi)$. Therefore, the assertion for $\textsf{f}$ and $\textsf{g}$ follows from Proposition \ref{th1}. By induction on $r\geq 1$, this holds true for all toric maps.
\end{proof}
By (\ref{eqoct15a}), the set consisting of $\textsf{F}$ and its coset $\textsf{F}\circ\textsf{g}$ is a dihedral group $\textsf{D}_{n+1}$ of order $2(n+1)$. Clearly, $\textsf{D}_{n+1}$ fixes $\iota$. Now, Proposition \ref{prop10oct} has the following corollary.
\begin{cor}
\label{teor1} The automorphism group of $\Cay$ contains a dihedral subgroup $\textsf{D}_{n+1}$ of order $2(n+1)$ fixing the identity permutation.
\end{cor}
{}From now on, the term of \emph{toric-reverse group} stands for $\textsf{D}_{n+1}$, and $\textsf{G}$ denotes the stabilizer of $\iota$ in the automorphism group of $\Cay$. By Corollary \ref{teor1}, the problem arises whether $\textsf{D}_{n+1}$ is already $\textsf{G}$. We state our result on this problem.

Clearly, $\textsf{G}$ preserves the subgraph
of $\Cay$ whose vertices are the block transpositions. We call this subgraph $\Gamma$ the \emph{block transposition graph} and denote $\textsf{R}$ its automorphism group. The kernel of the permutation representation of $\textsf{G}$ on $S_n$ is a normal subgroup $\textsf{N}$, and the factor group $\textsf{G}/\textsf{N}$ is a subgroup of $\textsf{R}$. Since $\textsf{D}_{n+1}$ and $\textsf{N}$ have trivial intersection, by Lemma \ref{lemc19ott2013}, the toric-reverse group can be regarded as a subgroup of $\textsf{G}/\textsf{N}.$ One of the main results in our thesis is a proof of the theorem below.
\begin{thm}
\label{teor2} The automorphism group of $\Gamma$ is the toric-reverse group.
\end{thm}As a corollary, $\textsf{G}=\textsf{N}\rtimes\textsf{D}_{n+1}$. From this the following result is obtained.
\begin{cor}
\label{teor2b} The automorphism group of $\Cay$ is the product of the right translation group by $\emph{\textsf{N}}\rtimes \emph{\textsf{D}}_{n+1}.$
\end{cor}
\begin{rem} \emph{Computation shows that $\textsf{N}$ is trivial for $n\leq 8$. This motivates to make the following conjecture.}\end{rem}
\begin{conj}
\emph{The automorphism group of $\Cay$ is the product of the right translation group by the toric-reverse group.}
\end{conj}

In this context, the following result is out of interest, where the set of $\textsf{d}\circ\textsf{h}$ with $\textsf{d}\in\textsf{D}_{n+1}$ and $\textsf{h}\in R(\Cay)$ is $R(\Cay) \textsf{D}_{n+1}$.
\begin{prop}
\label{29oct}
The product of the right multiplicative group by the toric-reverse group is isomorphic to the direct product of $\Sym_{n+1}$ by a group of order $2$.
\end{prop}
\begin{proof} Two automorphisms of $\Cay$ arise from the reverse permutation, namely $\textsf{g}$ and the right translation $\textsf{w},$ and $\textsf{g}\circ\textsf{w}$ is the automorphism $\textsf{t}$ that takes $\pi$ to $w\circ \pi$. Obviously, $\textsf{t}\in R(\Cay)\textsf{D}_{n+1}$ is an involution as $\textsf{g}$ and $\textsf{w}$ are involutions.

Here, we show that $\textsf{t}$ centralizes $R(\Cay)\textsf{F}$. In order to do that, we show that $\textsf{t}$ commutes with any right translation $\textsf{h}$. For every $\pi\in \Sym_n$,
$$\textsf{h}\circ\textsf{g}\circ\textsf{w}(\pi)=\textsf{h}(\rho)\Longleftrightarrow[0\,\rho]=[0\, w]\circ [0\,\pi].$$Then $\textsf{h}\circ\textsf{g}\circ\textsf{w}(\pi)=w\circ\pi\circ h$. On the other side, $$\textsf{g}\circ\textsf{w}\circ\textsf{h}(\pi)=\textsf{g}\circ(\pi\circ h\circ w)=\rho'\Longleftrightarrow[0\,\rho']=[0\, w]\circ [0\,\pi]\circ [0\, h].$$Thus $\textsf{g}\circ\textsf{w}\circ\textsf{h}(\pi)=\textsf{h}\circ\textsf{g}\circ\textsf{w}(\pi)$. Now, it suffices to prove $\textsf{t}\circ\textsf{f}=\textsf{f}\circ\textsf{t}$. For every $\pi\in \Sym_n$, $$\textsf{t}\circ\textsf{f}(\pi)=\textsf{g}(\textsf{f}\circ w)(\pi)=\xi\Longleftrightarrow[0\,\xi]=[0\, w]\circ [0\,\textsf{f}(\pi)]=[0\, w]\circ\alpha^{-\pi_{1}}\circ[0\,\pi]\circ\alpha.$$ As $[0\,
 w]\circ\alpha^{-\pi_{1}}\circ[0\,w]=\alpha^{\pi_1}$ by (\ref{eqoct15a}), $[0\, \xi]=\alpha^{\pi_1}\circ[0\,w\circ\pi]\circ\alpha$. On the other hand, from Lemma \ref{mat27oct} (i) we have
$$\textsf{f}\circ\textsf{t}(\pi)=\textsf{f}_{\pi_1}(w)\circ\textsf{f}(\pi)= \xi'\Longleftrightarrow
 [0\,\xi']=\alpha^{-w_{\pi_1}}\circ[0\,w]\circ[0\,\pi]\circ\alpha.$$Since $\alpha^{-w_{\pi_1}}=\alpha^{n+1-w_{\pi_1}}$ and $w_{\pi_1}=n+1-\pi_1$, $\textsf{f}\circ\textsf{t}(\pi)=\textsf{t}\circ\textsf{f}(\pi)$. This yields that $\textsf{t}$ commutes with $\textsf{F}$.

Now, we show that $\textsf{t}$ is off $R(\Cay)\textsf{F}$. Suppose on the contrary that there exists some right translation $\textsf{h}$ such that $\textsf{t}=\textsf{h}\circ\textsf{f}^{\,r}$ with $0\le r \le n$. Since $\textsf{t}$ is an involution this implies $\textsf{t}\circ\textsf{h}\in\textsf{F},$  and then $\textsf{t}\circ\textsf{h}$ fixes $\iota$. On the other hand, $\textsf{t}\circ\textsf{h}(\iota)=w\circ h.$ Thus, $h=w$ is an involution. Therefore, $\textsf{t}\circ\textsf{h}=\textsf{t}\circ\textsf{w}$ is an involution as well. Since $\iota$ is the only involution in $\textsf{F},$ $\textsf{t}\circ\textsf{h}=\iota,$ hence $\textsf{t}=\textsf{h}.$ Thus, we have proven that $\textsf{t}$ is a right translation. Since the center of $R(\Cay)$ is trivial while $\textsf{t}$ commutes with any right translation, we have $\textsf{t}\not\in R(\Cay),$ a contradiction.

Therefore, $\textsf{t}\in R(\Cay)\textsf{D}_{n+1}\supseteq R(\Cay)\textsf{F}\times \langle\textsf{t}\rangle$. Actually, the two sets coincide since $$\textsf{h}\circ\textsf{f}^{\,r}\circ\textsf{g}=\textsf{h'}\circ\textsf{f}^{-r}\circ\textsf{g}\circ\textsf{w},$$where $\textsf{h'}=\textsf{h}\circ\textsf{w}$, for any right translation $\textsf{h}$ and $0\le r\le n$. In fact, as $\textsf{t}$ commutes with every right translation and with $\textsf{F}$, (\ref{eqoct15a}) yields
$$\textsf{h}\circ\textsf{w}\circ\textsf{f}^{-r}\circ\textsf{t}=\textsf{h}\circ\textsf{g}\circ\textsf{f}^{-r}=\textsf{h}\circ\textsf{f}^{\,r}
\circ\textsf{g}.$$

To prove the isomorphism $R(\Cay)\textsf{F}\cong \Sym_{n+1}$, let $\Phi$ be the map that takes $\textsf{h}\circ\textsf{f}^{\,r}$ to $[0\,h^{-1}]\circ \alpha^{-r}$. For any $\textsf{k}\in R(\Cay),\,\pi\in \Sym_n,$ and $0\leq r,u \le n$, by Lemma \ref{mat27oct} (i),
$$\textsf{h}\circ\textsf{f}_r\circ \textsf{k}\circ\textsf{f}_u(\pi)=\textsf{h}\circ\textsf{f}_r(\textsf{f}_u(\pi)\circ k)=\textsf{f}_{u+k_r}(\pi)\circ\textsf{f}_{r}(k)\circ h.$$This shows that $\textsf{h}\circ\textsf{f}_r\circ \textsf{k}\circ\textsf{f}_u=\textsf{d}\circ \textsf{f}_{u+k_r}$ with $d=\textsf{f}_r(k) \circ h$ and $\textsf{d}$ the right translation associated to $d$. Then
$$\begin{array}{lll}
\Phi(\textsf{h}\circ\textsf{f}_r\circ \textsf{k}\circ\textsf{f}_u)&=&[0\,h^{-1}]\circ[0\,\textsf{f}_r(k)^{-1}]\circ\alpha^{-u-k_r}\\
{}&=&[0\,h^{-1}]\circ\alpha^{-r}\circ[0\,k^{-1}]\circ\alpha^{-u}.
\end{array}$$ On the other hand,
$$\Phi(\textsf{h}\circ\textsf{f}_r)\circ\Phi(\textsf{k}\circ\textsf{f}_u)=[0\,h^{-1}]\circ\alpha^{-r}\circ[0\,k^{-1}]\circ\alpha^{-u}.$$Hence, $\Phi$ is a group homomorphism from $R(\Cay)\textsf{F}$ into the symmetric group on $[n]^0$. Furthermore, $\ker(\Phi)$ is trivial. In fact,  $[0\,h^{-1}]\circ \alpha^{-r}=[0\,\iota]$ only occurs for $h=\iota$ since the inverse of $\alpha^{-r}$ is the permutation $\alpha^{r}$ not fixing $0$. This together with $(n+1)!=|R(\Cay)\textsf{F}|$ shows that $\Phi$ is bijective.
\end{proof}

The proof of Theorem \ref{teor2} depends on several results on combinatorial properties of $\Gamma$, especially on the set of its maximal cliques of size $2$. These results of independent interest are stated and proven in the next sections.

\section{Properties of the block transposition graph}\label{s61}
In this section we refer to $\Cat$ as the (right-invariant) Cayley graph, where $\{\pi,\rho\}$ is an edge if and only if $\rho=\pi \circ\sigma(i,j,k),$ for some $\sigma(i,j,k)\in T_n.$ Obviously, the vertices of $\Cat$ as well as of $\Cay$ adjacent to $\iota$ are the block transpositions. Also, the left-invariant and right-invariant Cayley graphs are isomorphic. In fact, the map taking any permutation to its inverse is such an isomorphism. Our choice is advantageous as the proofs in this section are formally simpler with the right-invariant Cayley graph notation. This change may be justified by (\ref{feb9}), which shows that computing $\pi\circ\sigma$ is more natural and immediate than $\sigma\circ\pi,$ whenever $\pi\in\Sym_n$ and $\sigma\in T_n.$

Now, every toric map $\textsf{f}_r$ is replaced by $\bar{\textsf{f}}_r$ defined as
\begin{equation}
\label{March29+}
\bar{\textsf{f}}_r(\pi)=(\textsf{f}_r(\pi^{-1}))^{-1},\qquad \pi\in \Sym_n.
\end{equation}
In addition, from (\ref{lem3oct11}) applied to $r=1,$
\begin{equation}
\label{March29}
\bar{\textsf{f}}(\pi)=\textsf{f}(\pi)^{\pi^{-1}_1},\qquad \pi\in \Sym_n.
\end{equation}
This shows that $\bar{\textsf{f}}\not\in \textsf{F}.$ Nevertheless, $\bar{\textsf{f}}_r=\bar{\textsf{f}}^{\,r}$, as $\textsf{f}_r=\textsf{f}^{\,r}$ for any integer $r$ with $0\leq r\leq n$. Then $\bar{\textsf{F}}\cong \textsf{F},$ where
$\bar{\textsf{F}}$ is the group generated by $\bar{\textsf{f}},$ and the natural map ${\bar{\textsf{f}}}_{r}\rightarrow \textsf{f}_{r}$ is an isomorphism.

Furthermore, since $\textsf{g}(\pi^{-1})^{-1}=\textsf{g}(\pi)$ for any $\pi\in \Sym_n,$ $\bar{\textsf{g}}$ coincides with $\textsf{g}.$ In addition, the group $\overline{\textsf{D}}_{n+1}$ generated by $\bar{\textsf{f}}$ and
$\textsf{g}$ is isomorphic to $\textsf{D}_{n+1}$, and then this is the \emph{toric-reverse group} of $\Cat.$
\begin{lem}
\label{22marchC2015}
Let $\sigma(i,j,k)$ be any block transposition on $[n].$ Then
\begin{equation}
\label{22march2015}
\bar{\emph{\textsf{f}}}(\sigma(i,j,k))=\left\{\begin{array}{ll}
\sigma(i-1,j-1,k-1), & i >0 ,\\
\sigma(j-1,k-1,n), & i=0.
\end{array}
\right.
\end{equation}
\end{lem}
\begin{proof}
Let $\sigma=\sigma(i,j,k).$ For $i>0,$ we obtain $\sigma_1=1$ from (\ref{eq22ott12}). Therefore, $\bar{\textsf{f}}(\sigma)=\textsf{f}(\sigma)$ by (\ref{March29}). Hence the statement for $i>0$ follows from Lemma \ref{lemc19ott2013}.

Now, suppose $i=0.$ By (\ref{March29+}) and Lemma \ref{lemc19ott2013},
$$\bar{\textsf{f}}(\sigma)=(\textsf{f}(\sigma^{-1}))^{-1}=(\textsf{f}(\sigma(0,k-j,k)))^{-1}=\sigma(j-1,n-(k-j),n)^{-1}$$
which is equal to $\sigma(j-1,k-1,n),$ by (\ref{eqa18oct}). Therefore, the statement also holds for $i=0.$
\end{proof}
Now, we transfer our terminology from Section \ref{sec:3}. In particular, $\bar{\textsf{f}}$ and its powers are the \emph{toric maps},
$\bar{\textsf{F}}$ the \emph{toric group}, and $\bar{\Gamma}$ is the \emph{block transposition graph} of $\Cat.$
\begin{prop}
\label{22march22D2015} Toric maps and the reverse map are automorphisms of $\Cat.$
\end{prop}
\begin{proof}{}From Lemma \ref{mat27oct} (ii) and Corollary \ref{th1} follows that the reverse map $\textsf{g}$ is also an automorphism of $\Cat.$

Now, it suffices to prove the claim for $\bar{\textsf{f}}.$ Take an edge $\{\pi,\rho\}$ of $\Cat.$ Then $\rho=\pi\circ\sigma$ with $\sigma\in T_n,$ and
$$\bar{\textsf{f}}(\pi\circ \sigma)=(\textsf{f}(\sigma^{-1}\circ \pi^{-1}))^{-1}=(\textsf{f}_{\pi_1^{-1}}(\sigma^{-1})\circ\textsf{f}(\pi^{-1}))^{-1}=
\bar{\textsf{f}}(\pi)\circ\textsf{f}_{\pi_1^{-1}}(\sigma^{-1})^{-1},$$by Lemma \ref{mat27oct} (i). Here $\textsf{f}_{\pi_1^{-1}}(\sigma^{-1})^{-1}\in T_n$ since $T_n$ is inverse closed, by (\ref{eqa18oct}), and $\textsf{F}$ leaves $T_n$ invariant, by Corollary \ref{th1}. Therefore, $\bar{\textsf{f}}(\pi)$ and $\bar{\textsf{f}}(\rho)$ are incident in $\Cat.$
\end{proof}
As consequence of Proposition \ref{22march22D2015}, all the results in Section \ref{sec:3} hold true for \\$\Cat$ up to the obvious change from ``right-translation'' to ``left-translation''.

Now, we introduce some subsets in $T_n$ that plays a relevant role in our study. Every permutation $\bar{\pi}$ on $[n-1]$ extends to a permutation $\pi$ on $[n]$ such that $\pi_t=\bar{\pi}_t$ for $1\le t \le n-1$ and $\pi_n=n.$ Hence,
$T_{n-1}$ is naturally embedded in $T_n$ since every $\sigma(i,j,k)\in T_n$ with $k\neq n$ is identified with the block transposition $\bar\sigma(i,j,k)$. On the other side, every permutation $\pi'$ on $\{2,3,\ldots,n\}$ extends to a permutation on $[n]$ such that $\pi_t=\pi'_t,$ for $2\le t \le n$ and $\pi_1=1.$ Thus, $\sigma(i,j,k)\in T_n$ with $i\neq 0$ is identified with the block transposition $\sigma'(i,j,k).$ The latter block transpositions form the set
$$S_{n-1}^\triangledown=\{ \sigma(i,j,k)|\,i\neq 0\}.$$ Also, $$S_{n-2}^\vartriangle=T_{n-1}\cap S_{n-1}^\triangledown$$ is the set of all block transpositions on the set $\{2,3,\ldots,n-1\}$.
Our discussion leads to the following results.
\begin{lem}[Partition lemma] Let $L=T_{n-1}\setminus S_{n-2}^\vartriangle$ and let $F=S_{n-1}^\triangledown\setminus S_{n-2}^\vartriangle$. Then
\label{lem1oct9}
$$T_n=B\cupdot L \cupdot F \cupdot S_{n-2}^\vartriangle.$$
\end{lem}With the above notation, $L$ is the set of all $\sigma(0,j,k)$ with $k\neq n$, and $F$ is the set of all $\sigma(i,j,n)$ with $i\neq0.$ Furthermore, $|B|=n-1,\,|L|=|F|=(n-1)(n-2)/2,$ and $|S_{n-2}^\vartriangle|=(n-1)(n-2)(n-3)/6.$

Since $B$ consists of all nontrivial elements of a subgroup of $T_n$ of order $n,$ the block transpositions in $B$ are the vertices of a complete graph  of size $n-1.$ Lemma \ref{lem1oct9} and (\ref{eq1oct11}) give the following property.
\begin{cor}\label{lem2oct11}
The reverse map preserves both $B$ and $S_{n-2}^\vartriangle$ while it switches $L$ and $F$.
\end{cor}
\begin{lem}\label{lem1oct11}
No edge of $\Cat$ has one endpoint in $B$ and the other in $S_{n-2}^\vartriangle$.
\end{lem}
\begin{proof}
Suppose on the contrary that $\{\sigma(i',j',k'),\sigma(0,j,n)\}$ with $i'\neq 0$ and $k'\neq n$ is an edge of $\Cat.$ By (\ref{eqa18oct}),
$\rho=\sigma(0,n-j,n)\circ \sigma(i',j',k')\in T_n.$ Also, $\rho\in B$ as $\rho_1\neq 1$ and $\rho_n\neq n$. Since $B$ together with the identity is a group, $\sigma(0,j,n)\circ\rho$ is also in $B$. This yields $\sigma(i',j',k')\in B,$ a contradiction with Lemma \ref{lem1oct9}.
\end{proof}
The proofs of the subsequent properties use a few more equations involving block transpositions which are stated in the following two lemmas.
\begin{lem}\label{lem2oct13}
In each of the following cases $\{\sigma(i,j,k),\sigma(i',j',k')\}$ is an edge of \\$\Cat.$
\begin{itemize}
\item[\rm(i)]$(i',j')=(i,j);$
\item[\rm(ii)]$(i',j')=(j,k)$ for $k<k';$
\item[\rm(iii)]$(j',k')=(j,k);$
\item[\rm(iv)]$(j',k')=(i,j)$ for $i'<i;$
\item[\rm(v)]$(i,k)=(i',k')$ for $j<j'.$
\end{itemize}
\end{lem}
\begin{proof}
(i) W.l.g. $k'<k.$ By (\ref{eq22ott12}), $\sigma(i,j,k)=\sigma(i,j,k')\circ\sigma(k'-j+i,k',k).$
(iii) W.l.g. $i'<i.$ From (\ref{eq22ott12}),
$\sigma(i,j,k)=\sigma(i',j,k)\circ\sigma(i',k-j+i',k-j+i).$

In the remaining cases, from (\ref{eq22ott12}),
$$\begin{array}{lll}
\sigma(i,j,k)=\sigma(j,k,k')\circ\sigma(i,k'-k+j,k'),\\
\sigma(i,j,k)=\sigma(i',i,j)\circ\sigma(i',j-i+i',k),\\
\sigma(i,j,k)=\sigma(i,j',k)\circ\sigma(i,k-j+j',k).
\end{array}$$Hence the statements hold.
\end{proof}
The proof of the lemma below is straightforward and requires only (\ref{eq22ott12}).
\begin{lem}
\label{prop1oct12}
The following equations hold.
\begin{itemize}
\item[\rm(i)]$\sigma(i,j,n)=\sigma(0,j,n)\circ\sigma(0,n-j,n-j+i)$ for $i\neq 0;$
\item[\rm(ii)]$\sigma(i,j,n)=\sigma(0,i,j)\circ\sigma(0,j-i,n)$ for $i\neq 0;$
\item[\rm(iii)]$\sigma(0,j,n)=\sigma(i,j,n)\circ\sigma(0,i,n-j+i);$
\item[\rm(iv)]$\sigma(0,j,n)=\sigma(0,j,j+i)\circ\sigma(i,j+i,n)$ for $i\neq 0.$
\end{itemize}
\end{lem}
\begin{lem}
\label{leaoct19}
Let $i$ be an integer with $0<i\leq n-2.$
\begin{itemize}
\item[\rm(i)] If $\sigma(i,j,n)=\sigma(0,\bar{j},n)\circ\sigma(i',j',k'),$ then $\bar{j}=j.$
\item[\rm(ii)] If $\sigma(i,j,n)=\sigma(i',j',k')\circ \sigma(0,\bar{j},n),$ then $\bar{j}=i-j.$
\end{itemize}
\end{lem}
\begin{proof} (i) Assume $\bar{j}\neq j.$ From Lemma \ref{prop1oct12} (i) and (\ref{eqa18oct}),
\begin{equation}\label{matoct25}
\sigma(i',j',k')=\sigma(0,j^*,n)\circ\sigma(0,n-j,n-j+i),
\end{equation}where $j^*$ denotes the smallest positive integer such that $j^*\equiv j-\bar{j} \pmod n$. First we prove $i'=0$. Suppose on the contrary, then $$(\sigma(0,j^*,n)\circ\sigma(0,n-j,n-j+i))_1=1.$$On the other hand, $\sigma(0,n-j,n-j+i)_1=n-j+1$ and $\sigma(0,j^*,n)_{n-j+1}=n-\bar{j}+1$ since
$\sigma(0,j^*,n)_t=t+j^*\pmod n$ by (\ref{feb9}). Thus, $n-\bar{j}+1=1,$ a contradiction since $\bar{j}<n.$

Now, from (\ref{matoct25}), $\sigma(0,j',k')_n\neq n.$ Hence $k'=n$. Therefore, $$\sigma(0,n-j,n-j+i)=\sigma(0,j+\bar{j},n)\circ \sigma(0,j',n)\in B.$$A contradiction since $i\neq j$. This proves the assertion.

(ii) Taking the inverse of both sides of the equation in (ii) gives by (\ref{eqa18oct})
$$\sigma(i,n-j+i,n)=\sigma(0,n-\bar{j},n)\circ \sigma(i',j',k')^{-1}.$$ Now, from (i), $n-\bar{j}=n-j+i$, and the assertion follows.
\end{proof}
\begin{prop}
\label{prova}
The bipartite graphs arising from the components of the partition in Lemma \ref{lem1oct9} have the following properties.
\begin{itemize}
\item[\rm(i)] In the bipartite subgraph $(L\cup F,B)$ of $\Cat,$ every vertex in $L\cup F$ has degree $1$ while every vertex of $B$ has degree $n-2.$
\item[\rm(ii)] The bipartite subgraph $(L,F)$ of $\Cat$ is a $(1,1)$-biregular graph.
\end{itemize}
\end{prop}
\begin{proof} (i) Lemma \ref{leaoct19} (i) together with Lemma \ref{prop1oct12} (i) show that every vertex in $F$ has degree $1.$ Corollary \ref{lem2oct11} ensures
that this holds true for $L.$

For every $1\le j \le n-1,$, Lemma  \ref{prop1oct12} (iii) shows that there exist at least $j-1$ edges incident with $\sigma(0,j,n)$ and a vertex in $F.$ Furthermore, from Lemma  \ref{prop1oct12} (iv), there exist at least $n-j-1$ edges incident with $\sigma(0,j,n)$ and a vertex in $L$. Therefore, at least $n-2$ edges incident with $\sigma(0,j,n)$ have a vertex in $L\cup F$. On the other hand, this number cannot exceed $n-2$ since $|L\cup F|=(n-1)(n-2)$ from Lemma \ref{lem1oct9}. This proves the first assertion.

(ii) From Lemma  \ref{prop1oct12} (ii), there exists at least one edge with a vertex in $F$ and another in $L$. Also, Lemma \ref{leaoct19} (ii) ensures the uniqueness of such an edge.
\end{proof}
From now on, $\bar{\Gamma}(W)$ stays for the induced subgraph of $\bar{\Gamma}$ on the vertex-set $W.$
 \begin{cor}
\label{cor3oct13ter}
$B$ is the unique maximal clique of $\bar{\Gamma}$ of size $n-1$ containing an edge of $\bar{\Gamma}(B).$
\end{cor}
\begin{proof}
Proposition \ref{prova} (i) together with Lemma \ref{lem1oct9} show that the endpoints of an edge of $\bar{\Gamma}(B)$ do not have a common neighbor outside $B.$
\end{proof}
Computations performed by using the package ``grape'' of GAP ~\cite{gap} show that $\bar{\Gamma}$ is a $6$-regular subgraph for $n=5$ and $8$-regular subgraph for $n=6,$ but $\bar{\Gamma}$ is only $3$-regular for $n=4.$ This generalizes to the following result.
\begin{prop}\label{thm1oct11}
$\bar{\Gamma}$ is a $2(n-2)$-regular graph whenever $n\geq 5.$
\end{prop}
\begin{proof}
Since $B$ is a maximal clique of size $n-1$, every vertex of $B$ is incident with $n-2$ edges of $\bar{\Gamma}(B).$ From Proposition \ref{prova} (i), as many as $n-2$ edges incident with a vertex in $B$ have an endpoint in $L\cup F$. Thus, the assertion holds for the vertices in $B.$

In $\bar{\Gamma}(F)$ every vertex has degree $2(n-1)-4=2n-6,$ by induction on $n.$ This together with Proposition \ref{prova} (ii) show that every vertex of $\bar{\Gamma}(F)$ has degree $2n-5$ in $\bar{\Gamma}(L\cup F).$ By Corollary \ref{lem2oct11}, this holds true for every vertex of $\bar{\Gamma}(L).$ The degree increases to $2n-4$ when we also count the unique edge in $\bar{\Gamma}(B),$ according to the first assertion of Proposition \ref{prova} (i).

In $\bar{\Gamma}(S_{n-2}^\vartriangle)$ every vertex has degree $2n-8,$ by induction on $n.$ Furthermore, in $\bar{\Gamma}(L\cup S_{n-2}^\vartriangle)$ every vertex has degree $2n-6$ by induction on $n$, and the same holds for  $\bar{\Gamma}(F\cup S_{n-2}^\vartriangle).$ This together with Lemma \ref{lem1oct11} show that every vertex in $S_{n-2}^\vartriangle$ is the endpoint of exactly $2(2n-6)-(2n-8)$ edges in $\bar{\Gamma}.$
\end{proof}
Our next step is to determine the set of all maximal cliques of $\bar{\Gamma}$ of size $2.$ From now on, we will be referring to the edges of the complete graph arising from a clique as the edges of the clique.
According to Lemma \ref{lem2oct13} (v), let $\Lambda$ be the set of all edges $$e_l=\{\sigma(l,l+1,l+3),\sigma(l,l+2,l+3)\},$$ where $l$ ranges over $\{0,1,\ldots n-3\}$. From (\ref{eqa18oct}), the endpoints of such an edge are the inverse of one another.
\begin{prop}
\label{thm2oct11} Let $n\geq 5.$ The edges in $\Lambda$ together with three more edges
\begin{equation}
\label{eq1aoct12}
\begin{array}{lll}
e_{n-2}&=&\{\sigma(0,n-2,n-1),\sigma(0,n-2,n)\};\\
e_{n-1}&=&\{\sigma(1,n-1,n),\sigma(0,1,n-1)\};\\
e_{n}&=&\{\sigma(0,2,n),\sigma(1,2,n)\};\\
\end{array}
\end{equation}are pairwise disjoint edges of maximal cliques of $\bar{\Gamma}$ of size $2.$
\end{prop}
\begin{proof} Since $n\geq 5,$ the above edges are pairwise disjoint.

Now, by (\ref{22march2015}), the following equations
\begin{equation}\label{eq1oct13}
\begin{array}{llllll}
\bar{\textsf{f}}(\sigma(l,l+1,l+3))&=&\sigma(l-1,l,l+2)& \mbox{ for } l\geq 1;\\
\bar{\textsf{f}}(\sigma(l,l+2,l+3))&=&\sigma(l-1,l+1,l+2) & \mbox{ for } l\geq 1; \\
\bar{\textsf{f}}(\sigma(0,1,3))&=&\sigma(0,2,n);\\
\bar{\textsf{f}}(\sigma(0,2,3))&=&\sigma(1,2,n);\\
\bar{\textsf{f}}(\sigma(0,2,n)&=&\sigma(1,n-1,n);\\
\bar{\textsf{f}}(\sigma(1,2,n))&=&\sigma(0,1,n-1);\\
\bar{\textsf{f}}(\sigma(1,n-1,n)&=&\sigma(0,n-2,n-1); \\
\bar{\textsf{f}}(\sigma(0,1,n-1))&=&\sigma(0,n-2,n);\\
\bar{\textsf{f}}(\sigma(0,n-2,n-1))&=&\sigma(n-3,n-2,n); \\
\bar{\textsf{f}}(\sigma(0,n-2,n))&=&\sigma(n-3,n-1,n).
\end{array}
\end{equation}
hold. This shows that $\bar{\textsf{f}}$ leaves the set $\Lambda\cup\{e_{n-2},e_{n-1},e_n\}$ invariant acting on it as the cycle permutation
$(e_n,\,e_{n-1},\cdots, e_1,\,e_0).$

Now, it suffices to verify that $e_n$ is a maximal clique of $\bar\Gamma.$ Assume on the contrary that $\sigma=\sigma(i,j,k)$ is adjacent to both $\sigma(1,2,n)$ and $\sigma(0,2,n).$ As $\sigma(0,2,n)\in B,$ Lemma \ref{lem1oct11} implies that $\sigma\in L\cup F$. Also, Proposition \ref{prova} (i) shows that $\sigma(0,2,n)$ has degree $n-2$ in $L\cup F$. In particular, in the proof of Proposition \ref{prova} (i), we have seen that $\sigma(0,2,n)$ must be adjacent to $n-3$ vertices of $L,$ as $\sigma(1,2,n)\in F.$ Then, by Lemma \ref{prop1oct12} (iv), $\sigma=\sigma(0,2,l)$ for some $l$ with $3\leq l < n.$

On the other hand, Proposition \ref{prova} (ii) shows that $\sigma\in L$ is uniquely determined by $\sigma(1,2,n)\in F,$ and, by Lemma \ref{prop1oct12} (ii), $\sigma=\sigma(0,1,2),$ a contradiction.
\end{proof}
From now on, $V$ denotes the set of the vertices of the edges $e_m$ with $m$ ranging over $\{0,1,\ldots,n\}.$ For $n=4,$ the edges $e_m$ are not pairwise disjoint, but computations show that they are also edges of maximal cliques of $\bar{\Gamma}$ of size $2.$
\begin{lem}
\label{lem1oct18} The toric maps and the reverse map preserve $V.$ Then, the toric-reverse group is regular on $V,$ and $\bar{\Gamma}(V)$ is a vertex-transitive graph.
\end{lem}
\begin{proof}
Since $\bar{\textsf{F}}$ is the subgroup generated by $\bar{\textsf{f}}$, from (\ref{eq1oct13}) follows that $\bar{\textsf{F}}$ preserves $V$ and has two orbits on $V,$ each of them containing one of the two endpoints of the edges $e_m$ with $0\le m \le n.$

In addition, by (\ref{eq1oct11}), the reverse map $\textsf{g}$ interchanges the endpoints of $e_m$ with $0\le m\le n-3$ and $m=n-1$ while
$$\begin{array}{lll}
\bar{\textsf{g}}(\sigma(0,n-2,n-1))&=&\sigma(1,2,n);\\
\bar{\textsf{g}}(\sigma(0,n-2,n)&=&\sigma(0,2,n);\\
\bar{\textsf{g}}(\sigma(1,2,n))&=&\sigma(0,n-2,n-1);\\
\bar{\textsf{g}}(\sigma(0,2,n))&=&\sigma(0,n-2,n).
\end{array}$$This implies that $\textsf{g}$ preserves $V$, and $\overline{\textsf{D}}_{n+1}$ acts transitively on $V.$

Now, since $|V|=2(n+1)$ and $\overline{\textsf{D}}_{n+1}$ has order $2(n+1)$, then $\overline{\textsf{D}}_{n+1}$ is regular on $V.$
\end{proof}
Our next step is to show that the $e_m$ with $0\le m \le n$ are the edges of all maximal cliques of $\bar{\Gamma}$ of size $2.$ Computations performed by using the package ``grape'' of GAP ~\cite{gap} show that the assertion is true for $n =4,5,6.$
\begin{lem}\label{cor1oct13}
The edge of every maximal clique of $\bar{\Gamma}$ of size $2$ is one of the edges $e_m$ with $0\le m \le n.$
\end{lem}
\begin{proof}
On the contrary take an edge $e$ of a maximal clique of $\bar{\Gamma}$ of size $2$ other than the edges $e_m$. Since $L\cup S_{n-2}^\vartriangle\subset T_{n-1},$ by induction on $n\geq4$, $e$ is an edge of $\bar{\Gamma}(F\cup B).$ Also, $e$ has one endpoint in $B$ and the other in $F$, as $B$ is clique.

Now, let the endpoint of $e$ in $B$ be $\sigma(0,j,n)$ for some $1\leq j\leq n-1.$ Then, by the proof of the first assertion of Proposition \ref{prova} (i), the vertex $\sigma(0,j,n)$ is adjacent to $\sigma(\bar{i},j,n)$ for any $0\le \bar{i} < j.$ As the vertices $\sigma(\bar{i},j,n)$ for $0\le \bar{i} < j$ are adjacent by Lemma \ref{lem2oct13} (iii), $e$ is and edge of the triangle of vertices $\sigma(0,j,n),\,\sigma(i',j,n),$ and $\sigma(\bar{i},j,n)$ with $i'\neq\bar{i}$ and $0\le i',\bar{i} < j,$ a contradiction.
\end{proof}

Lemma \ref{cor1oct13} shows that $V$ consists of the endpoints of the edges of $\bar{\Gamma}$ which are the edges of maximal cliques of size $2.$ Thus $\bar{\Gamma}(V)$ is relevant for the study of $\Cat.$ We show some properties of $\bar{\Gamma}(V).$
\begin{prop}\label{lemgoct12}
$\bar{\Gamma}(V)$ is a $3$-regular graph.
\end{prop}
\begin{proof} First we prove the assertion for the endpoint $v=\sigma(0,2,n)$ of $e_{n}.$ By Lemma \ref{lem2oct13} (i) (iii) (v), $\sigma(0,2,3),\sigma(1,2,n),$ and $\sigma(0,n-2,n)$ are neighbors of $v$. Since $\sigma(1,2,n)\in F$ and $\sigma(0,2,3)\in L,$ from the first assertion of Proposition \ref{prova} (i), $v\in B$ is not adjacent to any other vertex in either $V\cap F$ or $V\cap L.$ Also, Lemma \ref{lem1oct11} yields that no vertex in $V\cap S_{n-2}^\vartriangle$ is adjacent to $\sigma(0,2,n).$ Thus, $v$ has degree $3$ in $\bar{\Gamma}(V).$

Now the claim follows from Lemma \ref{lem1oct18}.
\end{proof}
\begin{rem} \em{By a famous conjecture of Lov\'asz, every finite, connected, and  vertex-transitive graph contains a Hamiltonian cycle, except the five known counterexamples; see~\cite{LL,BL}. Then, the second assertion of Lemma \ref{lem1oct18} and Proposition \ref{thm1oct13} show that the Lov\'asz conjecture holds for the graph $\bar\Gamma(V).$}
\end{rem}
\begin{prop}\label{thm1oct13}
$\bar{\Gamma}(V)$ is a Hamiltonian graph whenever $n\geq 5.$
\end{prop}
\begin{proof}
Let $v_1=\sigma(n-4,n-3,n-1),\quad v_2=\sigma(n-4,n-2,n-1)$ be the endpoints of $e_{n-4}$. We start by exhibiting a path $\mathcal{P}$ in $V$ beginning with $\sigma(0,2,3)$ and ending with $v_1$ that visits all vertices $\sigma(l,l+1,l+3),\sigma(l,l+2,l+3)\in\Lambda$ with $0\leq l\leq n-4.$

For $n=5,\,v_1=\sigma(1,2,4),$ and $$\mathcal{P}=\sigma(0,2,3),\sigma(0,1,3),\sigma(1,3,4),v_1.$$

Assume $n>5.$ For every $l$ with $0\leq l\leq n-4,$ Lemma \ref{lem2oct13} (ii) (v) show that both edges below are incident to $\sigma(l,l+1,l+3)$:
$$\{\sigma(l,l+1,l+3),\sigma(l+1,l+3,l+4)\},\quad\{\sigma(l,l+2,l+3),\sigma(l,l+1,l+3)\}.$$Therefore,
$$\begin{array}{ll}
\sigma(0,2,3),\sigma(0,1,3),\sigma(1,3,4),\ldots,\sigma(l,l+2,l+3),\sigma(l,l+1,l+3),\\\sigma(l+1,l+3,l+4),\ldots,\,v_1
\end{array}$$is a path $\mathcal{P}$ with the requested property.

By Lemma \ref{lem2oct13}, there also exists a path $\mathcal{P'}$ beginning with $v_1$ and ending with $\sigma(0,2,3)$ which visits the other vertices of $V,$ namely
$$\begin{array}{lll}
v_1,\sigma(n-3,n-1,n),\sigma(n-3,n-2,n),\sigma(0,n-2,n),\sigma(0,n-2,n-1),\\\sigma(0,1,n-1),\sigma(1,n-1,n),\sigma(1,2,n),
\sigma(0,2,n),\sigma(0,2,3).
\end{array}$$By Theorem \ref{thm2oct11}, the vertices are all pairwise distinct. Therefore the union of $\mathcal{P}$ and $\mathcal{P'}$ is a cycle in $V$ that visits all vertices. This completes the proof.
\end{proof}
\begin{rem}
\em{For $n\geq 4,$ by Proposition \ref{lemgoct12} and Theorem \ref{teor2}, Proposition \ref{thm1oct13} also follows from a result of Alspach and Zhang~\cite{AZ} who proved that all cubic Cayley graphs on dihedral groups have Hamilton cycles.}
\end{rem}

\section{The automorphism group of the block transposition graph}\label{s62}
We are in a position to give a proof for Theorem \ref{teor2}. Since $\Gamma\cong \bar\Gamma$ and $\textsf{D}_{n+1}\cong\overline{\textsf{D}}_{n+1},$ we may prove Theorem \ref{teor2} using the right-invariant notation.

From Proposition \ref{22march22D2015}, the toric-reverse group $\overline{\textsf{D}}_{n+1}$ is a subgroup $\overline{\textsf{R}},$ the automorphism group of $\bar{\Gamma}.$ Also, $\overline{\textsf{D}}_{n+1}$ is regular on $V,$ by the second assertion of Lemma \ref{lem1oct18}. Therefore, Theorem \ref{teor2} is a corollary of the following lemma.
\begin{lem}
\label{lemfoct12} The identity is the only automorphism of $\bar{\Gamma}$ fixing a vertex of $V$ whenever $n\geq 5.$
\end{lem}
\begin{proof} We prove the assertion by induction on $n.$ Computation shows that the assertion is true for $n=5,6.$ Therefore, we assume $n\geq 7.$

First we prove that any automorphism of $\bar{\Gamma}$ fixing a vertex $v\in V$ is an automorphism of $\bar{\Gamma}(V)$ as well. Since $\overline{\textsf{D}}_{n+1}$ is regular on $V,$ we may
limit ourselves to take $\sigma(0,2,n)$ for $v.$ Let $\bar{\textsf{H}}$ be the subgroup of $\overline{\textsf{R}}$ which fixes $\sigma(0,2,n).$

We look inside the action of $\bar{\textsf{H}}$ on $\bar{\Gamma}(V)$ and show that $\bar{\textsf{H}}$ fixes the edge $\{\sigma(0,2,n),\sigma(0,n-2,n)\}.$ By Proposition \ref{lemgoct12}, $\bar{\Gamma}(V)$ is $3$-regular. More precisely, the endpoints of the edges of $\bar{\Gamma}(V)$ which are incident with $\sigma(0,2,n)$ are $\sigma(0,2,3),\,\sigma(1,2,n),$ and $\sigma(0,n-2,n)$; see Lemma \ref{prop1oct12} (i) (iii) (v). Also, by Proposition \ref{thm2oct11}, the edge $e_{n-1}=\{\sigma(0,2,n),\sigma(1,2,n)\}$ is the edge of a maximal clique of $\bar{\Gamma}$ of size $2$, and no two distinct edges of maximal cliques of $\bar{\Gamma}$ of size $2$ have a common vertex. Thus, $\bar{\textsf{H}}$ fixes $\sigma(1,2,n).$ Now, from Corollary \ref{cor3oct13ter}, the edge $\{\sigma(0,2,n),\sigma(0,n-2,n)\}$ lies in a unique maximal clique of size $n-1$. By Lemma \ref{lem2oct13} (i), the edge $\{\sigma(0,2,n),\sigma(0,2,3)\}$ lies on a clique of size $n-2$ whose set of vertices is $\{\sigma(0,2,k)| 3\le k\le n\}.$ Here, we prove that any clique $C$ of size $n-2$ containing the edge $\{\sigma(0,2,n),\sigma(0,2,3)\}$ is maximal. By the first assertion of Proposition \ref{prova} (i), $\sigma(0,2,3)$ is adjacent to a unique vertex in $B$, namely $\sigma(0,2,n).$ On the other hand, among the $2(n-2)$ neighbors of $\sigma(0,2,n)$ off $V$, only as many as $n-3$ vertices are off $B\cap V$, by the proof of Proposition \ref{thm1oct11}. Then, $C$ does not extend to a clique of size $n-1$. Therefore, $\bar{\textsf{H}}$ cannot interchange the edges $\{\sigma(0,2,n),\sigma(0,n-2,n)\}$ and $\{\sigma(0,2,n),\sigma(0,2,3)\}$ but fixes both.

Also, by Proposition \ref{lemgoct12} and Lemma \ref{prop1oct12} (i) (iii), $\sigma(0,n-2,n)$ is adjacent to $\sigma(0,n-2,n-1)$ and $\sigma(n-3,n-2,n).$ Since $e_{n-2}$ is the edge of a maximal clique of $\bar{\Gamma}$ of size $2$, $\bar{\textsf{H}}$ fixes $e_{n-2}=\{\sigma(0,n-2,n-1),\sigma(0,n-2,n)\}$. This together with what we have proven so far shows that $\bar{\textsf{H}}$ fixes $\sigma(n-3,n-2,n),$ and then the edges $e_{n-3}=\{\sigma(n-3,n-2,n),\sigma(n-3,n-1,n)\}.$

Now, as the edge $\{\sigma(0,2,n),\sigma(0,n-2,n)\}$ is in $\bar{\Gamma}(B)$, Corollary \ref{cor3oct13ter} implies that
$\bar{\textsf{H}}$ preserves $B$. And, as $\bar{\textsf{H}}$ fixes $\{\sigma(0,2,n),\sigma(0,2,3)\}$, $\bar{\textsf{H}}$ must fix $\sigma(0,2,n)\in B$ and $\sigma(0,2,3)\notin B.$ Also, $e_0=\{\sigma(0,1,3),\sigma(0,2,3)\}$ is preserved by $\bar{\textsf{H}}$, as we have seen above. Therefore,  $\sigma(0,1,3)$ is also fixed by $\bar{\textsf{H}}.$ Furthermore, $\sigma(2,3,5)\in S_{n-2}^\vartriangle$ is adjacent to $\sigma(0,2,3)$ in $\bar{\Gamma}(V)$, by Proposition \ref{lemgoct12} and Lemma \ref{prop1oct12} (ii); and then it is fixed by $\bar{\textsf{H}}$, as $\bar{\textsf{H}}$ preserves $S_{n-2}^\vartriangle,$ by Lemma \ref{lem1oct9}. Therefore, we have that $\bar{\textsf{H}}$ induces an automorphism group of $\bar{\Gamma}(S_{n-2}^\vartriangle)$ fixing a vertex $\sigma(2,3,5)\in S_{n-2}^\vartriangle.$ Then $\bar{\textsf{H}}$ fixes every block transpositions in $S_{n-2}^\vartriangle\cong T_{n-2},$ by the inductive hypothesis. In particular, $\bar{\textsf{H}}$ fixes all the vertices in $V \cap S_{n-2}^\vartriangle,$ namely all vertices in $\Lambda$ belonging to $e_l$ with $0<l<n-3.$

This together with what proven so far shows that $\bar{\textsf{H}}$ fixes all vertices of $V$ with only two possible exceptions, namely the endpoints of the edge $e_{n-1}=\{\sigma(1,n-1,n),\sigma(0,1,n-1)\}.$ In this exceptional case, $\bar{\textsf{H}}$ would swap $\sigma(0,1,n-1)$ and $\sigma(1,n-1,n).$ Actually, this exception cannot occur since $\sigma(0,1,n-1)$ and $\sigma(1,n-1,n)$ do not have a common neighbor, and $\bar{\textsf{H}}$ fixes their neighbors in $V.$ Therefore, $\bar{\textsf{H}}$ fixes every vertex in $V$. Hence, $\bar{\textsf{H}}$ is the kernel of the permutation representation of $\overline{\textsf{R}}$ on $V$. Thus $\bar{\textsf{H}}$ is a normal subgroup of $\overline{\textsf{R}}.$

Our final step is to show that the block transpositions in $L\cup B$ are also fixed by $\bar{\textsf{H}}$. Take any block transposition $\sigma(0,j,k).$ Then the toric class of $\sigma(0,j,k)$ contains a block transposition $\sigma(i',j',k')$ from $S_{n-2}^\vartriangle$. This is a consequence of the equations below which are obtained by using (\ref{22march2015})
\begin{equation}\label{eqa14oct}
\begin{array}{llll}
{\bar{\textsf{f}}}^{\,2}(\sigma(0,j,k))&=&\sigma(j-2,k-2,n-1),& j\geq 3; \\
{\bar{\textsf{f}}}^{\,3}(\sigma(0,1,k))&=&\sigma(k-3,n-2,n-1),&k\geq 4;\\
{\bar{\textsf{f}}}^{\,4}(\sigma(0,1,2))&=&\sigma(n-3,n-2,n-1);&{}\\
{\bar{\textsf{f}}}^{\,5}(\sigma(0,1,3))&=&\sigma(n-4,n-3,n-1);&{}\\
{\bar{\textsf{f}}}^{\,4}(\sigma(0,2,k))&=&\sigma(k-4,n-3,n-1),& k\geq 5;\\
{\bar{\textsf{f}}}^{\,5}(\sigma(0,2,3))&=&\sigma(n-4,n-2,n-1);&{}\\
{\bar{\textsf{f}}}^{\,6}(\sigma(0,2,4))&=&\sigma(n-5,n-3,n-1).&{}\\
\end{array}
\end{equation}Since $\sigma(i',j',k')\in S_{n-2}^\vartriangle,$ we know that $\bar{\textsf{H}}$ fixes $\sigma(i',j',k').$ From this we infer that $\bar{\textsf{H}}$ also fixes $\sigma(0,j,k).$ In fact, as $\sigma(0,j,k)$ and $\sigma(i',j',k')$ are torically equivalent, $\bar{\textsf{u}}(\sigma(i',j',k'))=\sigma(0,j,k)$ for some $\bar{\textsf{u}}\in\bar{\textsf{F}}$. Take any $\bar{\textsf{h}}\in\bar{\textsf{H}}$. As $\bar{\textsf{H}}$ is a normal subgroup of $\overline{\textsf{R}}$, there exists $\bar{\textsf{h}}_1\in \bar{\textsf{H}}$ such that ${\bar{\textsf{u}}}\circ\bar{\textsf{h}}_1=\bar{\textsf{h}}\circ \bar{\textsf{u}}.$ Hence
$$\sigma(0,j,k)={\bar{\textsf{u}}}(\sigma(i',j',k'))=
{\bar{\textsf{u}}}\circ\bar{\textsf{h}}_1(\sigma(i',j',k'))=\bar{\textsf{h}}\circ \bar{\textsf{u}}(\sigma(i',j,',k'))$$
whence $\sigma(0,j,k)=\bar{\textsf{h}}(\sigma(0,j,k)).$ Therefore, $\bar{\textsf{H}}$ fixes every block transposition in $L\cup B.$

Also, this holds true for $F,$ by the second assertion of Proposition \ref{prova}. Thus, by Lemma \ref{lem1oct9}, $\bar{\textsf{H}}$ fixes
every block transposition. This completes the proof.
\end{proof}
\begin{rem}
{\em{Lemma \ref{lemfoct12} yields Theorem \ref{teor2} for $n\geq 5.$ For $n=4,$ computations performed by using the package ``grape'' of GAP~\cite{gap} show that Theorem \ref{teor2} is also true.}}
\end{rem}

\chapter{Related rearrangement problems}\label{c7}

We have treated the concept of a rearrangement distance in a general setting in Section \ref{rear dist} and discussed the block transposition rearrangement problem in Chapter \ref{c4},\ref{c5},\ref{c6}. In this chapter, we give a brief survey of two other rearrangement problems. In the first section, we focus on reversals while in the last section we treat cut-and-paste moves, an operation that involves both block transpositions and reversals.
\section{Reversals}\label{s71}
Analysis of genomes evolving by inversions led to the combinatorial problem of sorting a permutation by reversals; for further biological knowledge see Chapter \ref{biol}. Introduced in 1982 by Watterson et al. \cite{W}, sorting by reversals is the first combinatorially studied rearrangement problem. For every any $0\leq i<k\leq n$, a \emph{reversal} $\rho(i,k)$ is the permutation
\begin{equation}
\rho(i,k)=\left\{\begin{array}{ll}
[1\cdots i\,\, k\cdots i+1\,\, k+1 \cdots n], & 1\leq i<k<n,\\
{[k\cdots 1\,\, k+1 \cdots n]}, & i=0\quad k< n,\,i=j,\\
w, & i=0\quad k=n.
\end{array}
\right.
\end{equation}
In \cite{W} Watterson et al. also suggested the first heuristic algorithm that sort a permutation in at most $n-1$ steps. It took more than a decade since Bafna and Pevzner were able to prove that $n-1$ is, actually, the reversal diameter of the symmetric group $\Sym_n$. They provided examples of permutations on $[n]$ with reversal distance equal to $n-1$. Such permutations are the \emph{Gollan permutation} $\gamma_n$ and its inverse, defined by Gola as follows
\[\arraycolsep=1.4pt\def\arraystretch{2.2}
\gamma_n=\left\{
\begin{array}{cl}
(1,3,5,7,\ldots,n-1,n,\ldots, 8,6,4,2), & n\,\, even,\\
(1,3,5,7,\ldots,n,n-1,\ldots, 8,6,4,2), & n\,\, odd.
\end{array}
\right.
\]Also, in \cite{BP1} the notion of breakpoint graph of a permutation was introduced, and important links between the maximum cycle decomposition of this graph and reversal distance were presented.
\begin{defi}
{\em{The \emph{breakpoint graph} $BG=BG(\pi)$ of a permutation $\pi$ on $[n]$ is the undirected graph whose vertex set is the vertex set of the cycle graph $G(\pi)$ and whose edges are the edges of $G(\pi)$ without their orientation.}}
\end{defi}
As we have seen in Section \ref{labarre} for cycle graphs, breakpoint graphs decompose into edge-disjoint alternating cycles. However, such a decomposition is not unique, differently from what occurs for cycle graphs. This property is the main reason why sorting by reversals was proven to be a $\textsf{NP}$-hard problem by Caprara in \cite{Ca}.

Furthermore, in \cite{BK} Berman and Karpinski proved that sorting a permutation by reversals is not approximable within $1.0008$. Before the result of Caprara was known, Kececioglu and Sankoff \cite{KS} gave a $2$-approximation algorithm, and Bafna and Pevzner \cite{BP1} presented an $\dfrac{7}{4}$-approximation algorithm to sort a permutation by reversals. The approximation ratio was improved to $\dfrac{3}{2}$ by Christie \cite{C} and then to $\dfrac{11}{8}$ by Berman, Hannenhalli, and Karpinski \cite{BHK}.

Table \ref{table:tabf} shows the distribution of the reversal distance $rd(\pi)$ with $\pi$ a permutation on $[n]$ with $1\leq n\leq 10$. Such a table was computed by Fertin et al. in \cite{FL}.
\begin{table}[b]
\caption[Reversal distances]{The number of permutations $\pi$ in $Sym_n$ with $rd(\pi)=k$, for $1\leq n\leq 10$.} 
\centering
\label{table:tabf}
\begin{tabular}{r*{11}{r}}
\toprule
 $n \backslash k$ & 0 & 1 & 2 & 3& 4 & 5 & 6 & 7 & 8 & 9\\
  \midrule
  1 & 1 & 0 & 0 & 0 & 0 & 0 & 0 & 0 & 0 & 0\\
  2 & 1 & 1 & 0 & 0 & 0 & 0 & 0 & 0 & 0 & 0 \\
  3 & 1 & 3 &  2 & 0 & 0 & 0 & 0 & 0 & 0 & 0\\
  4 & 1 & 6 & 15 & 2 & 0 & 0  & 0 & 0 & 0 & 0\\
  5 & 1 &10 &52 & 55 & 2 &0 & 0 & 0 & 0 & 0\\
  6 & 1 & 15&129 &389 & 184 & 2 & 0 & 0 & 0 & 0 \\
  7 & 1 &21&266&1563& 2539 & 648 & 2 & 0 & 0 & 0\\
  8 & 1 &28 & 487 & 4642 & 16445 & 16604 & 2111 & 2 & 0 & 0 \\
  9 & 1 & 36 & 820 & 11407 & 69863 & 169034 & 105365 & 6352 & 2 & 0\\
  10 & 1 & 45 & 1297 & 24600& 228613 & 1016341 & 1686534 & 654030 & 17337 & 2\\
 \bottomrule
\end{tabular}
\end{table}

\section{Cut-and-paste moves}

For any cut points, the \emph{cut-and-paste move} $\chi(i,j,k)$ acts on a permutation $\pi$ on $[n]$ either switching two adjacent subsequences of $\pi$ and possibly reversing one of them or simply reversing a subsequence of $\pi$. If $\chi(i,j,k)$ only
switches two adjacent subsequences of $\pi$, such a move is the block transposition $\sigma(i,j,k)$; see Section \ref{s31}. $\lambda(i,j,k)$ is any cut-and-paste move that switches two adjacent subsequences of $\pi$ and revers the second of them. This is formally
defined as follows:
\begin{equation}
\lambda(i,j,k)=\left\{\begin{array}{ll}
[1\cdots i\,\, k\cdots j+1\,\, i+1\cdots j\,\, k+1 \cdots n], & 1\leq i<j<k< n,\\
{[k\cdots j+1\,\, 1\cdots j\,\, k+1 \cdots n]}, & i=0\quad k< n,\\
{[1\cdots i\,\,n\cdots j+1\,\, i+1\cdots j]}, & 1\leq i\quad k=n,\\
{[n\cdots j+1\,\, 1\cdots j]}, & i=0\quad k=n.
\end{array}
\right.
\end{equation}$\varrho(i,j,k)$ is any cut-and-paste move that switches two adjacent subsequences of $\pi$ and then revers the first of them. This is formally defined
as follows:
\begin{equation}
\gamma(i,j,k)=\left\{\begin{array}{ll}
[1\cdots i\,\, j+1\cdots k\,\, j\cdots i+1\,\, k+1 \cdots n], & 1\leq i<j<k< n,\\
{[j+1\cdots k\,\, j\cdots 1\,\, k+1 \cdots n]}, & i=0\quad k< n,\\
{[1\cdots i\,\,j+1\cdots n\,\, j\cdots i+1]}, & 1\leq i\quad k=n,\\
{[j+1\cdots n\,\, j\cdots 1]}, & i=0\quad k=n.
\end{array}
\right.
\end{equation}If $\chi(i,j,k)$ only reverses one subsequence of $\pi$, such a move is the reversal $\rho(i,k)$; see Section \ref{s71}. The action of $\chi(i,j,k)$ on $\pi$ is defined as
\begin{equation}
\chi(i,j,k)=\left\{\begin{array}{ll}
[\pi_1\cdots \pi_i\,\,\pi_{j+1}\cdots \pi_k\,\,\pi_{i+1}\cdots \pi_{j}\,\,\pi_{k+1}\cdots \pi_n], & \chi=\sigma,\\
{[\pi_1\cdots \pi_i\,\,\pi_k\cdots \pi_{j+1}\,\,\pi_{i+1}\cdots \pi_{j}\,\,\pi_{k+1}\cdots \pi_n]}, & \chi=\lambda,\\
{[\pi_1\cdots \pi_i\,\,\pi_{j+1}\cdots \pi_k\,\,\pi_j\cdots \pi_{i+1}\,\,\pi_{k+1}\cdots \pi_n]}, & \chi=\gamma,\\
{[\pi_1\cdots \pi_i\,\, \pi_k\cdots \pi_{i+1}\,\, \pi_{k+1} \cdots \pi_n]}, & \chi=\rho.\\
\end{array}
\right.
\end{equation}Therefore, applying a cut-and-paste move $\chi(i,j,k)$ on the right of $\pi$ changes subsequences of $\pi$ in a way that may also be represented by
\begin{equation}
\begin{array}{ll}
[\pi_1\cdots\pi_i|\pi_{i+1}\cdots\pi_j|\pi_{j+1}\cdots\pi_k|\pi_{k+1}\cdots\pi_n], & \chi\neq\rho,\\
{[\pi_1\cdots\pi_i|\pi_{i+1}\cdots\pi_k|\pi_{k+1}\cdots\pi_n]}, & \chi=\rho.\\
\end{array}
\end{equation}
We observe that each of
$\lambda$ and $\gamma$ may also be expressed as a product of a block transposition and a reversal. In fact, it is straightforward to check
$$\lambda(i,j,k)=\rho(i,k-j+i)\circ\sigma(i,j,k);\qquad \gamma(i,j,k)=\rho(k-j+i,k)\circ\sigma(i,j,k).$$Since the reversals are involutory permutations and
the rearrangement set $S_n$ is inverse-closed; see Section \ref{equabl}, then the set $T$ of cut-and-paste moves is inverse-closed. Clearly, $T$ is a generator set of $\Sym_n$ since $S_n$ has this property; see Section \ref{equabl}. Now, since $T$ is a generator set, the following definition is meaningful.
\begin{defi}
{\em{The \emph{cut-and-paste distance of a permutation} $\pi$ on $[n]$ is $d_T(\pi)$ if $\pi$ is the product of $d_T(\pi)$ cut-and-paste moves, but it cannot be obtain as the product of less than $d_T(\pi)$ cut-and-paste moves.}}
\end{defi}A natural measure of the cut-and-paste distance of a permutation $\pi$ is the number of pairs bonds that occur in $\pi$, where a \emph{bond}
consists of two consecutive integers $x,x+1$ in the sequence; see \cite{EE}. Since for $n\le 3$ every permutation has a bound we assume in this section that
$n\geq 4$. We may observe that at most three bonds are created at each move, and that the identity permutation is the only permutation with maximum number $n+1$ of bonds. Therefore, any permutation of $[n]$ has distance at least $\lceil(n + 1)/3\rceil$. On the other hand,
the cut-and-paste distance is at most $n - \sqrt n + 1$. Indeed, in \cite{ES} Erd\H{o}s and Szekeresit prove that every string of $n$ distinct numbers has a monotone substring of length at least $\sqrt n$. Therefore, inserting the remaining elements one at a time into a longest monotone sequence,
and then reversing the full list at the end if necessary, give a sort of a permutation in at most $n - \sqrt n + 1$  cut-and paste moves. Let $d_T(n)$ indicate the \emph{cut-and-past diameter} in $\Sym_n$. Then
\begin{equation}\label{firstbounds}
\left\lceil(n + 1)/3\right\rceil \leq d_T(n)\leq n - \sqrt n + 1.
\end{equation}Actually, the upper bound in (\ref{firstbounds}) can be improved using the results on the block transposition diameter since $d_T(\pi)\leq d(\pi)$, for every permutation $\pi\in \Sym_n$. Since $\left\lfloor{2n-2}/3\right\rfloor=n+3/2$ for $n=13,15$, from Theorem \ref{EEmainbis} and Table \ref{table:tab3},
\[\arraycolsep=1.4pt\def\arraystretch{2.2}
d_T(n)\leq\left\{
\begin{array}{lll}
\left\lfloor\dfrac{n+2}{2}\right\rfloor, &n=14\quad3\leq n\leq 12,\\
\left\lfloor\dfrac{2n-2}{3}\right\rfloor, &n=13\quad15\leq n.
\end{array}
\right.
\]
Our contribution is to compute the cut-and-paste distribution for $1\leq n\leq 10$. Table \ref{table:tab4} shows such a distribution, performed by using the package ``grape'' of GAP \cite{gap}.
\begin{table}[ht!]
\caption[Cut-and-paste distances]{The number of permutations $\pi$ in $Sym_n$ with $d_T(\pi)=k$, for $1\leq n\leq 10$.} 
\centering
\label{table:tab4}
\begin{tabular}{r*{7}{r}}
\toprule
 $n \backslash k$ & 0 & 1 & 2 & 3& 4 & 5 \\
  \midrule
  1 & 1 & 0 & 0 & 0 & 0 & 0  \\
  2 & 1 & 1 & 0 & 0 & 0 & 0  \\
  3 & 1 & 5 & 0 & 0 & 0 & 0  \\
  4 & 1 & 16 & 8 & 0 & 0 & 0  \\
  5 & 1 &34&85&0&0&0 \\
  6 & 1 &65&511&143&0&0 \\
  7 & 1 &111&2096&2832&0&0\\
  8 & 1 &175&6592&29989&3563&0 \\
  9 & 1 &260&17208&206429&138982&0\\
  10 & 1 &369&39233&1015876&2487046&86275\\
 \bottomrule
\end{tabular}
\end{table}

\subsection{The Cranston lower bound}

The easy lower bound of $\lceil(n+1)/3\rceil$ was obtained by considering bonds. In order to achieve a better lower bound, it is useful to consider parity adjacencies. A \emph{parity adjacency} of $0\,\pi_1 \cdots \pi_n\, n+1$ is a pair of consecutive values in $\pi=[\pi_1\,\pi_2\cdots\pi_n]$ having opposite parity. For instance, the identity permutation has the maximum number $n+1$ of parity adjacencies. This suggests that we should count the number of moves $f(\pi)$ needed to obtain $n+1$ parity adjacencies. A lower bound of $f(n)$, the maximum $f(\pi)$ with $\pi\in \Sym_n$, is also a lower bound on the cut-and-paste diameter $d_T(n)$. Nevertheless the reverse permutation has also $n+1$ parity adjacencies. Therefore, we certainly hope that $f(n)$ would play any role in the investigation of upper bounds on $d_T(n)$.

In the following proposition Cranston et al. \cite[Theorem 1]{CS} prove that we can increase the number of parity adjacencies by at most $2$ at each step. Indeed, they observe that for every $\pi\in\Sym_n$, $f(\pi)$ cannot be increased by $3$ in any step.

\begin{prop}
For every $n\geq 4$,
$$d_T(n)\geq f(n)\geq \left\lfloor \frac{n}{2}\right\rfloor.$$
\end{prop}

In \cite{CS} Cranston et al. also suggest that every permutation with either one parity adjacency or two parity adjacencies when $n$ is odd has cut-and-paste distance equal to the lower bound. Nevertheless a proof of this conjecture is still not available in the literature. Here, we prove this conjecture in two special cases.
\begin{lem}\label{mat1110}
For every $\pi$ permutation on $[n]$,
\begin{itemize}
\item[\rm(I)] $\pi$ has one parity adjacency if and only if $n$ is even, and
$$0\,\pi\,n+1=0\,\pi_{i_1}\cdots\pi_{i_{r}}\,\pi_{j_1}\cdots\pi_{j_{r}}\,n+1,$$where $\pi_{i_l}$ is even and $\pi_{j_l}$ is odd, for every $1\leq l\leq r$, and $n=2r$ for some $r\geq 2$.
\item[\rm(II)] Let $n$ be even. $\pi$ has two pair adjacencies if and only if
$$0\,\pi\,n+1=0\,\pi_{i_1}\cdots\pi_{i_t}\,\pi_{i_{t+1}}\cdots\pi_{i_q}\,\pi_{i_{q+1}}\cdots\pi_{i_n}\,n+1,$$where $\pi_{i_l}$ is odd for every $t+1\leq l\leq q$, and $\pi_{i_l}$ is even for every $1\leq l\leq t$ and $q+1\leq l\leq n$.
\item[\rm(III)] Let $n$ be odd. $\pi$ has two pair adjacencies if and only if
$$0\,\pi\,n+1=0\,\pi_{i_1}\cdots\pi_{i_{r}}\,\pi_{j_1}\cdots\pi_{j_{r+1}}\,n+1,$$where $\pi_{i_l}$ is odd and $\pi_{j_l}$ is even, for every $1\leq l\leq r$, and $n=2r+1$ for some $r\geq 2$.
\end{itemize}
\end{lem}
\begin{prop}
Let $\pi$ be a permutation as in either case \em{(I)} or \em{(III)} of Lemma \ref{mat1110} with a monotone subsequence of either entirely even or odd numbers. Then
$$d_T(\pi)\leq\left\lfloor \dfrac{n}{2}\right\rfloor.$$
\end{prop}
\begin{proof}
Let $S=\pi_{i_1}\cdots\pi_{i_{r}}$ and let $r=\lfloor{n}/{2}\rfloor$. Suppose $S$ is monotone increasing. At each step cut an element $x$ from the remaining subsequence and paste it between $x-1$ and $x+1$ as follows
$$0\,\pi\, n+1=0\,\pi_{i_1}\cdots x-1|x+1\cdots\pi_{i_{r}}\,\pi_{j_{1}}\cdots |x|\cdots\pi_{j_{r+1}}\,n+1.$$Hence, $\pi$ is sorted in at most $r$ steps if either $n$ is even or $n$ is odd, and $S$ consists of odd numbers. When $n$ is odd with $S$ consisting of even numbers two cases occur: either $S$ is at the beginning of $\pi$ or $s$ is at the end of $\pi$. In the former case, cut every $x$ expect $n$; while in the latter move every $x$ expect $1$. Hence the statement holds.

Now, suppose $S$ is monotone decreasing. Cut an element $x$ from the remaining subsequence, and paste it between $x+1$ and $x-1$ as follows
$$0\,\pi\, n+1=0\,\pi_{i_1}\cdots x+1|x-1\cdots\pi_{i_{r}}\,\pi_{j_{1}}\cdots |x|\cdots\pi_{j_{r+1}}\,n+1.$$Assume $n$ to be odd, and consider the case when $S$ consists of odd numbers. Cutting every $x$, after $r-1$ steps, we obtain the reverse permutation $w$. Carrying out $w$ yields the claim holds. Here, assume $S$ to consist of even numbers. If $S$ is at the end of $\pi$, then move every $x$ expect $n$; while cut every $x$ expect $1$ when $S$ is at the beginning of $\pi$. In both cases, we obtain $w$ after taking $r-1$ steps. Hence, the statement holds as in the previous case. Now, suppose $n$ is even. If $S$ consists of odd numbers, then move every $x$ expect $n$. Then, we obtain $w$ if $S$ is at the end of $\pi$, and the statement holds as in the previous case. When $S$ is at the beginning of $\pi$, after taking $r-1$, we have$$n-1\,n-2\cdots 1\,n.$$Hence, carrying out the move $\rho(0,n-1)$ the statement holds. Here, assume $S$ to consist of even numbers and be at the beginning of $\pi$. Move every $x$ expect $1$, then that gives the reverse permutation $w$. Hence, the statement holds as we have seen before. If $S$ is at the end of $\pi$, move every $x$ expect $n-1$. Therefore, we obtain
$$n-1\,n\,n-2\cdots 1.$$Carrying out the move $\lambda(0,2,n)$, we sort $\pi$ in at most $r$ steps. This completes the proof.\end{proof}

\chapter{Block transposition graph for small n}\label{c8}
All computation are performed by using the package ``grape'' of GAP \cite{gap}.

\section{Case n=4}
\label{n=4}

The $10$ block transpositions of $\Sym_4$ are listed below.
$$\begin{array}{llll}
{\bf{1}}=\sigma( 0, 1, 2 ), & {\bf{2}}=\sigma( 0, 1, 3 ), & {\bf{3}}=\sigma( 0, 1, 4 ), & {\bf{4}}=\sigma( 0, 2, 3 ), \\
{\bf{5}}=\sigma( 0, 2, 4 ), & {\bf{6}}=\sigma( 0, 3, 4 ), & {\bf{7}}=\sigma( 1, 2, 3 ), & {\bf{8}}=\sigma( 1, 2, 4 ),\\
{\bf{9}}=\sigma( 1, 3, 4 ), & {\bf{10}}=\sigma( 2, 3, 4 ).
\end{array}$$The edges of the block transposition graph $\Gamma$ of ${\rm{Cay}}(\Sym_4,S_4)$ are
$$\begin{array}{l}
\{\bf{ 1},\bf{2}\}, \{\bf{1}, \bf{4}\}, \{\bf{ 1},\bf{8} \}, \{\bf{ 1}, \bf{10} \}, \{\bf{ 2}, \bf{3} \}, \{\bf{ 2}, \bf{5} \}, \{\bf{ 2}, \bf{8} \}, \\
\{\bf{ 3}, \bf{4} \}, \{\bf{ 3},\bf{5} \}, \{\bf{3}, \bf{9} \}, \{\bf{ 4}, \bf{6} \},\{\bf{ 4}, \bf{10} \}, \{\bf{ 5},\bf{6} \},\{\bf{ 5}, \bf{7} \}, \\
\{\bf{ 6}, \bf{7} \}, \{\bf{ 6}, \bf{8} \}, \{\bf{ 7},\bf{9} \}, \{\bf{ 7}, \bf{10} \}, \{\bf{ 8}, \bf{9} \}, \{\bf{ 9}, \bf{10}\}.
\end{array}$$ $\Gamma$ is a $4$-regular.
The full automorphism group of $\Gamma$ is the dihedral group $\textsf{D}_5$ of order $10$.
The toric classes are
$$\begin{array}{l}
\{ \bf{1,3,6,10,7}\},\\
\{ \bf{2,5,9,4,8}\}.\\
\end{array}
$$
The edges of the maximal cliques of $\Gamma$ of size $2$ are
$$\{\bf{4},5\}, \{2,4\}, \{5,8\}, \{8,9\}, \{2,9\}.$$
$\Gamma(V)$ is a Hamiltonian and $2$-regular graph. The full automorphism group of $\Gamma(V)$ has order $10$. The full automorphism group ${\rm{Aut}({\rm{Cay}}(\Sym_4,S_4))}$ has order $240$.
\section{Case n=5}
\label{n=5}

The $20$ block transpositions of $\Sym_5$ are listed below.
$$\begin{array}{llll}
{\bf{1}}=\sigma(0,1,2),& {\bf{2}}=\sigma(0,1,3), & {\bf{3}}=\sigma( 0, 1, 4 ), & {\bf{4}}=\sigma( 0, 1, 5 ),\\
{\bf{5}}=\sigma( 0, 2, 3 ),& {\bf{6}}=\sigma( 0, 2, 4 ), & {\bf{7}}=\sigma( 0, 2, 5 ), & {\bf{8}}=\sigma( 0, 3, 4 ), \\
{\bf{9}}=\sigma( 0, 3, 5 ),& {\bf{10}}=\sigma( 0, 4, 5 ), & {\bf{11}}=\sigma( 1, 2, 3 ), & {\bf{12}}=\sigma( 1, 2, 4 ),\\
 {\bf{13}}=\sigma( 1, 2, 5 ), & {\bf{14}}=\sigma( 1, 3, 4 ), & {\bf{15}}=\sigma( 1, 3, 5 ), & {\bf{16}}=\sigma( 1, 4, 5 ), \\
 {\bf{17}}=\sigma( 2, 3, 4 ), &  {\bf{18}}=\sigma( 2, 3, 5 ), & {\bf{19}}=\sigma( 2, 4, 5 ), & {\bf{20}}=\sigma( 3, 4, 5 ).
\end{array}
$$
The edges of the block transposition graph $\Gamma$ of ${\rm{Cay}}(\Sym_5,S_5)$ are
$$\begin{array}{l}
\{ \bf{1}, 2 \},\{ 1, 3 \}, \{ 1, 4 \}, \{ 1, 11 \}, \{ 1, 12 \}, \{ 1, 13 \}, \{ 2, 3 \}, \{ 2, 4 \}, \{ 2, 5 \},\\
\{ \bf{2}, 14 \}, \{ \bf{2}, 15 \},  \{ 3, 4 \}, \{ 3, 6 \}, \{ 3, 8 \}, \{ 3, 16 \}, \{ 4, 7 \}, \{ 4, 9 \}, \{ 4, 10 \},\\
\{ \bf{5}, 6 \},\{ 5, 7 \}, \{ 5, 11 \}, \{ 5, 17 \},  \{ 5, 18 \}, \{ 6, 7 \}, \{ 6, 8 \}, \{ 6, 12 \}, \{ 6, 19 \},\\
\{ \bf{7}, 9 \}, \{ 7, 10 \}, \{ 7, 13 \}, \{ 8, 9 \}, \{ 8, 14 \},  \{ 8, 17 \}, \{ 8, 20 \}, \{ 9, 10 \}, \{ 9, 15 \},\\
\{ \bf{9}, 18 \}, \{ 10, 16 \}, \{ 10, 19 \}, \{ 10, 20 \}, \{ 11, 12 \}, \{ 11, 13 \}, \{ 11, 17 \},\\
\{ \bf{11}, 18 \},\{ 12, 13 \}, \{ 12, 14 \}, \{ 12, 19 \}, \{ 13, 15 \}, \{ 13, 16 \}, \{ 14, 15 \},\\
\{ \bf{14}, 17 \}, \{ 14, 20 \},\{ 15, 16 \},\{ 15, 18 \}, \{ 16, 19 \}, \{ 16, 20 \}, \{ 17, 18 \},\\
\{ \bf{17}, 20 \},\{ 18, 19 \}, \{ 19, 20 \}.
\end{array}$$ $\Gamma$ is a $6$-regular graph.
The full automorphism group of $\Gamma$ is the dihedral group $\textsf{D}_6$ of order $12$.
The toric classes are
$$
\begin{array}{l}
\{ \bf{1},4,10,20,17,11\},\\
\{ \bf{2},7,16,8,18,12\},\\
\{ \bf{3},9,19,1,4,5,13\},\\
\{ \bf{6},15\}.\\
\end{array}
$$
The edges of the maximal cliques of $\Gamma$ of size $2$ are
$$\{\bf{2},5\}, \{13,7\}, \{8,9\}, \{12,14\}, \{3,16\}, \{18,19\}.$$
$\Gamma(V)$ is a Hamiltonian and $3$-regular graph. The full automorphism group of $\Gamma(V)$ has order $48$. The full automorphism group ${\rm{Aut}({\rm{Cay}}(\Sym_5,S_5))}$ has order $1440$.

\section{Case n=6}
\label{n=6}

The $35$ block transpositions of $\Sym_6$ are listed below.
$$
\begin{array}{llll}
{\bf{1}}=\sigma( 0, 1, 2 ),& {\bf{2}}=\sigma( 0, 1, 3 ),& {\bf{3}}=\sigma( 0, 1, 4 ),& {\bf{4}}=\sigma( 0, 1, 5 ),\\
{\bf{5}}=\sigma( 0, 1, 6 ),&{\bf{6}}=\sigma( 0, 2, 3 ),& {\bf{7}}=\sigma( 0, 2, 4 ),& {\bf{8}}=\sigma( 0, 2, 5 ),\\
{\bf{9}}=\sigma( 0, 2, 6 ),& {\bf{10}}=\sigma( 0, 3, 4 ),& {\bf{11}}=\sigma( 0, 3, 5 ),& {\bf{12}}=\sigma( 0, 3, 6 ),\\
{\bf{13}}=\sigma( 0, 4, 5 ), & {\bf{14}}=\sigma( 0, 4, 6 ),& {\bf{15}}=\sigma( 0, 5, 6 ), & {\bf{16}}=\sigma( 1, 2, 3 ),\\
{\bf{17}}=\sigma( 1, 2, 4 ),& {\bf{18}}=\sigma( 1, 2, 5 ),& {\bf{19}}=\sigma( 1, 2, 6 ),& {\bf{20}}=\sigma( 1, 3, 4 ),\\
{\bf{21}}=\sigma( 1, 3, 5 ),& {\bf{22}}=\sigma( 1, 3, 6 ),& {\bf{23}}=\sigma( 1, 4, 5 ),& {\bf{24}}=\sigma( 1, 4, 6 ),\\
{\bf{25}}=\sigma( 1, 5, 6 ), & {\bf{26}}=\sigma( 2, 3, 4 ),& {\bf{27}}=\sigma( 2, 3, 5 ),& {\bf{28}}=\sigma( 2, 3, 6 ),\\
{\bf{29}}=\sigma( 2, 4, 5 ),& {\bf{30}}=\sigma( 2, 4, 6 ), &{\bf{31}}=\sigma( 2, 5, 6 ),& {\bf{32}}=\sigma( 3, 4, 5 ),\\
{\bf{33}}=\sigma( 3, 4, 6 ),& {\bf{34}}=\sigma( 3, 5, 6 ),& {\bf{35}}=\sigma( 4, 5, 6 ).
\end{array}
$$
The edges of the block transposition graph $\Gamma$ of ${\rm{Cay}}(\Sym_6,S_6)$ are
$$\begin{array}{l}
\bf{\{ 1, 2 \}, \{ 1, 4 \}, \{ 1, 8 \}, \{ 1, 10 \}, \{ 1, 18 \}, \{ 1, 20 \}, \{ 1, 33 \}, \{ 1, 35 \}, \{ 2, 3 \},}\\
\bf{\{ 2, 5 \}, \{ 2, 8 \}, \{ 2, 15 \}, \{ 2, 18 \}, \{ 2, 30 \}, \{ 2, 33 \}, \{ 3, 4 \}, \{ 3, 5 \}, \{ 3, 9 \},}\\
\bf{\{ 3, 15 \}, \{ 3, 19 \},\{ 3, 30 \},\{ 3, 34 \}, \{ 4, 6 \}, \{ 4, 10 \}, \{ 4, 16 \}, \{ 4, 20 \}, \{ 4, 31 \},}\\
\bf{\{ 4, 35 \}, \{ 5, 6 \}, \{ 5, 7 \},\{ 5, 11 \}, \{ 5, 15 \}, \{ 5, 26 \}, \{ 5, 30 \}, \{ 6, 7 \}, \{ 6, 8 \},}\\
\bf{\{ 6, 11 \}, \{ 6, 16 \}, \{ 6, 26 \}, \{ 6, 31 \},\{ 7, 9 \}, \{ 7, 10 \}, \{ 7, 11 \}, \{ 7, 17 \}, \{ 7, 26 \},}\\
\bf{\{ 7, 32 \}, \{ 8, 9 \}, \{ 8, 12 \}, \{ 8, 18 \}, \{ 8, 27 \},\{ 8, 33 \}, \{ 9, 10 \}, \{ 9, 12 \},\{ 9, 19 \}}\\
\bf{\{ 9, 19 \}, \{ 9, 27 \}, \{ 9, 34 \}, \{ 10, 13 \}, \{ 10, 20 \}, \{ 10, 28 \},\{ 10, 35 \},\{ 11, 12 \},}\\
\bf{\{ 11, 13 \}, \{ 11, 14 \}, \{ 11, 21 \}, \{ 11, 26 \}, \{ 12, 13 \}, \{ 12, 14 \},\{ 12, 15 \}, } \\
\bf{ \{ 12, 21 \},\{ 12, 27 \}, \{ 13, 14 \}, \{ 13, 16 \}, \{ 13, 18 \}, \{ 13, 21 \}, \{ 13, 28 \},}\\
\bf{\{ 14, 17 \}, \{ 14, 19 \},\{ 14, 20 \}, \{ 14, 21 \}, \{ 14, 29 \}, \{ 15, 16 \}, \{ 15, 17 \},}\\
\bf{ \{ 15, 22 \},\{ 15, 30 \}, \{ 16, 17 \}, \{ 16, 18 \}, \{ 16, 22 \}, \{ 16, 31 \},\{ 17, 19 \},}\\
\bf{ \{ 17, 20 \}, \{ 17, 22 \}, \{ 17, 32 \}, \{ 18, 19 \}, \{ 18, 23 \}, \{ 18, 33 \},\{ 19, 20 \}, }\\
\bf{\{ 19, 23 \}, \{ 19, 34 \}, \{ 20, 24 \}, \{ 20, 35 \}, \{ 21, 22 \}, \{ 21, 23 \}, \{ 21, 24 \},}\\
\bf{ \{ 21, 25 \}, \{ 22, 23 \}, \{ 22, 24 \},\{ 22, 25 \}, \{ 22, 26 \}, \{ 23, 24 \}, \{ 23, 27 \},}\\
\bf{  \{ 23, 30 \}, \{ 24, 25 \},\{ 24, 28 \}, \{ 24, 31 \}, \{ 24, 33 \},\{ 23, 25 \}, \{ 25, 29 \},}\\
\bf{ \{ 25, 32 \}, \{ 25, 34 \}, \{ 25, 35 \}, \{ 26, 27 \}, \{ 26, 28 \}, \{ 26, 29 \},\{ 27, 28 \}, }\\
\bf{  \{ 27, 29 \},\{ 27, 30 \}, \{ 28, 29 \}, \{ 28, 31 \}, \{ 28, 33 \}, \{ 29, 32 \},\{ 29, 34 \}, }\\
\bf{ \{ 29, 35 \}, \{ 30, 31 \},  \{ 30, 32 \}, \{ 31, 32 \}, \{ 31, 33 \}, \{ 32, 34 \},\{ 32, 35 \},  }\\
\bf{\{ 33, 34 \},\{ 34, 35 \}.}
\end{array}$$ $\Gamma$ is a $8$-regular graph.
The full automorphism group of $\Gamma$ is the dihedral group $\textsf{D}_7$ of order $14$.
The toric classes are
$$
\begin{array}{l}
\bf{\{1,2,5,11,21,25,35\}},\\
\bf{\{3,6,12,22,29,20,33\}},\\
\bf{\{4,8,15,26,14,24,34\}},\\
\bf{\{7,13,23,32,10,18,30\}},\\
\bf{\{9,16,27,17,28,19,31\}}.
\end{array}
$$
The edges of the maximal cliques of $\Gamma$ of size $2$ are
$$\bf{\{3,4\},\{6,8\},\{12,15\},\{14,29\},\{22,26\},\{24,20\},\{33,34\}.}$$
$\Gamma(V)$ is a Hamiltonian and $3$-regular graph. The full automorphism group of $\Gamma(V)$ has order $336$. The full automorphism group ${\rm{Aut}({\rm{Cay}}(\Sym_6,S_6))}$ has order $10080$.

\backmatter


\bibliographystyle{amsplain}
\cleardoublepage
\bibliography{references} 

\providecommand{\bysame}{\leavevmode\hbox to3em{\hrulefill}\thinspace}
\providecommand{\MR}{\relax\ifhmode\unskip\space\fi MR }
\providecommand{\MRhref}[2]{%
  \href{http://www.ams.org/mathscinet-getitem?mr=#1}{#2}
}
\providecommand{\href}[2]{#2}
\begin{thebibliography}{10}

\bibitem{AB}
B.~Alspach, \emph{Cayley graphs}, Handbook of graph theory (J.~L. Gross and
  J.~Yellen, eds.), CRC Press, Boca Raton, FL, 2004, pp.~505--515.

\bibitem{AZ}
B.~Alspach and Cun~Quan Zhang, \emph{Hamilton cycles in cubic {C}ayley graphs
  on dihedral groups}, Ars Combin. \textbf{28} (1989), 101--108.

\bibitem{AA}
A.~Auyeung, A.~Abraham, and Oklahoma State University. Computer~Science
  Department, \emph{Estimating genome reversal distance by genetic algorithm},
  OSU-CS-TR, Oklahoma State University, Computer Science Department, 2003.

\bibitem{BL}
L.~Babai, \emph{Automorphism groups, isomorphism, reconstruction}, Handbook of
  Combinatorics (R.~L. Graham, M.~Gr{\u{o}}tschel, and L.~Lov{\'a}sz, eds.),
  vol.~II, MIT Press, Cambridge, MA, 1995, pp.~1447--1540.

\bibitem{BP1}
V.~Bafna and P.~A. Pevzner, \emph{Genome rearrangements and sorting by
  reversals}, SIAM J. Comput. \textbf{25} (1996), no.~2, 272--289.

\bibitem{BP}
\bysame, \emph{Sorting by transpositions}, SIAM J. Discrete Math. \textbf{11}
  (1998), no.~2, 224--240.

\bibitem{BE}
S.~Beno\^{\i}t and H.~S. Gagn{\'e}, \emph{A new and faster method of sorting by
  transpositions}, Proceedings of the 18th Annual Conference on Combinatorial
  Pattern Matching, Springer-Verlag, 2007, pp.~131--141.

\bibitem{BHK}
P.~Berman, S.~Hannenhalli, and M.~Karpinski, \emph{1.375-approximation
  algorithm for sorting by reversals}, Algorithms---{ESA} 2002, Lecture Notes
  in Comput. Sci., vol. 2461, Springer, Berlin, 2002, pp.~200--210.

\bibitem{BK}
P.~Berman and M.~Karpinski, \emph{On some tighter inapproximability results
  (extended abstract)}, Automata, languages and programming ({P}rague, 1999),
  Lecture Notes in Comput. Sci., vol. 1644, Springer, Berlin, 1999,
  pp.~200--209.

\bibitem{BN}
\emph{{Biology Notes for IGCSE 2014}}, 2014, Available on line.

\bibitem{B}
M.~B\'{o}na, \emph{Combinatorics of permutations}, second ed., Discrete
  Mathematics and its Applications (Boca Raton), CRC Press, Boca Raton, FL,
  2012.

\bibitem{BM}
J.A. Bondy and U.S.R. Murty, \emph{Graph theory with applications}, North
  Holland, 1976.

\bibitem{BF}
L.~Bulteau, G.~Fertin, and I.~Rusu, \emph{Sorting by transpositions is
  difficult}, SIAM J. Discrete Math. \textbf{26} (2012), no.~3, 1148--1180.

\bibitem{Ca}
A.~Caprara, \emph{Sorting by reversals is difficult}, Proceedings of the First
  Annual International Conference on Computational Molecular Biology, ACM,
  1997, pp.~75--83.

\bibitem{C}
D.~A. Christie, \emph{Genome rearrangement problems}, Ph.D. thesis, University
  of Glasgow, 1998.

\bibitem{CS}
D.~W. Cranston, I.~H. Sudborough, and D.~B. West, \emph{Short proofs for
  cut-and-paste sorting of permutations}, Discrete Math. \textbf{307} (2007),
  no.~22, 2866--2870.

\bibitem{CK}
L.~F.~I. Cunha, L.~A.~B. Kowada, R.~de~A. Hausen, and C.~M.~H. de~Figueiredo,
  \emph{Advancing the transposition distance and diameter through lonely
  permutations}, SIAM J. Discrete Math. \textbf{27} (2013), no.~4, 1682--1709.

\bibitem{DL}
J.~Doignon and A.~Labarre, \emph{On hultman numbers}, J. Integer Seq.
  \textbf{10} (2007), no.~6.

\bibitem{EH}
I.~Elias and T.~Hartman, \emph{A 1.375-approximation algorithm for sorting by
  transpositions}, IEEE/ACM Trans. Comput. Biol. \textbf{3} (2007), 369--379.

\bibitem{ES}
P.~Erd{\"o}s and G.~Szekeres, \emph{A combinatorial problem in geometry},
  Compositio Math. \textbf{2} (1935), 463--470.

\bibitem{EE}
H.~Eriksson, K.~Eriksson, J.~Karlander, L.~Svensson, and J.~W{\u{a}}stlund,
  \emph{Sorting a bridge hand}, Discrete Math. \textbf{241} (2001), no.~1-3,
  289--300.

\bibitem{FZ}
J.~Feng and D.~Zhu, \emph{Faster algorithms for sorting by transpositions and
  sorting by block interchanges}, ACM Trans. Algorithms \textbf{3} (2007),
  no.~3.

\bibitem{FL}
G.~Fertin, A.~Labarre, I.~Rusu, \'{E}. Tannier, and S.~Viallette,
  \emph{Combinatorics of genome rearrangements}, MIT Press, 2009.

\bibitem{GD}
G.~R. Galv\~{a}o and Z.~Dias, \emph{On the distribution of rearrangement
  distances}, Proceedings of the Brazilian Symposium on Bioinformatics, BSB'11
  (Bras\'{\i}lia, Brazil), 2011, pp.~41--48.

\bibitem{GBH}
J.~Gon\c{c}alves, L.~R. Bueno, and R.~de~A. Hausen, \emph{Assembling a new
  improved transposition distance database}, Proceedings of the 2013 XLV
  Simp{\'{o}}sio Brasileiro de Pesquisa Operacional, SBPO'13 (Natal, Brasil),
  2013, pp.~2355--2365.

\bibitem{gap}
The~GAP Group, \emph{{GAP -- Groups, Algorithms, and Programming, Version
  4.7.6}}, 2014.

\bibitem{Gu}
Q.~Gu, S.~Peng, and Q.~Chen, \emph{Sorting permutations and its applications in
  genome analysis}, Lecture Notes on Mathematics in the Life Science, 1999,
  pp.~191--201.

\bibitem{HP}
S.~Hannenhalli and P.~A. Pevzner, \emph{Transforming cabbage into turnip:
  polynomial algorithm for sorting signed permutations by reversals}, J. ACM
  \textbf{46} (1999), no.~1, 1--27.

\bibitem{HS}
T.~Hartman and R.~Shamir, \emph{A simpler and faster 1.5-approximation
  algorithm for sorting by transpositions}, Inform. and Comput. \textbf{204}
  (2006), no.~2, 275--290.

\bibitem{Ha}
R.~de~A. Hausen, \emph{Rearranjos de genomas: teoria e aplica{\c{c}\~{o}}es},
  Ph.D. thesis, COPPE Sistemas, Universidade Federal do Rio de Janeiro, Rio de
  Janeiro, 2007.

\bibitem{J}
M.~R. Jerrum, \emph{The complexity of finding minimum-length generator
  sequences}, Theoret. Comput. Sci. \textbf{36} (1985), no.~2-3, 265--289.

\bibitem{KS}
J.~Kececioglu and D.~Sankoff, \emph{Exact and approximation algorithms for
  sorting by reversals with application to genome rearrangement}, Algorithmica
  \textbf{13} (1995), no.~1-2, 180--210.

\bibitem{la}
A.~Labarre, \emph{Lower bounding edit distances between permutations}, SIAM J.
  Discrete Math. \textbf{27} (2013), no.~3, 1410--1428.

\bibitem{LL}
L.~Lov{\'a}sz, \emph{The factorization of graphs}, Combinatorial {S}tructures
  and their {A}pplications ({P}roc. {C}algary {I}nternat. {C}onf., {C}algary,
  {A}lta., 1969), Gordon and Breach, New York, 1970, pp.~243--246.

\bibitem{Mo}
V.~Moulton and M.~Steel, \emph{The `butterfly effect' in {C}ayley graphs with
  applications to genomics}, J. Math. Biol. \textbf{65} (2012), no.~6-7,
  1267--1284.

\bibitem{O}
S.J. O'Brien, \emph{Genetic maps: Locus maps of complex genomes}, GENETIC MAPS
  BOOK 2, Cold Spring Harbor Laboratory Press, 1993.

\bibitem{P}
J.~D. Palmer, B.~Osorio, and W.~F. Thompson, \emph{Evolutionary significance of
  inversions in legume chloroplast dnas}, Current Genetics \textbf{14} (1988),
  no.~1, 65--74.

\bibitem{DS}
A.~H. Sturtevant and T.~Dobzhansky, \emph{Inversions in the third chromosome of
  wild races of drosophila pseudoobscura, and their use in the study of the
  history of the species}, Proceedings of the National Academy of Sciences of
  the United States of America \textbf{22} (1936), no.~7, 448--450.

\bibitem{W}
G.A. Watterson, W.J. Ewens, T.E. Hall, and A.~Morgan, \emph{The chromosome
  inversion problem}, Journal of Theoretical Biology \textbf{99} (1982), no.~1,
  1--7.

\bibitem{WS}
D.~P. Williamson and D.~B. Shmoys, \emph{The design of approximation
  algorithms}, Cold Spring Harbor Laboratory Press, Cambridge, 2011.

\end{thebibliography}









\end{document}